\pdfoutput=1
\RequirePackage{ifpdf}
\ifpdf 
\documentclass[pdftex]{sigma}
\else
\documentclass{sigma}
\fi

\usepackage[all]{xy}

\numberwithin{equation}{section}
\newtheorem{Theorem}{Theorem}[section]
\newtheorem{Lemma}[Theorem]{Lemma}
\newtheorem{Proposition}[Theorem]{Proposition}
{ \theoremstyle{definition}
\newtheorem{Definition}[Theorem]{Definition}
\newtheorem{Example}[Theorem]{Example}
\newtheorem{Remark}[Theorem]{Remark}
\newtheorem{Claim}[Theorem]{Claim}
}

\begin{document}

\allowdisplaybreaks

\newcommand{\arXivNumber}{1505.00469}

\renewcommand{\PaperNumber}{086}

\FirstPageHeading

\ShortArticleName{BiHom-Associative Algebras, BiHom-Lie Algebras and BiHom-Bialgebras}

\ArticleName{BiHom-Associative Algebras, BiHom-Lie Algebras\\ and BiHom-Bialgebras}

\Author{Giacomo {GRAZIANI}~$^{\dag^1}\!$, Abdenacer {MAKHLOUF}~$^{\dag^2}\!$,  Claudia {MENINI}~$^{\dag^3}\!$ and Florin {PANAITE}~$^{\!\dag^4}\!$}

\AuthorNameForHeading{G.~Graziani, A.~Makhlouf, C.~Menini and F.~Panaite}

\Address{$^{\dag^1}$~Universit\'{e} Joseph Fourier Grenoble I Institut Fourier,\\
\hphantom{$^{\dag^1}$}~100, Rue des Maths BP74 38402 Saint-Martin-d'H\`{e}res, France}
\EmailDD{\href{mailto:Giacomo.Graziani@ujf-grenoble.fr}{Giacomo.Graziani@ujf-grenoble.fr}}

\Address{$^{\dag^2}$~Universit\'{e} de Haute Alsace,
Laboratoire de Math\'{e}matiques, Informatique et Applications,\\
\hphantom{$^{\dag^2}$}~4, Rue des fr\`{e}res Lumi\`{e}re, F-68093 Mulhouse, France}
\EmailDD{\href{Abdenacer.Makhlouf@uha.fr}{Abdenacer.Makhlouf@uha.fr}}

\Address{$^{\dag^3}$~University of Ferrara, Department of Mathematics,\\
\hphantom{$^{\dag^3}$}~Via Machiavelli 30, Ferrara, I-44121, Italy}
\EmailDD{\href{mailto:men@unife.it}{men@unife.it}}

\Address{$^{\dag^4}$~Institute of Mathematics of the Romanian Academy,\\
\hphantom{$^{\dag^4}$}~PO-Box 1-764, RO-014700 Bucharest, Romania}
\EmailDD{\href{florin.panaite@imar.ro}{florin.panaite@imar.ro}}

\ArticleDates{Received May 12, 2015, in f\/inal form October 13, 2015; Published online October 25, 2015}

\Abstract{A~BiHom-associative algebra is a (nonassociative) algebra~$A$ endowed with
two commuting multiplicative linear maps $\alpha , \beta\colon A\rightarrow A$
such that $\alpha (a)(bc)=(ab)\beta (c)$, for all $a, b, c\in A$. This
concept arose in the study of algebras in so-called group Hom-categories. In
this paper, we introduce as well BiHom-Lie algebras (also by using the
categorical approach) and BiHom-bialgebras. We discuss these new structures
by presenting some basic properties and constructions (representations,
twisted tensor products, smash products etc).}


\Classification{17A99; 18D10; 16T99}

\section{Introduction}

 The origin of Hom-structures may be found in the physics
literature around 1990, concerning $q$-deformations of algebras of vector
f\/ields, especially Witt and Virasoro algebras, see for instance~\cite{AizawaSato,ChaiIsKuLuk,CurtrZachos1,Liu1}.
Hartwig, Larsson and Silvestrov studied this kind of algebras in~\cite{HLS,LS1} and called them Hom-Lie algebras because they involve a
homomorphism in the def\/ining identity. More precisely, a Hom-Lie algebra is
a linear space $L$ endowed with two linear maps $[-]\colon L\otimes L\rightarrow L$
and $\alpha\colon L\rightarrow L$ such that $[-]$ is skew-symmetric and  $\alpha$ is an algebra endomorphism
with respect to the bracket   satisfying the
so-called Hom-Jacobi identity
\begin{gather*}
[\alpha (x),[y,z]]+[\alpha (y),[z,x]]+[\alpha (z),[x,y]]=0,\qquad \forall
\, x,y,z\in L.
\end{gather*}
Since any associative algebra becomes  a Lie algebra by taking the commutator
$[a, b]=ab-ba$, it was natural to look for a Hom-analogue of this property.
This was accomplished in~\cite{ms1}, where the concept of Hom-associative
algebra was introduced, as being a linear space $A$ endowed with a
multiplication $\mu \colon A\otimes A\rightarrow A$, $\mu (a\otimes b)=ab$, and a
linear map $\alpha \colon A\rightarrow A$ satisfying the so-called
Hom-associativity condition
\begin{gather*}
 \alpha (a)(bc)=(ab)\alpha (c), \qquad \forall\, a, b, c\in A.
\end{gather*}
If $A$ is Hom-associative then $(A, [a, b]=ab-ba, \alpha)$ becomes a Hom-Lie
algebra, denoted  by~$L(A)$. Notice that Hom-Lie algebras, in this paper, were considered
without the assumption of multiplicativity of $\alpha$.

In subsequent literature (see for instance \cite{yau2}) were studied
subclasses of these classes of algebras where
the linear maps $\alpha $ involved in the def\/inition of a~Hom-Lie algebra or Hom-associative algebra are required  to be multiplicative, that is
$\alpha ([x,y])=[\alpha(x),\alpha(y)]$ for all \mbox{$x, y\in L$}, respectively
$\alpha (ab)=\alpha (a)\alpha (b)$ for all $a, b\in A$, and these subclasses
were called multiplicative Hom-Lie algebras, respectively
multiplicative Hom-associative algebras. Since we will always assume multiplicativity of the maps
$\alpha $ and to simplify terminology, we will call Hom-Lie or Hom-associative algebras what was called
above multiplicative Hom-Lie or Hom-associative algebras.

{\sloppy The Hom-analogues of coalgebras, bialgebras and Hopf algebras have been
introduced in~\mbox{\cite{ms3,ms4}}. The original def\/inition of a
Hom-bialgebra involved two linear maps, one twisting the associativity
condition and the other one the coassociativity condition. Later, two
directions of study on Hom-bialgebras were developed, one in which the two
maps coincide (these are still called Hom-bialgebras) and another one,
started in~\cite{stef}, where the two maps are assumed to be inverse to each
other (these are called monoidal Hom-bialgebras).

}

In the last years, many concepts and properties from classical algebraic
theories have been extended to the framework of Hom-structures, see for
instance \cite{AC1,AC2,said,stef,chenwangzhang,hassanzadeh,LB,
mp1,ms3,ms4,sheng,yau1,yau2}.

The main tool for constructing examples of Hom-type algebras is the
so-called ``twisting principle''
introduced by D.~Yau for Hom-associative algebras and extended afterwards to
other types of Hom-algebras. For instance, if~$A$ is an associative algebra
and $\alpha \colon A\rightarrow A$ is an algebra map, then $A$ with the new
multiplication def\/ined by $a\ast b=\alpha (a)\alpha (b)$ is a~Hom-associative algebra, called the \emph{Yau twist }of~$A$.

A categorical interpretation of Hom-associative algebras has been given by
Caenepeel and Goyvaerts in~\cite{stef}. First, to any monoidal category $\mathcal{C}$ they associate a new monoidal category $\widetilde{\mathcal{H}}(\mathcal{C})$, called a Hom-category, whose objects are pairs consisting of
an object of $\mathcal{C}$ and an automorphism of this object ($\widetilde{\mathcal{H}}(\mathcal{C})$ has nontrivial associativity constraint even if
the one of $\mathcal{C}$ is trivial). By taking $\mathcal{C}$ to be ${_{\Bbbk }\mathcal{M}}$, the category of linear spaces over a base f\/ield $\Bbbk $, it turns out that an algebra in the (symmetric) monoidal category $\widetilde{\mathcal{H}}({_{\Bbbk }\mathcal{M}})$ is the same thing as a
Hom-associative algebra $(A,\mu ,\alpha )$ with bijective $\alpha $. The
bialgebras in $\widetilde{\mathcal{H}}({_{\Bbbk }\mathcal{M}})$ are the monoidal
Hom-bialgebras we mentioned before.

In \cite{giacomo}, the f\/irst author extended the construction of the
Hom-category $\widetilde{\mathcal{H}}(\mathcal{C})$ to include the action of a
given group~$\mathcal{G}$. Namely, given a monoidal category~$\mathcal{C}$,
a group $\mathcal{G}$, two elements $c,d\in Z(\mathcal{G})$ and $\nu $ an
automorphism of the unit object of $\mathcal{C}$, the group Hom-category $\mathcal{H}^{{c,d,\nu}}(\mathcal{G},\mathcal{C})$ has as objects pairs $(A,f_{A})$, where $A$ is an object in $\mathcal{C}$ and $f_{A}\colon \mathcal{G} \rightarrow \operatorname{Aut}_{\mathcal{C}}(A)$ is a group homomorphism. The associativity
constraint of $\mathcal{H}^{{c,d,\nu}}(\mathcal{G},\mathcal{C})$ is
naturally def\/ined by means of~$c$,~$d$,~$\nu $ (see Claim~\ref{Cl:monoidal} and Theorem~\ref{Th:Monoidal}) and it is, in general, non trivial. A~braided structure
is also def\/ined on $\mathcal{H}^{{c,d,\nu }}(\mathcal{G},\mathcal{C})$ (see Claim~\ref{cl:braid} and Theorem~\ref{Theo:braidcat}) turning it into a~braided
category which is symmetric whenever $\mathcal{C}$ is. When $\mathcal{G}=\mathbb{Z}$, $c=d=1_{\mathbb{Z}}$ and $\nu
=\operatorname{id}_{\mathbf{1}}$ one gets the category~$\mathcal{H}(\mathcal{C})$ from~\cite{stef}, while for $c=1_{\mathbb{Z}}$, $d=-1_{\mathbb{Z}}$ and $\nu =\operatorname{id}_{\mathbf{1}}$ one gets the category~$\widetilde{\mathcal{H}}(\mathcal{C})$.

We f\/irst look at the case when $\mathcal{\mathcal{G}}=\mathbb{Z\times Z}$, $c=(1,0) $, $d=(0,1) $, $\nu =\operatorname{id}_{\mathbf{1}}$ and
$\mathcal{C}={_{\Bbbk }\mathcal{M}}$.

If $M\in {_{\Bbbk }\mathcal{M}}$, a group homomorphism $f_{M}\colon
\mathbb{Z}
\times
\mathbb{Z}
\rightarrow \operatorname{Aut}_{\Bbbk }(M) $ is completely determined
by
\begin{gather*}
f_{M}((1,0)) =\alpha _{M}\qquad \text{and}\qquad
f_{M}((0,1)) =\beta _{M}^{-1}.
\end{gather*}
Thus, an object in $\mathcal{H}(
\mathbb{Z}
\times
\mathbb{Z}
,{_{\Bbbk }\mathcal{M}})$ identif\/ies with a triple $(
M,\alpha _{M},\beta _{M}) $, where $\alpha _{M},\beta _{M}\in \operatorname{Aut}_{\Bbbk }(M) $ and $\alpha _{M}\circ \beta _{M}=\beta
_{M}\circ \alpha _{M}$. For $( X,\alpha _{X},\beta _{X}) $, $( Y,\alpha _{Y},\beta _{Y}) $, $( Z,\alpha _{Z},\beta
_{Z}) $ objects in the category $\mathcal{H}^{(1,0),(0,1) ,1}(
\mathbb{Z}
\times
\mathbb{Z}
,{_{\Bbbk }\mathcal{M}})$, the associativity constraint in $\mathcal{H}^{(1,0) ,(0,1) ,1}(
\mathbb{Z}
\times
\mathbb{Z}
,{_{\Bbbk }\mathcal{M}})$ is given by
\begin{gather*}
\big( \overline{a}^{c,d,\nu }\big) _{( X,\alpha _{X},\beta
_{X}) ,( Y,\alpha _{Y},\beta _{Y}) ,( Z,\alpha
_{Z},\beta _{Z}) }=a_{X,Y,Z}\circ \big[ ( \alpha _{X}\otimes
Y ) \otimes \beta _{Z}^{-1}\big] ,
\end{gather*}
and the braiding is
\begin{gather*}
\overline{\gamma }_{( X,\alpha _{X},\beta _{X}) ,( Y,\alpha
_{Y},\beta _{Y}) }^{{c,d,\nu}}=\tau \big[ \big( \alpha _{X}\beta
_{X}^{-1}\big) \otimes \big( \alpha _{Y}^{-1}\beta _{Y}\big) \big] ,
\end{gather*}
where $\tau \colon X\otimes Y\rightarrow Y\otimes X$ denotes the usual f\/lip in the
category of linear spaces. Note that $\overline{\gamma }$ is a symmetric
braiding. Being $\mathcal{H}^{(1,0) ,(0,1) ,1}
(\mathbb{Z}
\times
\mathbb{Z}
,{_{\Bbbk }\mathcal{M}})$ an additive braided monoidal category,
all the concepts of algebra, Lie algebra and so on, can be introduced in
this case.

By writing down the axioms for an algebra in
$\mathcal{H}^{(1,0) ,(0,1) ,1}
(
\mathbb{Z}
\times
\mathbb{Z}
,{_{\Bbbk }\mathcal{M}})$
and discarding the invertibility of $\alpha $ and~$\beta $ if not needed, we arrived at the following concept. A~BiHom-associative algebra over $\Bbbk $ is a linear space $A$ endowed with a
multiplication $\mu \colon A\otimes A\rightarrow A$, $\mu (a\otimes b)=ab$, and
two commuting multiplicative linear maps $\alpha ,\beta \colon A\rightarrow A$
satisfying what we call the BiHom-associativity condition
\begin{gather*}
\alpha (a)(bc)=(ab)\beta (c),\qquad \forall \, a,b,c\in A.
\end{gather*}

Thus, a BiHom-associative algebra with \textit{bijective} structure maps is
exactly an algebra in $\mathcal{H}^{(1,0) ,(0,1) ,1}
(\mathbb{Z}
\times
\mathbb{Z}
,{_{\Bbbk }\mathcal{M}})$.

Obviously, a BiHom-associative algebra for which $\alpha =\beta $ is just a
Hom-associative algebra.

The remarkable fact is that the twisting principle may be also applied: if $A
$ is an associative algebra and $\alpha ,\beta \colon A\rightarrow A$ are two
commuting algebra maps, then $A$ with the new multiplication def\/ined by $
a\ast b=\alpha (a)\beta (b)$ is a BiHom-associative algebra, called the
\emph{Yau twist} of~$A$. As a matter of fact, although we arrived at the
concept of BiHom-associative algebra via the categorical machinery presented
above, it is the possibility of twisting the multiplication of an
associative algebra by \textit{two} commuting algebra endomorphisms that led
us to believe that BiHom-associative algebras are interesting objects in
their own. One can think of this as follows. Take again an associative
algebra $A$ and $\alpha ,\beta \colon A\rightarrow A$ two commuting algebra
endomorphisms; def\/ine a new multiplication on~$A$ by $a\ast b=\alpha
(a)\beta (b)$. Then it is natural to ask the following question: what kind
of structure is $( A,\ast )$? Example~\ref{giacomoex} in this paper shows
that, in general, $( A,\ast) $ is \textit{not} a Hom-associative
algebra, so the theory of Hom-associative algebras is \textit{not} general
enough to cover this natural operation of twisting the multiplication of an
associative algebra by \textit{two} maps; but this operation f\/its in the
framework of BiHom-associative algebras. The Yau twisting of an associative
algebra by two maps should thus be considered as the ``natural'' example of a~BiHom-associative algebra. We would like to emphasize that for this
operation the two maps are \textit{not} assumed to be bijective, so the
resulting BiHom-associative algebra has possibly \textit{non bijective}
structure maps and as such it \textit{cannot} be regarded, to our knowledge,
as an algebra in a monoidal category.

Take now the group $\mathcal{G}$ to be arbitrary. It is natural to describe how an
algebra in the monoidal category $\mathcal{H}^{c,d,\nu }(\mathcal{G},{_{\Bbbk }\mathcal{M}})$
looks like. By writing down the axioms, it turns out (see Claim~\ref{Cl:Alg2} and Remark \ref{remada})
that an algebra in such
a category is a BiHom-associative algebra with bijective structure maps
(the associativity of the algebra in the category is equivalent to the BiHom-associativity condition)
having some extra structure (like an action of the group on the algebra). So, morally,
the group $\mathcal{G}=\mathbb{Z}\times \mathbb{Z}$ leads to BiHom-associative algebras
but any other group would not lead to something like a ``higher'' structure than BiHom-associative algebras
(for instance, one cannot have something like TriHom-associative algebras).

We initiate in this paper the study of what we will call BiHom-structures.
The next structure we introduce is that of a BiHom-Lie algebra; for this, we
use also a categorical approach. Unlike the Hom case, to obtain a BiHom-Lie
algebra from a BiHom-associative algebra we need the structure maps $\alpha $
and $\beta $ to be bijective; the commutator is def\/ined by the formula $[a,b]=ab-\alpha ^{-1}\beta (b)\alpha \beta ^{-1}(a)$. Nevertheless, just as
in the Hom-case, the Yau twist works:
if $( L,[ -]) $ is a Lie algebra over a
f\/ield $\Bbbk $ and $\alpha ,\beta\colon L\rightarrow L$ are two commuting multiplicative linear
maps and we def\/ine the linear map $ \{ - \} \colon L\otimes L\rightarrow L$,
$ \{ a,b \} = [ \alpha  ( a ) ,\beta  ( b ) ]$, for all $a,b\in L$,
then $L_{( \alpha ,\beta ) }:=(L, \{ -\}, \alpha ,
\beta )$ is a BiHom-Lie algebra, called the {\em Yau twist} of $( L,[ -] ) $.

We def\/ine representations of
BiHom-associative algebras and BiHom-Lie algebras and f\/ind some of their
basic properties. Then we introduce BiHom-coassociative coalgebras and
BiHom-bialgebras together with some of the usual ingredients (comodules,
duality, convolution product, primitive elements, module and comodule
algebras). We def\/ine antipodes for a certain class of BiHom-bialgebras,
called monoidal BiHom-bialgebras, leading thus to the concept of monoidal
BiHom-Hopf algebras. We def\/ine smash products, as particular cases of
twisted tensor products, introduced in turn as a particular case of twisting
a~BiHom-associative algebra by what we call a~BiHom-pseudotwistor. We write
down explicitly such a smash product, obtained from an action of a~Yau
twist of the quantum group~$U_{q}(\mathfrak{sl}_{2})$ on a Yau twist of the
quantum plane $\mathbb{A}_{q}^{2|0}$.

As a f\/inal remark, let us note that one could introduce a less restrictive concept of BiHom-associative
algebra by dropping the assumptions that~$\alpha $ and~$\beta $ are multiplicative
and/or that they commute (note that all the examples of $q$-deformations of Witt or Virasoro algebras are not multiplicative). Unfortunately, by dropping any of these assumptions, one loses
the main class of examples, the Yau twists, in the sense that if~$A$ is an
associative algebra and $\alpha, \beta \colon A\rightarrow A$ are two arbitrary linear maps, and we def\/ine as before
$a*b=\alpha (a)\beta (b)$, then $(A, *)$ in general is not a BiHom-associative algebra even in this
more general sense.

\section[The category $\mathcal{H}(\mathcal{G},\mathcal{C})$]{The category $\boldsymbol{\mathcal{H}(\mathcal{G},\mathcal{C})}$}

Our aim in this section is to introduce so-called group
Hom-categories; proofs of the results in this section may be found in \cite{giacomo}.

\begin{Definition}
Let $\mathcal{G}$ be a group and let $\mathcal{C}$ be a category. The
\textit{group Hom-category} $\mathcal{H}(\mathcal{G},\mathcal{C})$ associated to $\mathcal{G}$ and $\mathcal{C}$ is the category having as objects pairs $(A,f_{A})$, where $A\in \mathcal{C}$ and $f_{A}$ is a~group homomorphism $\mathcal{G\rightarrow }\operatorname{Aut}_{\mathcal{C}}(A)$. A morphism $\xi \colon
(A,f_{A})\rightarrow (B,f_{B})$ in $\mathcal{H}(\mathcal{G},\mathcal{C})$ is a~morphism $\xi \colon A\rightarrow B$ in $\mathcal{C}$ such that $f_{B}(g) \circ \xi =\xi \circ f_{A}(g)$, for all $g\in \mathcal{G}$.
\end{Definition}

\begin{Definition}
A \textit{monoidal category} (see \cite[Chapter~XI]{Kassel}) is a category $\mathcal{C}$ endowed with an object $\mathbf{1}\in \mathcal{C}$ (called
\textit{unit}), a functor $\otimes \colon \mathcal{C}\times \mathcal{C}\rightarrow
\mathcal{C}$ (called \textit{tensor product}) and functorial isomorphisms $a_{X,Y,Z}\colon (X\otimes Y)\otimes Z\rightarrow  X\otimes (Y\otimes Z)$, $l_{X}\colon
\mathbf{1}\otimes X\rightarrow X$, $r_{X}\colon X\otimes \mathbf{1}\rightarrow X$,
for every $X$, $Y$, $Z$ in~$\mathcal{C}$. The functorial isomorphisms~$a$ are
called the \textit{associativity constraints} and satisfy
the pentagon axiom, that is
\begin{gather*}
(U\otimes a_{V,W,X})\circ a_{U,V\otimes W,X}\circ (a_{U,V,W}\otimes
X)=a_{U,V,W\otimes X}\circ a_{U\otimes V,W,X}  
\end{gather*}
holds true, for every $U$, $V$, $W$, $X$ in~$\mathcal{C}$. The isomorphisms~$l$ and~$r $ are called the \textit{unit constraints} and they obey the Triangle
Axiom, that is
\begin{gather*}
(V\otimes l_{W})\circ a_{V,\mathbf{1},W}=r_{V}\otimes W,\qquad \text{for every $V$, $W$ in $\mathcal{C}$}.  
\end{gather*}

A \textit{monoidal functor} $(F,\phi _{2},\phi _{0})\colon (\mathcal{C},\otimes ,\mathbf{1},a,l,r )\rightarrow (\mathcal{C}^{\prime }
,\otimes ^{\prime },\mathbf{1}^{\prime },a^{\prime },l^{\prime
},r^{\prime })$ between two monoidal categories consists of a
functor $F\colon \mathcal{C}\rightarrow \mathcal{C}^{\prime }$, an isomorphism $
\phi _{2}(U,V)\colon F(U)\otimes ^{\prime }F(V)\rightarrow F(U\otimes V)$, natural
in $U,V\in \mathcal{C}$, and an isomorphism $\phi _{0}\colon \mathbf{1}^{\prime
}\rightarrow F(\mathbf{1})$ such that the diagram
\begin{gather*}
\xymatrixcolsep{67pt}\xymatrixrowsep{35pt}\xymatrix{ (F(U)\!\otimes'\!
F(V))\!\otimes'\! F(W) \ar[d]|{a'_{F(U),F(V),F(W)}} \ar[r]^{\phi_2(U,V)\otimes'
F(W)} & F(U\!\otimes\! V)\!\otimes'\! F(W) \ar[r]^{\phi_2(U\otimes V,W)} &
F((U\!\otimes\! V)\!\otimes\! W) \ar[d]|{F(a_{ U,V, W})} \\ F(U)\!\otimes'\!
(F(V)\!\otimes'\! F(W)) \ar[r]^{F(U)\otimes' \phi_2(V,W)} & F(U)\!\otimes'\!
F(V\!\otimes\! W) \ar[r]^{\phi_2(U,V\otimes W)} & F(U\!\otimes\! (V\!\otimes\! W)) }
\end{gather*}
is commutative, and the following conditions are satisf\/ied
\begin{gather*}
{F(l_{U})}\circ {\phi _{2}(\mathbf{1},U)}\circ ({\phi _{0}\otimes }^{\prime }
{F(U)})={l}^{\prime }{_{F(U)}},
\qquad
{F(r_{U})}\circ {\phi _{2}(U,\mathbf{1})}\circ ({F(U)\otimes }^{\prime }{\phi _{0}})={r}^{\prime }{_{F(U)}}.
\end{gather*}
\end{Definition}

\begin{Claim}
\label{Cl:monoidal} Let $\mathcal{G}$ be a group and let $( \mathcal{C}
,\otimes ,\mathbf{1},a,l,r) $ be a monoidal category. Given any pair
of objects $(A,f_{A}),(B,f_{B})\in \mathcal{H}(\mathcal{G},\mathcal{C})$, consider the
map $f_{A}\otimes f_{B}\colon \mathcal{G}\rightarrow\operatorname{Aut}_{\mathcal{C}
}(A\otimes B)$ defined by setting
\begin{gather*}
( f_{A}\otimes f_{B}) (g) =f_{A}(g)
\otimes f_{B}(g),  
\end{gather*}
for all $g\in \mathcal{G}$. Then $f_{A}\otimes f_{B}$ is a group
homomorphism and hence
\begin{gather*}
( A\otimes B,f_{A}\otimes f_{B}) \in \mathcal{H}(\mathcal{G},\mathcal{C}).
\end{gather*}
Moreover, if $\phi\colon (A,f_{A})\rightarrow (\tilde{A},f_{\tilde{A}})$ and $\xi
\colon (B,f_{B})\rightarrow (\tilde{B},f_{\tilde{B}})$ are morphisms in $\mathcal{H}(\mathcal{G},\mathcal{C})$, then
\begin{gather*}
\phi \otimes \xi \colon \  ( A\otimes B,f_{A}\otimes f_{B} ) \rightarrow \big(\tilde{A}\otimes \tilde{B},f_{\tilde{A}}\otimes f_{\tilde{B}}\big)
\end{gather*}
is a morphism in $\mathcal{H}(\mathcal{G},\mathcal{C})$.

Let $Z(\mathcal{G}) $ be the center of $\mathcal{G}$ and let $
c\in Z(\mathcal{G}) $. Then we can consider the functorial
isomorphism $\varphi ( c) \colon \operatorname{Id}_{\mathcal{H}(\mathcal{G},\mathcal{C})
}\rightarrow \operatorname{Id}_{\mathcal{H}(\mathcal{G},\mathcal{C})}$ def\/ined by setting
\begin{gather*}
\varphi (c) (A,f_{A})=f_{A}(c), \qquad \text{for every  $(A,f_{A})$ in $\mathcal{H}(\mathcal{G},\mathcal{C})$}.
\end{gather*}
Also, let $\widehat{\operatorname{Id}_{\mathbf{1}}}\colon \mathcal{G}\rightarrow \operatorname{Aut}_{\mathcal{C}}( \mathbf{1}) $ denote the constant map equal
to $\operatorname{Id}_{\mathbf{1}}$.

Let $c,d\in Z(\mathcal{G}) $ and let $\nu \in \operatorname{Aut}_{\mathcal{C}}( \mathbf{1}) $. We set
\begin{gather*}
\overline{a}^{c,d,\nu }=a\circ \big[ \big( \varphi (c) \otimes
\operatorname{Id}_{\mathcal{H}(\mathcal{G},\mathcal{C})}\big) \otimes \varphi  (
d ) \big],
\qquad
\overline{l}^{c,d,\nu }=\varphi \big( d^{-1}\big) \circ l\circ \big( \nu
\otimes \operatorname{Id}_{\mathcal{H}(\mathcal{G},\mathcal{C})}\big),
\\
\overline{r}^{c,d,\nu }=\varphi (c) \circ r\circ \big( \operatorname{Id}_{\mathcal{H}(\mathcal{G},\mathcal{C})}\otimes \nu \big) .
\end{gather*}
\end{Claim}

\begin{Theorem}
\label{Th:Monoidal}In the setting of Claim~{\rm \ref{Cl:monoidal}}, the category
\begin{gather*}
\mathcal{H}^{{c,d,\nu}}(\mathcal{G},\mathcal{C})=\big( \mathcal{H}(\mathcal{G},
\mathcal{C}),\otimes ,\big( \mathbf{1,}\widehat{\operatorname{Id}_{\mathbf{1}}}
\big) ,\overline{a}^{c,d,\nu },\overline{l}^{c,d,\nu },\overline{r}
^{c,d,\nu }\big)
\end{gather*}
is monoidal.
\end{Theorem}

From now on, when $( \mathcal{C},\otimes ,\mathbf{1},a,l,r) $ is
a monoidal category, $\mathcal{G}$ is a group, $c,d\in Z( \mathcal{G}
) $ and $\nu \in \operatorname{Aut}_{\mathcal{C}}( \mathbf{1}) $,
we will indicate the monoidal category def\/ined in Theorem~\ref{Th:Monoidal}
by $\mathcal{H}^{c,d,\nu }(\mathcal{G},\mathcal{C})$. In the case when $c=d=
\mathbf{1}_{\mathcal{G}}$ and $\nu =\operatorname{Id}_{\mathbf{1}}$, we will
simply write $\mathcal{H}(\mathcal{G},\mathcal{C})$.

\begin{Theorem}
\label{thm:monoidaliso}Let $( \mathcal{C},\otimes ,\mathbf{1}
,a,l,r) $ be a monoidal category and~$\mathcal{G}$ a group. Then the
identity functor $\mathcal{I}\colon \mathcal{H}^{c,d,\nu }(\mathcal{G},\mathcal{C})
\rightarrow \mathcal{H}(\mathcal{G},\mathcal{C})$ is a~monoidal isomorphism via
\begin{gather*}
 \phi _{0}=\nu ^{-1}\colon \ \big(\mathbf{1},\widehat{\operatorname{Id}_{\mathbf{1}}}
\big)\rightarrow \big(\mathbf{1},\widehat{\operatorname{Id}_{\mathbf{1}}}\big)\qquad \text{and}
\qquad \phi _{2} ( (A,f_{A}),(B,f_{B}) ) =f_{A}\big(c^{-1}\big)\otimes
f_{B}(d),
\end{gather*}
for every $(A,f_{A}),(B,f_{B})\in \mathcal{H}^{c,d,\nu }(\mathcal{G},\mathcal{C})$.
\end{Theorem}

\begin{Definition}[see \cite{Kassel}]
\label{def braiding} A \textit{braided monoidal category}
$(\mathcal{C},\otimes ,\mathbf{1},a,l,r,\gamma )$ is a monoidal category $(\mathcal{C},\otimes ,\mathbf{1,}a,l,r)$ equipped with a \textit{braiding} $\gamma $, that is, an isomorphism $\gamma _{U,V}\colon U\otimes V\rightarrow
V\otimes U$, natural in $U,V\in \mathcal{C}$, satisfying, for all $U,V,W\in
\mathcal{C}$, the hexagon axioms
\begin{gather*}
a_{V,W,U}\circ \gamma _{U,V\otimes W}\circ a_{U,V,W}=(V\otimes \gamma
_{U,W})\circ a_{V,U,W}\circ (\gamma _{U,V}\otimes W),  
\\
a_{W,U,V}^{-1}\circ \gamma _{U\otimes V,W}\circ a_{U,V,W}^{-1}=(\gamma
_{U,W}\otimes V)\circ a_{U,W,V}^{-1}\circ (U\otimes \gamma _{V,W}).
\end{gather*}
A braided monoidal category is called \textit{symmetric} if we further have $
\gamma _{V,U}\circ \gamma _{U,V}=\operatorname{Id}_{U\otimes V}$ for every $U,V\in
\mathcal{C}$. A \textit{braided monoidal functor} is a monoidal functor $F\colon
\mathcal{C}\rightarrow \mathcal{C}^{\prime }$ such that
\begin{gather*}
F( \gamma _{U,V}) \circ \phi _{2}(U,V)=\phi _{2}(V,U)\circ \gamma
_{F(U) ,F(V) }^{\prime }, \qquad \text{for  every} \ \ U,
V\in {\mathcal{C}}.  
\end{gather*}
\end{Definition}

\begin{Claim}
\label{cl:braid}Let $\mathcal{G}$ be a group and let $( \mathcal{C},\otimes ,\mathbf{1},a,l,r,\gamma) $ be a braided monoidal category.
Let $c,d\in Z(\mathcal{G}) $ and let $\nu \in \operatorname{Aut}_{\mathcal{C}}( \mathbf{1}) $. We will introduce a braided
structure on the monoidal category $\mathcal{H}^{{c,d,\nu}}(\mathcal{G},\mathcal{C})$ by setting, for every $(A,f_{A})$ and $( B,f_{B}) $
in $\mathcal{H}(\mathcal{G},\mathcal{C})$,
\begin{gather*}
\overline{\gamma }_{(A,f_{A}), ( B,f_{B} ) }^{{c,d,\nu}}=\gamma
_{A,B}\circ \big( f_{A}(cd)\otimes f_{B}\big(c^{-1}d^{-1}\big)\big) .
\end{gather*}
\end{Claim}

\begin{Theorem}
\label{Theo:braidcat}In the setting of Claim~{\rm \ref{cl:braid}}, the category
\begin{gather*}
\big(\mathcal{H}(\mathcal{\mathcal{G}},\mathcal{C}),\otimes ,\big( \mathbf{1},
\widehat{\operatorname{Id}_{\mathbf{1}}}\big) ,\overline{a}^{{c,d,\nu}},
\overline{l}^{{c,d,\nu}},\overline{r}^{{c,d,\nu}},\overline{\gamma }
^{{c,d,\nu}}\big)
\end{gather*}
is a braided monoidal category.
\end{Theorem}

From now on, when $( \mathcal{C},\otimes ,\mathbf{1},a,l,r,\gamma
) $ is a braided monoidal category and $\mathcal{G}$ is a group, we
will still denote the braided monoidal structure def\/ined in Theorem~\ref{Theo:braidcat} with $\mathcal{H}^{{c,d,\nu}}(\mathcal{G},\mathcal{C})$.
In the case when $c=d=\mathbf{1}_{\mathcal{G}}$ and $\nu =\operatorname{id}_{\mathbf{1}}$, we will simply write respectively $\mathcal{H}(\mathcal{G},\mathcal{C})$
instead of $\mathcal{H}^{{c,d,\nu}}(\mathcal{G},\mathcal{C})$ and $\gamma
_{(A,f_{A}),(B,f_{B})}$ instead of $\overline{\gamma }_{(A,f_{A}),(B,f_{B})}^{{c,d,\nu}}$.

\begin{Theorem}
\label{teo:htuttiiso}Let $\mathcal{G}$ be a group and let $( \mathcal{C}
,\otimes ,\mathbf{1},a,l,r,\gamma ) $ be a braided monoidal category.
Then the identity functor $\mathcal{I}\colon \mathcal{H}^{c,d,\nu }(\mathcal{G},
\mathcal{C})\rightarrow \mathcal{H}(\mathcal{G},\mathcal{C})$ is a braided monoidal
isomorphism via
\begin{gather*}
\phi _{0}=\nu ^{-1}\colon \ \big(\mathbf{1},\widehat{\operatorname{Id}_{\mathbf{1}}}
\big)\rightarrow (\mathbf{1},\widehat{\operatorname{Id}_{\mathbf{1}}})\qquad \text{and}
\qquad\phi _{2} ( (A,f_{A}),(B,f_{B}) ) =f_{A}\big(c^{-1}\big)\otimes
f_{B}(d),
\end{gather*}
for every $(A,f_{A}),(B,f_{B})\in \mathcal{H}^{c,d,\nu }(\mathcal{G},\mathcal{C})$.
\end{Theorem}

\begin{Remark}
Let $\mathcal{G}$ be a torsion-free abelian group. Corollary~4 in \cite{ABM-HomLie} states that, up to a~braided monoidal category
isomorphism, there is a unique braided monoidal structure (actually symmetric) on the category
of representations over the group algebra $\Bbbk[\mathcal{G}]$, considered monoidal via a structure induced by
that of vector spaces over the f\/ield~$\Bbbk $. Thus Theorem~\ref{teo:htuttiiso} can be deduced from this result whenever $\mathcal{G}$ is a torsion-free abelian group. We should remark that this result in~\cite{ABM-HomLie} stems from the fact that the third Harrison cohomology group
$H_{\mathrm{Harr}}^{3}(\mathcal{G},\Bbbk ,\mathbb{G}_{m})$ has, in this case, just one element.
If $\mathcal{G}$ is not a torsion-free abelian group then this might not happen. As one of the referees pointed out, in the case when $\Bbbk=\mathbb{C} $ and $\mathcal{G}=C_2$ then $H_{\mathrm{Harr}}^{3}(\mathcal{G},\Bbbk ,\mathbb{G}_{m})$ has exactly two elements and so in this case there are two distinct equivalence classes of braided monoidal structures on the category
of representations over the group algebra~$\Bbbk[\mathcal{G}]$, considered monoidal via a structure induced by
that of vector spaces over the f\/ield~$\Bbbk $. This does not contradict our Theorem~\ref{teo:htuttiiso}. In fact, there might exist braided monoidal structures dif\/ferent from the ones considered in the statement of Theorem~\ref{teo:htuttiiso}.
\end{Remark}

\begin{Claim}
\label{Cl:alg}Let $ ( \mathcal{C},\otimes ,\mathbf{1},a,l,r ) $ be
a monoidal category and $\mathcal{G}$ a group, let $c,d\in Z ( \mathcal{G} ) $ and $\nu \in \operatorname{Aut}_{\mathcal{C}} ( \mathbf{1} ) $. A \textit{unital algebra} in $\mathcal{H}^{c,d,\nu }(\mathcal{G},\mathcal{C})$ is a triple $( (A,f_{A}),\mu ,u) $ where

\begin{itemize}\itemsep=0pt
\item[1)] $(A,f_{A})\in \mathcal{H}(\mathcal{G},\mathcal{C})$;

\item[2)] $\mu \colon (A\otimes A, f_A\otimes f_A)\rightarrow (A, f_A)$ is a
morphism in $\mathcal{H}(\mathcal{G},\mathcal{C})$;

\item[3)] $u\colon (\mathbf{1},\widehat{\operatorname{Id}_{\mathbf{1}}})\rightarrow (A,
f_A)$ is a morphism in $\mathcal{H}(\mathcal{G},\mathcal{C})$;

\item[4)] $\mu \circ  ( \mu \otimes A ) =\mu \circ  ( A\otimes
\mu  ) \circ \overline{a}_{A,A,A}^{c,d,\nu }$;

\item[5)] $\mu \circ  ( u\otimes A ) \circ  ( \overline{l}_{A}^{c,d,\nu } ) ^{-1}=\operatorname{Id}_{A}$;

\item[6)] $\operatorname{Id}_{A}=\mu \circ  ( A\otimes u ) \circ  (
\overline{r}_{A}^{c,d,\nu } ) ^{-1}$.
\end{itemize}
\end{Claim}

\begin{Definition}
Given a monoidal category $\mathcal{M}$, a quadruple $(A,\mu ,u,c)$ is
called a \textit{braided unital algebra} in $\mathcal{M}$ if (for
simplicity, we will omit to write the associators):

\begin{itemize}\itemsep=0pt
\item $(A,\mu ,u)$ is a unital algebra in $\mathcal{M}$;

\item $(A,c)$ is a braided object in $\mathcal{M}$, i.e., $c\colon A\otimes
A\rightarrow A\otimes A$ is invertible and satisf\/ies the Yang--Baxter equation
\begin{gather*}
(c\otimes A) (A\otimes c) (c\otimes A)
=(A\otimes c) (c\otimes A) (A\otimes c) ;
\end{gather*}

\item the following conditions hold:
\begin{gather*}
c(\mu \otimes A)=(A\otimes \mu )(c\otimes A)(A\otimes c), \qquad 
c(A\otimes \mu )=(\mu \otimes A)(A\otimes c) (c\otimes A),\\ 
c(u\otimes A)l_{A}^{-1}= ( A\otimes u ) r_{A}^{-1},\qquad
c(A\otimes u)r_{A}^{-1}= ( u\otimes A ) l_{A}^{-1}.  
\end{gather*}
\end{itemize}

A braided unital algebra is called \textit{symmetric} whenever $c^{2}=\operatorname{Id}_{A}$.
\end{Definition}

\begin{Definition}
\label{def:Lie}Given an additive monoidal category $\mathcal{M}$, a \textit{braided Lie algebra} in $\mathcal{M}$ consists of a triple $( L,c,[
-] \colon L\otimes L\rightarrow L) $ where $( L,c) $ is a~braided object and the following equalities hold true:
\begin{gather}
[-] =-[-] \circ c \quad \text{(skew-symmetry)};\nonumber \\
[-] \circ ( L\otimes [-] ) \circ \big[
\operatorname{Id}_{L\otimes  ( L\otimes L ) }+ ( L\otimes c )
a_{L,L,L} ( c\otimes L ) a_{L,L,L}^{-1}\nonumber\\
\qquad{} +a_{L,L,L} ( c\otimes
L ) a_{L,L,L}^{-1} ( L\otimes c ) \big] =0 \quad \text{(Jacobi condition)} ; \nonumber \\
c\circ  ( L\otimes [-]  ) a_{L,L,L}= (  [ -
 ] \otimes L ) a_{L,L,L}^{-1} ( L\otimes c )
a_{L,L,L} ( c\otimes L ) ;  \label{cucu} \\
c\circ  ( [-] \otimes L ) a_{L,L,L}^{-1}= (
L\otimes [-]  ) a_{L,L,L} ( c\otimes L )
a_{L,L,L}^{-1}\left( L\otimes c\right) .  \label{lala}
\end{gather}
Let $\mathcal{M}$ be an additive braided monoidal category. A \textit{Lie
algebra} in $\mathcal{M}$ consists of a pair $ ( L,[-]
\colon L\otimes L\rightarrow L ) $ such that $ ( L,c_{L,L},[-]
 ) $ is a braided Lie algebra in the additive monoidal category $\mathcal{M}$, where $c_{L,L}$ is the braiding~$c$ of~$\mathcal{M}$ evaluated
on $L$ (note that in this case the conditions~(\ref{cucu}) and~(\ref{lala})
are automatically satisf\/ied).
\end{Definition}

\begin{Claim}
\label{Cl:LIE}Given a symmetric algebra $ ( A,\mu , u, c ) $, one
has that $[-] :=\mu \circ ( \operatorname{Id}_{A\otimes
A}-c) $ def\/ines a braided Lie algebra structure on $A$ (see \cite[Construction~2.16]{GV}).

In a symmetric monoidal category $( \mathcal{C},\otimes ,\mathbf{1},a,l,r,c)$, it is well known that any unital algebra $( A,\mu
,u) $ gives rise to a braided unital algebra $( A,\mu
,u,c_{A,A})$.
\end{Claim}

\section{Generalized Hom-structures}

Let $\Bbbk $ be a f\/ield and let ${_{\Bbbk }\mathcal{M}}$ be the
category of linear spaces regarded as a braided monoidal category in the
usual way. Then, for every group $\mathcal{G}$, the category $\mathcal{H}(\mathcal{G},{_{\Bbbk }\mathcal{M}})$ identif\/ies with the category $\Bbbk  [
\mathcal{G} ] $-$\text{Mod}$ of left modules over the group algebra $\Bbbk
 [ \mathcal{G} ] $.

Let $c,d\in Z(\mathcal{G}) $ and $\nu $ an automorphism of $\Bbbk $ regarded as linear space over~$\Bbbk $, that is~$\nu $ is the
multiplication by an element of $\Bbbk {\setminus}  \{ 0 \} $ that we
will also denote by~$\nu $. Note that, given $X,Y,Z\in \Bbbk  [ \mathcal{G} ] $-$\text{Mod}$, we have
\begin{gather*}
\overline{a}_{X,Y,Z}^{c,d,\nu } (  ( x\otimes y ) \otimes
z ) =c\cdot x\otimes  ( y\otimes d\cdot z ) , \qquad \text{for every} \ \
x\in X,y\in Y,z\in Z,
\\
\overline{l}_{X}^{c,d,\nu } ( t\otimes x ) =d^{-1}\cdot  ( \nu
tx ) \qquad \text{and} \qquad \overline{r}_{X}^{c,d,\nu } ( x\otimes t )
=c\cdot  ( \nu tx ) ,\\
\qquad{}\text{for every} \ \ t\in \Bbbk \ \ \text{and} \ \ x\in
X,
\end{gather*}
so that
\begin{gather*}
\big( \overline{l}_{A}^{c,d,\nu }\big) ^{-1} ( x ) = ( \nu
^{-1}\otimes d\cdot x ) \qquad \text{and} \qquad \big( \overline{r}^{c,d,\nu
}\big) ^{-1} ( x ) =\big( c^{-1}\cdot x\otimes \nu ^{-1}\big) , \\
\qquad{}
\text{for every} \ \ x\in X.
\end{gather*}
The unit object of $\mathcal{H}^{c,d,\nu }(\mathcal{G},{_{\Bbbk }\mathcal{M}}
)$ is $ \{ 1_{\Bbbk } \}$ regarded as a left $\Bbbk  [ \mathcal{G} ] $-module in the trivial way.

\begin{Claim}
\label{Cl:Alg2}In view of \ref{Cl:alg}, a unital algebra in $\mathcal{H}^{c,d,\nu }(\mathcal{G},{_{\Bbbk }\mathcal{M}})$ is a
triple $ ( (A,f_{A}),\mu ,u ) $, where

\begin{itemize}\itemsep=0pt
\item[1)] $A\in \Bbbk  [ \mathcal{G} ]$-$\text{Mod}$;

\item[2)] $\mu \colon A\otimes A\rightarrow A$ is a morphism in $\Bbbk  [
\mathcal{G} ] $-$\text{Mod}$, i.e., $g\cdot (ab)= ( g\cdot a )  (
g\cdot b ) $, for every $g\in \mathcal{G}$, $a,b\in A$;

\item[3)] $u\colon  \{ 1_{\Bbbk } \} \rightarrow A$ is a morphisms in $\Bbbk  [ \mathcal{G}] $-$\text{Mod}$, i.e., $g\cdot u ( 1_{\Bbbk
} ) =u ( 1_{\Bbbk } ) $, for every $g\in \mathcal{G}$;

\item[4)] $( x\cdot y) \cdot z=( c\cdot x) \cdot [
y\cdot ( d\cdot z) ] $, for every $x,y,z\in A$, which is
equivalent to
\begin{gather*}
 ( c\cdot x ) \cdot  ( y\cdot z ) = ( x\cdot
y ) \cdot \big( d^{-1}\cdot z\big), \qquad \forall \, x,y,z\in A;
\end{gather*}

\item[5)] $u ( \nu ^{-1} ) \cdot  ( d\cdot x ) =x$, for
every $x\in A$;

\item[6)] $ ( c^{-1}\cdot x ) \cdot u ( \nu ^{-1} ) =x$,
for every $x\in A$.
\end{itemize}
\end{Claim}

Note that when $c=d=1_{\mathcal{G}}$ and $\nu =1_{\Bbbk }$, it turns out
that $A$ is simply a $\Bbbk  [ \mathcal{G} ]$-module algebra.

\begin{Example}
\label{Ex:zetzet}Let $M$ be a ${\Bbbk }$-linear space and $\mathcal{G}=
\mathbb{Z}\times \mathbb{Z}$. Then a group morphism $f_{M}\colon
\mathbb{Z}
\times
\mathbb{Z}
\rightarrow \operatorname{Aut}_{\Bbbk }(M) $ is completely determined
by
\begin{gather*}
f_{M}((1,0)) =\alpha _{M}\qquad \text{and}\qquad
f_{M}((0,1)) =\beta _{M}^{-1}.
\end{gather*}
Thus an object in $\mathcal{H}(
\mathbb{Z}
\times
\mathbb{Z}
,{_{\Bbbk }\mathcal{M}})$ identif\/ies with a triple $(
M,\alpha _{M},\beta _{M}) $, where $\alpha _{M},\beta _{M}\in \operatorname{Aut}_{\Bbbk }(M) $ and $\alpha _{M}\circ \beta _{M}=\beta
_{M}\circ \alpha _{M}$. Also, a morphism $f\colon  ( M,\alpha _{M},\beta
_{M} ) \rightarrow  ( N,\alpha _{N},\beta _{N} ) $ is just a
linear map $f\colon M\rightarrow N$ such that $f\circ \alpha _{M}=\alpha _{N}\circ
f$ and $f\circ \beta _{M}=\beta _{N}\circ f$. Moreover, the tensor product,
in the category, of the objects $ ( M,\alpha _{M},\beta _{M} ) $
and $ ( N,\alpha _{N},\beta _{N} ) $ is the object $(M\otimes
N,\alpha _{M}\otimes \alpha _{N},\beta _{M}\otimes \beta _{N})$.

We set $c=(1,0) $, $d=(0,1) $ and $\nu =1_{\Bbbk }$.
For $ ( X,\alpha _{X},\beta _{X} ) $, $ ( Y,\alpha _{Y},\beta
_{Y} ) $, $ ( Z,\alpha _{Z},\beta _{Z} ) $ objects in $
\mathcal{H}^{(1,0) ,(0,1) ,1}(
\mathbb{Z}
\times
\mathbb{Z}
,{_{\Bbbk }\mathcal{M}})$, the associativity constraints in $
\mathcal{H}^{(1,0) ,(0,1) ,1}(
\mathbb{Z}
\times
\mathbb{Z}
,{_{\Bbbk }\mathcal{M}})$ are given by
\begin{gather*}
 \big( \overline{a}^{c,d,\nu }\big) _{( X,\alpha _{X},\beta
_{X}) ,( Y,\alpha _{Y},\beta _{Y}) ,( Z,\alpha
_{Z},\beta _{Z}) }\colon \ (X\otimes Y)\otimes Z\rightarrow X\otimes (Y\otimes
Z), \\
 \big( \overline{a}^{c,d,\nu }\big) _{( X,\alpha _{X},\beta
_{X}) ,( Y,\alpha _{Y},\beta _{Y}) ,( Z,\alpha
_{Z},\beta _{Z}) }=a_{X,Y,Z}\circ \big[  ( \alpha _{X}\otimes
Y ) \otimes \beta _{Z}^{-1}\big] ,
\end{gather*}
and the braiding is
\begin{gather*}
\overline{\gamma }_{( X,\alpha _{X},\beta _{X}) ,( Y,\alpha
_{Y},\beta _{Y}) }^{{c,d,\nu}}=\tau \big[ \big( \alpha _{X}\beta
_{X}^{-1}\big) \otimes \big( \alpha _{Y}\beta _{Y}^{-1}\big) ^{-1}
\big] =\tau \big[ \big( \alpha _{X}\beta _{X}^{-1}\big) \otimes \big(
\alpha _{Y}^{-1}\beta _{Y}\big) \big] ,
\end{gather*}
where $\tau \colon X\otimes Y\rightarrow Y\otimes X$ denotes the usual f\/lip in the
category of linear spaces. Note that $\overline{\gamma }$ is a symmetric braiding.

Then, in view of~\ref{Cl:Alg2}, an algebra in $\mathcal{H}^{(1,0) ,(0,1) ,1}(
\mathbb{Z}
\times
\mathbb{Z}
,{_{\Bbbk }\mathcal{M}})$ is a triple $(( A,\alpha
,\beta ) ,\mu ,u) $, where
\begin{itemize}\itemsep=0pt
\item[1)] $\alpha ,\beta \in \operatorname{Aut}_{\Bbbk } ( A ) $ and $\alpha \circ \beta =\beta \circ \alpha$;

\item[2)] $\mu \colon  ( A\otimes A,\alpha \otimes \alpha ,\beta \otimes \beta
 ) \rightarrow  ( A,\alpha ,\beta  ) $ is a morphism in $\Bbbk
[
\mathbb{Z}
\times
\mathbb{Z}
] $-$\text{Mod}$, i.e., $\alpha (a\cdot b)=\alpha  ( a ) \cdot \alpha
(b) $ and $\beta (a\cdot b)=\beta (a) \cdot \beta
(b) $ for every $a,b\in A$;

\item[3)] $u\colon  \{ 1_{\Bbbk } \} \rightarrow  ( A,\alpha ,\beta
 ) $ is a morphisms in $\Bbbk  [
\mathbb{Z}
\times
\mathbb{Z}
] $-$\text{Mod}$, i.e., $\alpha  ( u ( 1_{\Bbbk } )  )
=u ( 1_{\Bbbk } ) $ and $\beta  ( u ( 1_{\Bbbk } )
 ) =u ( 1_{\Bbbk } )$;

\item[4)] $\alpha (x) \cdot  ( y\cdot z ) = (
x\cdot y ) \cdot \beta (z) $, for every $x,y,z\in A;$

\item[5)] $u ( 1_{\Bbbk } ) \cdot  ( \beta ^{-1}(x)
 ) =x$, for every $x\in A$, which is equivalent to $u( 1_{\Bbbk
} ) \cdot x=\beta (x) $, for every $x\in A$;

\item[6)] $ ( \alpha ^{-1}(x)  ) \cdot u ( 1_{\Bbbk
} ) =x$, for every $x\in A$, which is equivalent to $x\cdot u (
1_{\Bbbk } ) =\alpha (x) $, for every $x\in A$.
\end{itemize}
\end{Example}

Inspired by Example~\ref{Ex:zetzet}, we introduce the following concept.

\begin{Definition}
Let ${\Bbbk}$ be a f\/ield. A \textit{BiHom-associative algebra} over $\Bbbk $
is a 4-tuple $( A,\mu ,\alpha ,\beta )$, where $A$ is
a $\Bbbk $-linear space, $\alpha \colon A\rightarrow A$, $\beta \colon A\rightarrow A$
and $\mu \colon A\otimes A\rightarrow A$ are linear maps, with notation $\mu
 ( a\otimes a^{\prime } ) =aa^{\prime }$, satisfying the following
conditions, for all $a,a^{\prime },a^{\prime \prime }\in A$:
\begin{gather*}
\alpha \circ \beta =\beta \circ \alpha , \\
\alpha  ( aa^{\prime } ) =\alpha (a) \alpha  (
a^{\prime } ) \qquad \text{and} \qquad \beta  ( aa^{\prime } ) =\beta
(a) \beta  ( a^{\prime } )  \quad \text{(multiplicativity)},  \\ 
\alpha (a)  ( a^{\prime }a^{\prime \prime } ) = (
aa^{\prime } ) \beta  ( a^{\prime \prime } )  \quad \text{(BiHom-associativity)}.  
\end{gather*}

We call $\alpha $ and $\beta $ (in this order) the \textit{structure maps}
of~$A$.

A morphism $f\colon (A, \mu _A , \alpha _A, \beta _A)\rightarrow (B, \mu _B ,
\alpha _B, \beta _B)$ of BiHom-associative algebras is a linear map $f\colon A\rightarrow B$ such that $\alpha _B\circ f=f\circ \alpha _A$, $\beta
_B\circ f=f\circ \beta _A$ and $f\circ \mu_A=\mu _B\circ (f\otimes f)$.

A BiHom-associative algebra $(A, \mu, \alpha, \beta )$ is called \textit{unital} if there exists an element $1_A\in A$ (called a \textit{unit}) such
that $\alpha (1_A)=1_A$, $\beta (1_A)=1_A$ and
\begin{gather*}
 a1_A=\alpha (a) \qquad \text{and}\qquad 1_Aa=\beta (a), \qquad \forall \, a\in A.
\end{gather*}

A morphism of unital BiHom-associative algebras $f\colon A\rightarrow B$ is called
\textit{unital} if $f(1_A)=1_B$.
\end{Definition}

\begin{Remark}
A Hom-associative algebra $ ( A,\mu ,\alpha ) $ can be regarded as
the BiHom-associative algebra $ ( A,\mu ,\alpha ,\alpha ) $.
\end{Remark}

\begin{Remark}\label{remada}
A BiHom-associative algebra with \textit{bijective} structure maps is
exactly an algebra in $\mathcal{H}^{(1,0) ,(0,1) ,1}
(
\mathbb{Z}
\times
\mathbb{Z}
,{_{\Bbbk }\mathcal{M}})$.
On the other hand, in the setting of Claim~\ref{Cl:Alg2}, if we def\/ine the maps
$\alpha , \beta \colon A\rightarrow A$ by $\alpha (a)=c\cdot a$ and $\beta (a)=d^{-1}\cdot a$, for
all $a\in A$, the axiom $2)$ in Claim~\ref{Cl:Alg2} implies that $\alpha $ and $\beta $ are multiplicative and then
the axiom $4)$ in Claim~\ref{Cl:Alg2} says that $(A, \mu , \alpha , \beta )$ is a~BiHom-associative algebra.
\end{Remark}

\begin{Example}
We give now two families of examples of 2-dimensional unital BiHom-associative algebras, that are obtained by a computer algebra system. Let $\{e_1,e_2\}$ be a basis; for $i=1,2$ the maps $\alpha_i$, $\beta_i $ and the multiplication $\mu_i$ are def\/ined by
\begin{gather*}
  \alpha_1(e_1)=e_1, \qquad \alpha_1(e_2)=\frac{2 a}{b-1}e_1-e_2, \\
  \beta_1(e_1)=e_1, \qquad \beta_1(e_2)=-ae_1+b e_2, \\
  \mu_1(e_1,e_1)=e_1, \qquad \mu_1(e_1,e_2)= -ae_1+b e_2, \\
  \mu_1(e_2,e_1)=\frac{2 a}{b-1}e_1-e_2, \qquad \mu_1(e_2,e_2)=-\frac{
a^2(b-2)}{(b-1)^2}e_1+ae_2,
\end{gather*}
and
\begin{gather*}
  \alpha_2(e_1)=e_1,\qquad \alpha_2(e_2)=\frac{b(1-a)}{a}e_1+a e_2, \\
 \beta_2(e_1)=e_1,\qquad \beta_2(e_2)=be_1+(1-a) e_2, \\
  \mu_2(e_1,e_1)=e_1, \qquad \mu_2(e_1,e_2)= be_1+(1-a) e_2, \\
 \mu_2(e_2,e_1)=\frac{b(1-a)}{a}e_1+ae_2, \qquad \mu_2(e_2,e_2)=\frac{b}{a}e_2,
\end{gather*}
where $a$, $b$ are parameters in ${\Bbbk}$, with $b\neq 1$ in the f\/irst case
and $a\neq 0$ in the second. In both cases, the unit is~$e_1$.
\end{Example}

\begin{Claim}\label{cl:YAU}
In view of Theorem \ref{thm:monoidaliso}, if $( A,\mu ,\alpha ,\beta
) $ is a BiHom-associative algebra, and~$\alpha $ and~$\beta $ are
invertible, then $( A,\mu \circ ( \alpha ^{-1}\otimes \beta
^{-1}), \operatorname{Id}_{A},\operatorname{Id}_{A}) $ is a~BiHom-associative algebra, i.e., the multiplication $\mu \circ ( \alpha
^{-1}\otimes \beta ^{-1}) $ is associative in the usual sense.
\end{Claim}

On the other hand, if $(A, \mu \colon A\otimes A\rightarrow A)$ is an associative algebra and $\alpha , \beta \colon A\rightarrow A$ are commuting algebra endomorphisms, then one can easily check that $ ( A,\mu
\circ  ( \alpha \otimes \beta  ) ,\alpha ,\beta  ) $ is a~BiHom-associative algebra, denoted by $A_{(\alpha ,\beta )}$ and called the
\textit{Yau twist} of $(A, \mu )$.

In view of Claim~\ref{cl:YAU}, a BiHom-associative algebra with bijective structure maps is a Yau twist of an associative algebra.

The Yau twisting procedure for BiHom-associative algebras admits a more
general form, which we state in the next result (the proof is
straightforward and left to the reader).

\begin{Proposition}
\label{yaugeneral} Let $(D, \mu , \tilde{\alpha }, \tilde{\beta })$ be a
BiHom-associative algebra and $\alpha , \beta \colon D\rightarrow D$ two
multiplicative linear maps such that any two of the maps $\tilde{\alpha }$,
$\tilde{\beta }$, $\alpha$, $\beta $ commute. Then $(D, \mu \circ (\alpha
\otimes \beta )$, $\tilde{\alpha }\circ \alpha , \tilde{\beta }\circ \beta )$
is also a BiHom-associative algebra, denoted by $D_{(\alpha , \beta )}$.
\end{Proposition}

\begin{Example}\label{giacomoex}
We present an example of a BiHom-associative algebra that cannot be
expressed as a Hom-associative algebra. Let $\Bbbk $ be a f\/ield and $A=\Bbbk
 [ X ] $. Let $\alpha \colon A\rightarrow A$ be the algebra map def\/ined
by setting $\alpha  ( X ) =X^{2}$ and let $\beta =\operatorname{Id}
_{\Bbbk  [ X ] }$. Then we can consider the BiHom-associative
algebra $A_{(\alpha ,\beta )}= ( A,\mu \circ  ( \alpha \otimes \beta  ) ,\alpha
,\beta  ) $, where $\mu \colon A\otimes A\rightarrow A$ is the usual
multiplication. For every $a,a^{\prime }\in A$ set
\begin{gather*}
a\ast a^{\prime }=\mu \circ  ( \alpha \otimes \beta  )  (
a\otimes a^{\prime } ) =\alpha (a) a^{\prime }.
\end{gather*}
Let us assume that there exists $\theta \in \mathrm{End} ( \Bbbk  [ X
 ]  ) $ such that $ ( A,\mu \circ  ( \alpha \otimes \beta
 ) ,\theta  ) $ is a Hom-associative algebra. Then we should have
that
\begin{gather}
\theta (X) \ast  ( X\ast X ) = ( X\ast X )
\ast \theta (X) .  \label{form: gamma}
\end{gather}
Write
\begin{gather*}
\theta (X) =\sum_{i=0}^{n}a_{i}X^{i}, \qquad \text{where $a_{i}\in
\Bbbk$  for every $i=0, 1,\ldots ,n$ and $a_{n}\neq 0$}.
\end{gather*}
Since
\begin{gather*}
X\ast X=\alpha (X) X=X^{3},
\end{gather*}
(\ref{form: gamma}) rewrites as
\begin{gather*}
\sum_{i=0}^{n}a_{i}X^{i}\ast X^{3}=X^{3}\ast \sum_{i=0}^{n}a_{i}X^{i},
\end{gather*}
and hence as
\begin{gather*}
\sum_{i=0}^{n}a_{i}\alpha (X) ^{i}X^{3}=\alpha (X)
^{3}\sum_{i=0}^{n}a_{i}X^{i},
\end{gather*}
i.e.,
\begin{gather*}
\sum_{i=0}^{n}a_{i}X^{2i+3}=\sum_{i=0}^{n}a_{i}X^{6+i},
\end{gather*}
which implies that
\begin{gather*}
2n+3=6+n, \qquad \text{i.e.}, \quad n=3, \quad \text{and hence}
\\
a_{0}X^{3}+a_{1}X^{5}+a_{2}X^{7}+a_{3}X^{9}=a_{0}X^{6}+a_{1}X^{7}+a_{2}X^{8}+a_{3}X^{9},
\end{gather*}
so that
\begin{gather*}
\theta (X) =a_{3}X^{3}.
\end{gather*}
Let us set $c=a_{3}$ and let us check the equality
\begin{gather*}
\theta \big( X^{2}\big) \ast  ( X\ast X ) =\big( X^{2}\ast
X\big) \ast \theta (X) .
\end{gather*}
The left-hand side is
\begin{gather*}
\theta \big( X^{2}\big) \ast  ( X\ast X ) =c^{2}X^{6}\ast
X^{3}=\alpha \big( c^{2}X^{6}\big) X^{3}=c^{2}X^{15}.
\end{gather*}
The right-hand side is
\begin{gather*}
\big( X^{2}\ast X\big) \ast \theta (X) =\big( \alpha \big(
X^{2}\big) X\big) \ast \theta (X) =X^{5}\ast \theta  (
X ) =cX^{10}X^{3}=cX^{13}.
\end{gather*}
Thus the equality does not hold.
\end{Example}

\begin{Remark}
\label{remtensprod} Given two algebras $ ( A,\mu _{A},1_{A} ) $ and
$ ( B,\mu _{B},1_{B} ) $ in a braided monoidal category $ (
\mathcal{C},\otimes ,\mathbf{1},a,l,r,c ) $, it is well known that $
A\otimes B$ becomes also an algebra in the category, with multiplication $
\mu _{A\otimes B}$ def\/ined by
\begin{gather*}
\mu _{A\otimes B}=( \mu _{A}\otimes \mu _{B})\circ
a_{A,A,B\otimes B}^{-1} \circ  ( A\otimes a_{A,B,B})\\
\hphantom{\mu _{A\otimes B}=}{}\circ (
A\otimes ( c_{B,A}\otimes B) ) \circ \big( A\otimes
a_{B,A,B}^{-1}\big) \circ a_{A,B,A\otimes B}.
\end{gather*}

In the case of our category $\mathcal{H}^{c,d,\nu }(\mathcal{G},{_{\Bbbk }
\mathcal{M}})$, we have, for every $x,y\in A$, $x^{\prime },y^{\prime }\in B$:
\begin{gather*}
 \mu _{A\otimes B}  (  ( x\otimes y ) \otimes  ( x^{\prime
}\otimes y^{\prime } )  )
 = ( ( \mu _{A}\otimes \mu _{B} ) \circ a_{A,A,B\otimes
B}^{-1}\circ  ( A\otimes a_{A,B,B} )   \\
\qquad\quad{}
 \circ  ( A\otimes  ( c_{B,A}\otimes B )  ) \circ \big(
A\otimes a_{B,A,B}^{-1}\big) \circ a_{A,B,A\otimes B}) (  (
x\otimes y ) \otimes  ( x^{\prime }\otimes y^{\prime } )
 ) \\
\qquad {}
 = \big( ( \mu _{A}\otimes \mu _{B} ) \circ a_{A,A,B\otimes
B}^{-1}\circ  ( A\otimes a_{A,B,B} ) \circ  ( A\otimes  (
c_{B,A}\otimes B )  ) \\
\qquad\quad{}
 \circ \big( A\otimes a_{B,A,B}^{-1}\big)\big)  ( cx\otimes  (
y\otimes  ( dx^{\prime }\otimes dy^{\prime } )  )  ) \\
\qquad{}
=\big(  ( \mu _{A}\otimes \mu _{B} ) \circ a_{A,A,B\otimes
B}^{-1}\circ  ( A\otimes a_{A,B,B} )  \big) \big( cx\otimes
\big( \big( c^{-1}x^{\prime }\otimes dy\big) \otimes y^{\prime }\big)
\big) \\
\qquad{}
 = \big(  ( \mu _{A}\otimes \mu _{B} ) \circ a_{A,A,B\otimes
B}^{-1}\big)  ( cx\otimes  ( x^{\prime }\otimes  ( dy\otimes
dy^{\prime } )  )  ) \\
\qquad {}
 =  ( \mu _{A}\otimes \mu _{B} )  (  ( x\otimes x^{\prime
} ) \otimes  ( y\otimes y^{\prime } )  ) = ( x\cdot
x^{\prime } ) \otimes  ( y\cdot y^{\prime } ) .
\end{gather*}

In particular, if $(A,\alpha _{A},\beta _{A})$ and $(B,\alpha _{B},\beta
_{B})$ are two algebras in $\mathcal{H}^{(1,0) ,(0,1) ,1}(
\mathbb{Z}
\times
\mathbb{Z}
,{_{\Bbbk }\mathcal{M}})$, their braided tensor product $A
\underline{\otimes }B$ in the category is the algebra $(A\otimes B,\alpha
_{A}\otimes \alpha _{B},\beta _{A}\otimes \beta _{B})$, whose multiplication
is given by $(a\otimes b)(a^{\prime }\otimes b^{\prime })=aa^{\prime
}\otimes bb^{\prime }$, for all $a,a^{\prime }\in A$ and $b,b^{\prime }\in B$.
\end{Remark}

\begin{Remark}
If $(A, \mu _A, \alpha _A, \beta _A)$ and $(B, \mu _B, \alpha _B, \beta _B)$
are two BiHom-associative algebras over a f\/ield~$\Bbbk$, then $(A\otimes B,
\mu _{A\otimes B}, \alpha _A\otimes \alpha _B, \beta _A\otimes \beta _B)$ is
a BiHom-associative algebra (called the tensor product of $A$ and $B$),
where $\mu _{A\otimes B}$ is the usual multiplication: $(a\otimes
b)(a^{\prime }\otimes b^{\prime })=aa^{\prime }\otimes bb^{\prime }$. If $A$
and $B$ are unital with units $1_A$ and respectively $1_B$ then $A\otimes B$
is also unital with unit $1_A\otimes 1_B$. This is consistent with Remark~\ref{remtensprod}.
\end{Remark}

\begin{Example}
\label{Ex:lie}In view of Def\/inition~\ref{def:Lie}, a \textit{Lie algebra}
in $\mathcal{H}^{c,d,\nu }(\mathcal{G},{_{\Bbbk }\mathcal{M}})$ is a pair $ ( (L,f_{L}),[-]  ) $, where

\begin{itemize}\itemsep=0pt
\item[1)] $(L,f_{L})\in \Bbbk  [ \mathcal{G} ]$-$\text{Mod}$;

\item[2)] $[-] \colon L\otimes L\rightarrow L$ is a morphism in $\Bbbk
[ \mathcal{G}]$-$\text{Mod}$;

\item[3)] $[-] =-[-] \circ $ $\overline{\gamma }_{L,L}$;

\item[4)]
\begin{gather*}
[-] \circ ( L\otimes [-] ) +
[-] \circ ( L\otimes [-] ) \circ (
L\otimes \overline{\gamma }_{L,L}) \overline{a}_{L,L,L}(
\overline{\gamma }_{L,L}\otimes L) \overline{a}_{L,L,L}^{-1}\\
\qquad{} +[-] \circ  ( L\otimes [-]  ) \overline{a}
_{L,L,L}( \overline{\gamma }_{L,L}\otimes L) \overline{a}
_{L,L,L}^{-1}( L\otimes \overline{\gamma }_{L,L}) =0,
\end{gather*}
where $\overline{\gamma }_{L,L}=\tau \circ  ( f_{L}(cd)\otimes
f_{L}(c^{-1}d^{-1}) ) $ and $\tau $ is the usual f\/lip.
\end{itemize}

We will write down 4) explicitly. We have
\begin{gather*}
\big( \big( L\otimes \overline{\gamma }_{L,L}\big) \overline{a}
_{L,L,L}\big( \overline{\gamma }_{L,L}\otimes L\big) \overline{a}
_{L,L,L}^{-1}\big) ( x\otimes  ( y\otimes z )  )   \\
\qquad{}
=\big( L\otimes \overline{\gamma }_{L,L}\big) \overline{a}_{L,L,L}\big(
\overline{\gamma }_{L,L}\otimes L\big) \big( \big( c^{-1}x\otimes
y\big) \otimes d^{-1}z\big) \\
\qquad{}
=\big( L\otimes \overline{\gamma }_{L,L}\big) \overline{a}_{L,L,L}\big(
\big( c^{-1}d^{-1}y\otimes cdc^{-1}x\big) \otimes d^{-1}z\big) \\
\qquad{}=\big( L\otimes \overline{\gamma }_{L,L}\big) \big(
cc^{-1}d^{-1}y\otimes \big( cdc^{-1}x\otimes dd^{-1}z\big) \big) \\
\qquad{}
=cc^{-1}d^{-1}y\otimes \big( c^{-1}d^{-1}dd^{-1}z\otimes cdcdc^{-1}x\big)
\\
\qquad{}
=d^{-1}y\otimes \big( c^{-1}d^{-1}z\otimes dcdx\big),
\end{gather*}
therefore
\begin{gather*}
[-] \circ  ( L\otimes [-]  ) \big( \big(
L\otimes \overline{\gamma }_{L,L}\big) \overline{a}_{L,L,L}\big(
\overline{\gamma }_{L,L}\otimes L\big) \overline{a}_{L,L,L}^{-1}\big)
 ( x\otimes ( y\otimes z )  ) \\
 \qquad{}
=\big[ d^{-1}y,\big[ c^{-1}d^{-1}z,cd^{2}x\big] \big],
\end{gather*}
and
\begin{gather*}
\big( \overline{a}_{L,L,L}\big( \overline{\gamma }_{L,L}\otimes L\big)
\overline{a}_{L,L,L}^{-1}\big( L\otimes \overline{\gamma }_{L,L}\big)
\big)  ( x\otimes  ( y\otimes z )  ) \\
\qquad{} =\overline{a}_{L,L,L}\big( \overline{\gamma }_{L,L}\otimes L\big)
\overline{a}_{L,L,L}^{-1}\big( x\otimes \big( c^{-1}d^{-1}z\otimes
cdy\big) \big) \\
\qquad{}
=\overline{a}_{L,L,L}\big( \overline{\gamma }_{L,L}\otimes L\big) \big(
c^{-1}x\otimes \big( c^{-1}d^{-1}z\otimes d^{-1}cdy\big) \big) \\
\qquad{}
=\overline{a}_{L,L,L}\big( \overline{\gamma }_{L,L}\otimes L\big) \big(
\big( c^{-1}x\otimes c^{-1}d^{-1}z\big) \otimes cy\big) \\
\qquad{}
=\overline{a}_{L,L,L}\big( \big( c^{-2}d^{-2}z\otimes cdc^{-1}x\big)
\otimes cy\big) \\
\qquad{}
=\big( \big( c^{-1}d^{-2}z\otimes dx\big) \otimes cdy\big),
\end{gather*}
hence
\begin{gather*}
[-] \circ  ( L\otimes [-]  ) \big(
\overline{a}_{L,L,L}\big( \overline{\gamma }_{L,L}\otimes L\big)
\overline{a}_{L,L,L}^{-1}\big( L\otimes \overline{\gamma }_{L,L}\big)
\big)  ( x\otimes ( y\otimes z )  )
=\big[ c^{-1}d^{-2}z, [ dx,cdy ] \big].
\end{gather*}
Thus 4) is equivalent to
\begin{gather*}
 [ x,[y,z]  ] +\big[ d^{-1}y,\big[
c^{-1}d^{-1}z,cd^{2}x\big] \big] +\big[ c^{-1}d^{-2}z, [ dx,cdy] \big] =0, \qquad \text{for every} \ \ x,y,z\in L,
\end{gather*}
which is equivalent to
\begin{gather*}
\big[ d^{-2}x,\big[ d^{-1}y,cz\big] \big] +\big[ d^{-2}y,\big[
d^{-1}z,cx\big] \big] +\big[ d^{-2}z,\big[ d^{-1}x,cy\big] \big]
=0,\qquad \text{for every} \ \ x,y,z\in L.
\end{gather*}

Thus, a Lie algebra\ in $\mathcal{H}^{c,d,\nu }(\mathcal{G},{_{\Bbbk }\mathcal{M}})$ is a pair $( L,[-]) $, where
\begin{itemize}\itemsep=0pt
\item[1)] $L\in \Bbbk  [ \mathcal{G} ]$-${\rm Mod}$;

\item[2)] $g[x,y] =[gx,gy] $, for every $x,y\in L$;

\item[3)] $[x,y] =- [ c^{-1}d^{-1}y,cdx ] $, for every $
x,y\in L$, i.e., $[ x,cdy] =-[ y,cdx] $, for every $x,y\in L$ (skew-symmetry);

\item[4)] $[ d^{-2}x,[ d^{-1}y,cz] ] +[ d^{-2}y,[ d^{-1}z,cx] ] +[ d^{-2}z,[ d^{-1}x,cy] ] =0$, for every $x,y,z\in L$ (Jacobi condition).
\end{itemize}

In particular, a Lie algebra in $\mathcal{H}^{(1,0) ,(0,1) ,1}(
\mathbb{Z}
\times
\mathbb{Z}
,{_{\Bbbk }\mathcal{M}})$ is a pair $(( L,\alpha ,\beta) ,[-] ) $, where
\begin{itemize}\itemsep=0pt
\item[1)] $\alpha ,\beta \in \operatorname{Aut}_{\Bbbk }(L) $ and $\alpha \circ \beta =\beta \circ \alpha$;

\item[2)] $[-] \colon ( L\otimes L,\alpha \otimes \alpha ,\beta
\otimes \beta ) \rightarrow ( L,\alpha ,\beta ) $ is a
morphism in $\Bbbk [
\mathbb{Z}
\times
\mathbb{Z}
]$-${\rm Mod}$, i.e., $\alpha  [ a,b ] = [ \alpha (a)
,\alpha (b)  ] $ and $\beta [a,b] = [ \beta
(a) ,\beta (b)  ] $, for every $a,b\in L$;

\item[3)] $[ a,\alpha \beta ^{-1}(b) ] =-[
b,\alpha \beta ^{-1}(a) ] $, for every $a, b\in L$, which
is equivalent to $[ \beta (a) ,\alpha (b)
] =-[ \beta (b) ,\alpha (a)] $,
for every $a, b\in L$;

\item[4)] $[ \beta ^{2}x,[ \beta y,\alpha z] ] +[
\beta ^{2}y,[ \beta z,\alpha x] ] +[ \beta ^{2}z,[
\beta x,\alpha y] ] =0$, for every $x,y,z\in L$.
\end{itemize}
\end{Example}

Inspired by Example \ref{Ex:lie}, we introduce the following concept.

\begin{Definition}
A \textit{BiHom-Lie algebra} over a f\/ield $\Bbbk $ is a 4-tuple $( L,[-] ,\alpha ,\beta ) $, where $L$ is a $\Bbbk $-linear
space, $\alpha \colon L\rightarrow L$, $\beta \colon L\rightarrow L$ and $[-]
\colon L\otimes L\rightarrow L$ are linear maps, with notation $[-]
( a\otimes a^{\prime }) =[ a,a^{\prime }] $,
satisfying the following conditions, for all $a,a^{\prime },a^{\prime \prime
}\in L$:
\begin{gather*}
\alpha \circ \beta =\beta \circ \alpha , \\
\alpha ([ a^{\prime },a^{\prime \prime }])=[ \alpha (
a^{\prime }) ,\alpha ( a^{\prime \prime }) ]\qquad
\text{and} \qquad \beta ([ a^{\prime },a^{\prime \prime }])=[
\beta ( a^{\prime }) ,\beta ( a^{\prime \prime }) ], \\
 [ \beta (a) ,\alpha  ( a^{\prime } )  ] =-
[ \beta ( a^{\prime }) ,\alpha (a) ]
\qquad \text{(skew-symmetry)}, \\
\big[ \beta ^{2}(a) , [ \beta  ( a^{\prime } )
,\alpha  ( a^{\prime \prime } )  ] \big] +\big[ \beta
^{2} ( a^{\prime } ) , [ \beta  ( a^{\prime \prime } )
,\alpha (a)  ] \big] +\big[ \beta ^{2} ( a^{\prime
\prime } ) , [ \beta (a) ,\alpha  ( a^{\prime
} ) ] \big] =0  \\
\qquad{} \text{(BiHom-Jacobi condition)}.
\end{gather*}

We call $\alpha $ and $\beta $ (in this order) the \textit{structure maps}
of~$L$.
A~morphism $f\colon  ( L,[-] ,\alpha ,\beta  )\rightarrow
 ( L^{\prime },[-]^{\prime },\alpha ^{\prime },\beta
^{\prime } )$ of BiHom-Lie algebras is a linear map $f\colon L\rightarrow
L^{\prime }$ such that $\alpha ^{\prime }\circ f=f\circ \alpha $, $\beta
^{\prime }\circ f=f\circ \beta $ and $f([x,y])=[f(x),
f(y)]^{\prime }$, for all $x, y\in L$.
\end{Definition}

Thus,  a Lie algebra in $\mathcal{H}^{(1,0) ,(0,1) ,1}(
\mathbb{Z}
\times
\mathbb{Z}
,{_{\Bbbk }\mathcal{M}})$ is exactly a BiHom-Lie algebra with \textit{bijective} structure maps.

\begin{Remark}
Obviously, a Hom-Lie algebra $( L,[-] ,\alpha ) $ is
a particular case of a BiHom-Lie algebra, namely $( L, [-]
,\alpha ,\alpha ) $. Conversely, a BiHom-Lie algebra $( L, [
-] ,\alpha ,\alpha ) $ with bijective $\alpha $ is the Hom-Lie
algebra $( L,[-] ,\alpha) $.
\end{Remark}

 In view of Claim~\ref{Cl:LIE}, we have:
\begin{Proposition}
\label{croset} If $ ( A,\mu ,\alpha ,\beta
 ) $ is a BiHom-associative algebra with bijective $\alpha $ and $\beta
$, then, for every $a,a^{\prime }\in A$, we can set
\begin{gather*}
[ a,a^{\prime }] =aa^{\prime }-\big(\alpha ^{-1}\beta  (
a^{\prime } )\big)\big(\alpha \beta ^{-1}(a)\big) .
\end{gather*}
Then $ ( A,[-] ,\alpha ,\beta  ) $ is a BiHom-Lie
algebra, denoted by~$L(A)$.
\end{Proposition}

The proofs of the following three results are straightforward and left to
the reader.

\begin{Proposition}
Let $ ( L,[-]  ) $ be an ordinary Lie algebra over a
field $\Bbbk $ and let $\alpha ,\beta \colon L\rightarrow L$ two commuting linear
maps such that $\alpha  (  [ a,a^{\prime } ]  ) = [
\alpha (a) ,\alpha  ( a^{\prime } )  ]$ and $\beta
 (  [ a,a^{\prime } ]  ) = [ \beta (a)
,\beta  ( a^{\prime } )  ]$, for all $a,a^{\prime }\in L$.
Define the linear map $ \{ - \} \colon L\otimes L\rightarrow L$,
\begin{gather*}
 \{ a,b \} = [ \alpha (a) ,\beta (b) ] ,\qquad \text{for all} \ \ a,b\in L.
\end{gather*}
Then $L_{( \alpha ,\beta ) }:=(L, \{-\}, \alpha ,
\beta )$ is a BiHom-Lie algebra, called the \textit{Yau twist} of $( L,[-]) $.
\end{Proposition}

\begin{Claim}
More generally, let $( L,[-] ,\alpha ,\beta ) $ be a
BiHom-Lie algebra and $\alpha^{\prime }, \beta^{\prime }\colon L\rightarrow L$
linear maps such that $\alpha ^{\prime }([ a, b])=[ \alpha
^{\prime }(a) ,\alpha ^{\prime }(b) ]$ and $\beta ^{\prime }([ a, b])=[ \beta ^{\prime }(a)
,\beta ^{\prime }(b) ]$ for all $a, b\in L$, and any two
of the maps $\alpha, \beta , \alpha ^{\prime }, \beta ^{\prime }$ commute.
Then $( L,[-]_{(\alpha^{\prime },\beta^{\prime }) }:=[
-]\circ(\alpha^{\prime }\otimes\beta^{\prime }),
\alpha\circ\alpha^{\prime },\beta\circ\beta^{\prime }) $ is a~BiHom-Lie algebra.
\end{Claim}

\begin{Proposition}
Let $( A,\mu ) $ be an associative algebra and $\alpha ,\beta
\colon A\rightarrow A$ two commuting algebra isomorphisms. Then $L( A_{(
\alpha ,\beta ) }) =L( A) _{( \alpha ,\beta
) }$, as BiHom-Lie algebras.
\end{Proposition}

\begin{Remark}
Let $\mathcal{G}$ be a group and $c,d\in Z(\mathcal{G}) $, $\nu \in \operatorname{Aut}_{\mathcal{C}} ( \mathbf{1} ) $. It is
straightforward to prove that the category $\mathcal{H}^{c,d,\nu }(\mathcal{G},{_{\Bbbk }\mathcal{M}})$ fulf\/ills the assumption of \cite[Theorem~6.4]{AM}. Hence, for any Lie algebra $ ( L,[-] ) $ in $\mathcal{H}^{c,d,\nu }(\mathcal{G},{_{\Bbbk }\mathcal{M}})$, we can consider the universal enveloping bialgebra $\overline{\mathcal{U}}( ( L,[-] ) ) $ as introduced
in~\cite{AM}. By \cite[Remark~6.5]{AM}, $\overline{\mathcal{U}}(  ( L,[-]  )  ) $ as a bialgebra is a quotient of the tensor
bialgebra $\overline{T}L$. The morphism giving the projection is induced by
the canonical projection $p\colon TL\rightarrow \mathcal{U} ( L,[-]
 ) $ def\/ining the universal enveloping algebra. At algebra level we have
\begin{gather*}
\mathcal{U} ( L,[-]  )  = \frac{TL}{ \big( [ x,y] -x\otimes y+\overline{\gamma }_{L,L}( x\otimes y) \,|\, x,y\in
L\big) } \\
\hphantom{\mathcal{U} ( L,[-]  ) }{}
 = \frac{TL}{\big( [x,y] -x\otimes y+\big(
f_{L}(c^{-1}d^{-1}) ( y ) \big) \otimes f_{L}(cd)(x)
\,|\, x,y\in L\big) }.
\end{gather*}
By Theorem \ref{teo:htuttiiso}, the identity functor $\mathcal{I}\colon \mathcal{H}^{c,d,\nu }(\mathcal{G},{_{\Bbbk }\mathcal{M}})\rightarrow
\mathcal{H}(\mathcal{G},{_{\Bbbk }\mathcal{M}})$ is a braided monoidal
isomorphism. Let $F\colon \mathcal{H}(\mathcal{G},{_{\Bbbk }\mathcal{M}})
\rightarrow  {_{\Bbbk }\mathcal{M}}$ be the forgetful functor. Then $
F\circ \mathcal{I}$ is a monoidal functor $\mathcal{H}^{c,d,\nu }(\mathcal{G}
,{_{\Bbbk }\mathcal{M}})\rightarrow {_{\Bbbk }\mathcal{M}}$ to
which we can apply \cite[Theorem~8.5]{AM} to get that $\mathcal{H}^{c,d,\nu }(\mathcal{G},{_{\Bbbk }\mathcal{M}})$ is what is called in~\cite{AM}
a Milnor--Moore category. This
implies that, by \cite[Theorem~7.2]{AM}, we have an isomorphism $ ( L,[-] ) \rightarrow \mathcal{P}\overline{\mathcal{U}}(
( L,[-] )) $, where $\mathcal{P}\overline{\mathcal{U}}( ( L,[-] ) ) $ denotes the
primitive part of $\overline{\mathcal{U}}( ( L,[-]
) ) $. That is, half of the Milnor--Moore theorem holds.

The case $\mathcal{G}=\mathbb{Z}$ can be found in \cite[Remark~9.10]{AM}.

In the particular case of a Lie algebra $ (  ( L,\alpha ,\beta
 ) ,[-]  ) $ in $\mathcal{H}^{(1,0)
,(0,1) ,1}(\mathbb{Z}\times \mathbb{Z},{_{\Bbbk }
\mathcal{M}})$ we have that
\begin{gather*}
\mathcal{U} ( L,[-]  ) =\frac{TL}{ \big(  [ x,y] -x\otimes y+ ( \alpha ^{-1}\beta ) ( y )
\otimes  ( \alpha \beta ^{-1} ) (x) \,|\, x,y\in L\big) }.
\end{gather*}

Enveloping algebras of Hom-Lie algebras where introduced in \cite{Y} (see
also \cite[Section~8]{stef}).
\end{Remark}

\section{Representations}

From now on, we will always work over a base f\/ield $\Bbbk$. All
algebras, linear spaces etc. will be over $\Bbbk $; unadorned $\otimes $
means~$\otimes_{\Bbbk}$. For a comultiplication $\Delta \colon C\rightarrow
C\otimes C$ on a linear space~$C$, we use a Sweedler-type notation $\Delta
(c)=c_1\otimes c_2$, for $c\in C$. Unless otherwise specif\/ied, the
(co)algebras ((co)associative or not) that will appear in what follows are
\emph{not} supposed to be (co)unital, and a multiplication $\mu \colon V\otimes
V\rightarrow V$ on a linear space~$V$ is denoted by juxtaposition: $\mu
(v\otimes v^{\prime })=vv^{\prime }$. For the composition of two maps~$f$
and~$g$, we will write either $g\circ f$ or simply~$gf$. For the identity
map on a linear space~$V$ we will use the notation~$\operatorname{id}_V$.

\begin{Definition}
Let $(A, \mu _A , \alpha _A, \beta _A)$ be a BiHom-associative algebra. A
\textit{left $A$-module} is a triple $(M, \alpha _M, \beta _M)$, where $M$
is a linear space, $\alpha _M, \beta _M\colon M \rightarrow M$ are linear maps and
we have a linear map $A\otimes M\rightarrow M$, $a\otimes m\mapsto a\cdot m$, such that, for all $a, a^{\prime }\in A$, $m\in M$, we have
\begin{gather}
 \alpha _M\circ \beta _M=\beta _M\circ \alpha _M, \\
 \alpha _M(a\cdot m)=\alpha _A(a)\cdot \alpha _M(m),  \label{ghommod11} \\
 \beta _M(a\cdot m)=\beta _A(a)\cdot \beta _M(m),  \label{ghommod12} \\
 \alpha _A(a)\cdot (a^{\prime }\cdot m)=(aa^{\prime })\cdot \beta _M(m).
\label{ghommod2}
\end{gather}

If $(M, \alpha _M, \beta _M)$ and $(N, \alpha _N, \beta _N)$ are left $A$-modules (both $A$-actions denoted by $\cdot$), a morphism of left $A$-modules $f\colon M\rightarrow N$ is a linear map satisfying the conditions $\alpha _N\circ f=f\circ \alpha _M$, $\beta _N\circ f=f\circ \beta _M$ and $f(a\cdot m)=a\cdot f(m)$, for all $a\in A$ and $m\in M$.

If $(A, \mu _A, \alpha _A, \beta _A, 1_A)$ is a unital BiHom-associative
algebra and $(M, \alpha _M, \beta _M)$ is a left $A$-module, them $M$ is
called \textit{unital} if $1_A\cdot m=\beta _M(m)$, for all $m\in M$.
\end{Definition}

\begin{Remark}
If $(A, \mu , \alpha , \beta )$ is a BiHom-associative algebra, then $(A,
\alpha , \beta )$ is a left $A$-module with action def\/ined by $a\cdot b=ab$,
for all $a, b\in A$.
\end{Remark}

\begin{Lemma}\label{endo} Let $(E, \mu , 1_E)$ be an associative unital algebra and $u,
v\in E$ two invertible elements such that $uv=vu$. Define the linear maps $\tilde{\alpha }, \tilde{\beta }\colon E\rightarrow E$, $\tilde{\alpha }(a)=uau^{-1} $, $\tilde{\beta }(a)=vav^{-1}$, for all $a\in E$, and the
linear map $\tilde{\mu }\colon E\otimes E\rightarrow E$, $\tilde{\mu }(a\otimes
b):=a*b=uau^{-1}bv^{-1}$, for all $a, b\in E$. Then $(E, \tilde{\mu },
\tilde{\alpha }, \tilde{\beta })$ is a unital BiHom-associative algebra with
unit~$v$, denoted by $E(u, v)$.
\end{Lemma}

\begin{proof}
Obviously $\tilde{\alpha }\circ \tilde{\beta }=\tilde{\beta }\circ \tilde{\alpha }$ because $uv=vu$. Then, for all $a, b, c\in E$:
\begin{gather*}
\tilde{\alpha }(a)*\tilde{\alpha }(b) = \big(uau^{-1}\big)*\big(ubu^{-1}\big)
 = uuau^{-1}u^{-1}ubu^{-1}v^{-1} \\
\hphantom{\tilde{\alpha }(a)*\tilde{\alpha }(b)}{}
=uuau^{-1}bu^{-1}v^{-1} =\tilde{\alpha }\big(uau^{-1}bv^{-1}\big) =\tilde{\alpha }(a*b),
\\
\tilde{\beta }(a)*\tilde{\beta }(b) = \big(vav^{-1}\big)*\big(vbv^{-1}\big) = uvav^{-1}u^{-1}vbv^{-1}v^{-1} \\
\hphantom{\tilde{\beta }(a)*\tilde{\beta }(b)}{}
 = uvau^{-1}bv^{-1}v^{-1}  =\tilde{\beta }\big(uau^{-1}bv^{-1}\big) =\tilde{\beta }(a*b),
\\
\tilde{\alpha }(a)*(b*c) = \big(uau^{-1}\big)*\big(ubu^{-1}cv^{-1}\big)
=uuau^{-1}u^{-1}ubu^{-1}cv^{-1}v^{-1} \\
\hphantom{\tilde{\alpha }(a)*(b*c)}{}
=uuau^{-1}bv^{-1}u^{-1}vcv^{-1}v^{-1}
=\big(uau^{-1}bv^{-1}\big)*\big(vcv^{-1}\big) =(a*b)*\tilde{\beta }(c),
\end{gather*}
so $(E, \tilde{\mu }, \tilde{\alpha },
\tilde{\beta })$ is indeed a
BiHom-associative algebra. To prove that $v$ is the unit, we compute
\begin{gather*}
\tilde{\alpha }(v)=uvu^{-1}=v, \qquad
\tilde{\beta }(v)=vvv^{-1}=v, \\
a*v=uau^{-1}vv^{-1}=uau^{-1}=\tilde{\alpha }(a), \qquad
v*a=uvu^{-1}av^{-1}=vav^{-1}=\tilde{\beta }(a),
\end{gather*}
f\/inishing the proof.
\end{proof}

\begin{Proposition}
Let $(A, \mu _A, \alpha _A, \beta _A)$ be a BiHom-associative algebra, $M$ a
linear space and $\alpha _M, \beta _M\colon M\rightarrow M$ two commuting linear
isomorphisms. Consider the associative unital algebra $E=\operatorname{End}(M)$ with its
usual structure, denote $u:=\alpha _M$, $v:=\beta _M$, and construct the
BiHom-associative algebra $(E, \tilde{\mu }, \tilde{\alpha }, \tilde{\beta }
)= \operatorname{End}(M)(\alpha _M, \beta _M)$ as in Lemma~{\rm \ref{endo}}. Then setting a
structure of a left $A$-module on $(M, \alpha _M, \beta _M)$ is equivalent
to giving a morphism of BiHom-associative algebras $\varphi \colon (A, \mu _A,
\alpha _A, \beta _A)\rightarrow (E, \tilde{\mu }, \tilde{\alpha }, \tilde{\beta })$. If $A$ is moreover unital with unit $1_A$, then the module $(M,
\alpha _M, \beta _M)$ is unital if and only if the morphism $\varphi $ is
unital.
\end{Proposition}

\begin{proof}
The correspondence is given as follows: the module structure $A\otimes
M\rightarrow M$ is def\/ined by setting $a\otimes m\mapsto a\cdot m$ if and
only if $a\cdot m=\varphi (a)(m)$, for all $a\in A$, $m\in M$. It is easy to
see that conditions~(\ref{ghommod11}) and~(\ref{ghommod12}) are
equivalent to $\tilde{\alpha }\circ \varphi =\varphi \circ
\alpha _A$ and respectively $\tilde{\beta }\circ \varphi =\varphi \circ \beta _A$. We
prove that, assuming~(\ref{ghommod11}) and~(\ref{ghommod12}), we have that~(\ref{ghommod2}) is equivalent to $\varphi \circ \mu _A=\tilde{\mu }\circ
(\varphi \otimes \varphi )$. Note f\/irst that~(\ref{ghommod11}) may be
written as $\alpha _M\circ \varphi (a)=\varphi (\alpha _A(a))\circ \alpha _M$, for all $a\in A$, or equivalently $\alpha _M\circ \varphi (a)\circ \alpha
_M^{-1}=\varphi (\alpha _A(a))$, for all $a\in A$. Thus, for all $a, b\in A$, we have
\begin{gather*}
\tilde{\mu }\circ (\varphi \otimes \varphi )(a\otimes b)=\varphi
(a)*\varphi (b)
=\alpha _M\circ \varphi (a)\circ \alpha _M^{-1}\circ \varphi (b)\circ
\beta _M^{-1}
=\varphi (\alpha _A(a))\circ \varphi (b)\circ \beta _M^{-1}.
\end{gather*}
Hence, we have
\begin{gather*}
\varphi \circ \mu _A=\tilde{\mu }\circ (\varphi \otimes \varphi)
\Longleftrightarrow   \varphi (ab)=\varphi (a)*\varphi (b), \qquad \forall \, a, b\in A, \\
\hphantom{\varphi \circ \mu _A=\tilde{\mu }\circ (\varphi \otimes \varphi) }{}
 \Longleftrightarrow  \varphi (ab)(n)=(\varphi (a)*\varphi (b))(n),
\qquad \forall \, a, b\in A, n\in M, \\
\hphantom{\varphi \circ \mu _A=\tilde{\mu }\circ (\varphi \otimes \varphi) }{}
 \Longleftrightarrow (ab)\cdot n=(\varphi (\alpha _A(a))\circ \varphi
(b)\circ \beta _M^{-1})(n), \qquad \forall \, a, b\in A, n\in M, \\
\hphantom{\varphi \circ \mu _A=\tilde{\mu }\circ (\varphi \otimes \varphi) }{}
 \Longleftrightarrow (ab)\cdot \beta _M(m)=(\varphi (\alpha _A(a))\circ
\varphi (b))(m), \qquad \forall \, a, b\in A, m\in M, \\
\hphantom{\varphi \circ \mu _A=\tilde{\mu }\circ (\varphi \otimes \varphi) }{}
 \Longleftrightarrow   (ab)\cdot \beta _M(m)=\alpha _A(a)\cdot (b\cdot m),
\qquad \forall \, a, b\in A, m\in M,
\end{gather*}
which is exactly~(\ref{ghommod2}).

Assume that $A$ is unital with unit $1_A$. The fact that~$\varphi $ is
unital is equivalent to $\varphi (1_A)=\beta _M$, which is equivalent to $1_A\cdot m=\beta _M(m)$, for all $m\in M$, which is equivalent to saying
that the module~$M$ is unital.
\end{proof}

We recall the following concept from~\cite{sheng} (see also~\cite{said} on
this subject).

\begin{Definition}[\cite{sheng}] Let $(L, [-] ,\alpha )$ be a Hom-Lie algebra. A
\textit{representation} of $L$ is a triple $(M, \rho , A)$, where $M$ is a
linear space, $A\colon M\rightarrow M$ and $\rho \colon L\rightarrow \operatorname{End}(M)$ are linear
maps such that, for all $x, y\in L$, the following conditions are satisf\/ied:
\begin{gather*}
 \rho (\alpha (x))\circ A=A\circ \rho (x), \qquad
 \rho ([x,y])\circ A=\rho (\alpha (x))\circ \rho (y)-\rho
(\alpha (y))\circ \rho (x).
\end{gather*}
\end{Definition}

\begin{Remark}
Let $(L, [-] ,\alpha )$ be a Hom-Lie algebra, $M$ a linear
space, $A\colon M\rightarrow M$ and $\rho \colon L\rightarrow \operatorname{End}(M)$ linear maps such
that $A$ is bijective. We can consider the Hom-associative algebra $\operatorname{End}(M)(A, A)$ as in Lemma~\ref{endo}, and then the Hom-Lie algebra $L(\operatorname{End}(M)(A, A))$. Then one can check that $(M, \rho , A)$ is a
representation of $L$ if and only if $\rho $ is a morphism of Hom-Lie
algebras from $L$ to $L(\operatorname{End}(M)(A, A))$.
\end{Remark}

Inspired by this remark, we can introduce now the following concept:

\begin{Definition}
Let $(L, [-] ,\alpha , \beta )$ be a BiHom-Lie algebra. A~\textit{representation} of $L$ is a 4-tuple $(M, \rho , \alpha _M, \beta _M
) $, where~$M$ is a~linear space, $\alpha _M, \beta _M\colon M\rightarrow M$ are
two commuting linear maps and $\rho \colon L\rightarrow \operatorname{End}(M)$ is a linear map
such that, for all $x, y\in L$, we have
\begin{gather}
 \rho (\alpha (x))\circ \alpha _M=\alpha _M\circ \rho (x),  \label{lierep1}
\\
 \rho (\beta (x))\circ \beta _M=\beta _M\circ \rho (x),  \label{lierep2} \\
 \rho ( [\beta (x), y ])\circ \beta _M=\rho (\alpha \beta
(x))\circ \rho (y)- \rho (\beta (y))\circ \rho (\alpha (x)).  \label{lierep3}
\end{gather}
\end{Definition}

A f\/irst indication that this is indeed the appropriate concept of
representation for BiHom-Lie algebras is provided by the following result
(extending the corresponding one for Hom-associative algebras in~\cite{bakayoko}), whose proof is straightforward and left to the reader.

\begin{Proposition}
Let $(A, \mu _A, \alpha _A, \beta _A)$ be a BiHom-associative algebra with
bijective structure maps and $(M, \alpha _M, \beta _M)$ a left $A$-module,
with action $A\otimes M\rightarrow M$, $a\otimes m\mapsto a\cdot m$. Then we
have a~representation $(M, \rho , \alpha _M, \beta _M)$ of the BiHom-Lie
algebra~$L(A)$, where $\rho \colon L(A)\rightarrow \operatorname{End}(M)$ is the linear map
defined by $\rho (a)(m)=a\cdot m$, for all $a\in A$, $m\in M$.
\end{Proposition}

A second indication is provided by the fact that, under certain
circumstances, we can construct the semidirect product (the Hom-case is done
in~\cite{sheng}).

\begin{Proposition}
Let $(L, [-] ,\alpha , \beta )$ be a BiHom-Lie algebra and $(M,
\rho , \alpha _M, \beta _M )$ a representation of $L$, with notation $\rho
(x)(a)=x\cdot a$, for all $x\in L$, $a\in M$. Assume that the maps $\alpha $
and $\beta _M$ are bijective. Then $L\ltimes M:=(L\oplus M, [-],
\alpha \oplus \alpha _M, \beta \oplus \beta _M)$ is a BiHom-Lie algebra
$($called the semidirect product$)$, where $\alpha \oplus \alpha _M,
\beta \oplus \beta _M\colon  L\oplus M \rightarrow L\oplus M$ are defined by $(\alpha \oplus \alpha _M)(x, a)=(\alpha (x), \alpha _M(a))$ and $(\beta
\oplus \beta _M)(x, a)=(\beta (x), \beta _M(a))$, and, for all $x, y\in L$
and $a, b\in M$, the bracket $[-]$ is defined by
\begin{gather*}
  [ (x, a), (y, b) ]=\big( [ x, y ], x\cdot b-\alpha
^{-1}\beta (y)\cdot\alpha _M\beta _M^{-1}(a)\big).
\end{gather*}
\end{Proposition}

\begin{proof}
Follows by a direct computation that is left to the reader.
\end{proof}

\begin{Proposition}
Let $(L, [-] ,\alpha , \beta )$ be a BiHom-Lie algebra such that
the map $\beta $ is surjective, $M$ a linear space, $\alpha _M, \beta
_M\colon M\rightarrow M$ two commuting linear isomorphisms and $\rho \colon L\rightarrow
\operatorname{End}(M)$ a~linear map. Then $(M, \rho , \alpha _M, \beta _M )$ is a~representation of $L$ if and only if $\rho $ is a morphism of BiHom-Lie
algebras from~$L$ to $L(\operatorname{End}(M)(\alpha _M, \beta _M))$.
\end{Proposition}

\begin{proof}
Obviously, (\ref{lierep1}) and (\ref{lierep2}) are respectively equivalent
to $\tilde{\alpha }\circ \rho =\rho \circ \alpha $ and $\tilde{\beta }\circ
\rho =\rho \circ \beta $, so we only need to prove that, assuming~(\ref{lierep1}) and~(\ref{lierep2}),~(\ref{lierep3}) is equivalent to $\rho ([x,y])= [\rho (x), \rho (y) ]$ for all $x, y\in L$.
First we write down explicitly the bracket of $L(\operatorname{End}(M)(\alpha _M, \beta
_M))$. In view of Proposition~\ref{croset}, this bracket looks as follows, for $f, g\in \operatorname{End}(M)$:
\begin{gather*}
 [f, g ] = f*g-\big(\tilde{\alpha }^{-1}\tilde{\beta }(g)\big)*\big(\tilde{\alpha }\tilde{\beta }^{-1}(f)\big) \\
\hphantom{[f, g]}{}
=f*g-\big(\tilde{\alpha }^{-1}\big(\beta _M\circ g\circ \beta _M^{-1}\big)\big)* \big(\tilde{\alpha }\big(\beta _M^{-1}\circ f\circ \beta _M\big)\big) \\
\hphantom{[f, g]}{}
= f*g-\big(\alpha _M^{-1}\circ \beta _M\circ g\circ \beta _M^{-1}\circ \alpha
_M\big)* \big(\alpha _M\circ \beta _M^{-1}\circ f\circ \beta _M\circ \alpha _M^{-1}\big)
\\
\hphantom{[f, g]}{}
=\alpha _M\circ f\circ \alpha _M^{-1}\circ g\circ \beta _M^{-1} \\
\hphantom{[f, g]=}{}- \alpha _M\circ \alpha _M^{-1}\circ \beta _M\circ g\circ \beta
_M^{-1}\circ \alpha _M\circ \alpha _M^{-1}\circ \alpha _M\circ \beta
_M^{-1}\circ f\circ \beta _M\circ \alpha _M^{-1}\circ \beta _M^{-1} \\
\hphantom{[f, g]}{}
= \alpha _M\circ f\circ \alpha _M^{-1}\circ g\circ \beta _M^{-1}- \beta
_M\circ g\circ \beta _M^{-1}\circ \alpha _M\circ \beta _M^{-1}\circ f\circ
\beta _M\circ \alpha _M^{-1}\circ \beta _M^{-1}.
\end{gather*}
Let $x, y\in L$; we take $f=\rho (\beta (x))$, $g=\rho (y)$. We obtain
\begin{gather*}
 [\rho (\beta (x)), \rho (y) ]\circ \beta _M =
 \alpha _M\circ \rho
(\beta (x))\circ \alpha _M^{-1} \circ \rho (y) \\
\hphantom{[\rho (\beta (x)), \rho (y) ]\circ \beta _M =}{}
 - \beta _M\circ \rho (y)\circ \beta _M^{-1}\circ \alpha _M\circ \beta
_M^{-1}\circ \rho (\beta (x))\circ \beta _M\circ \alpha _M^{-1} \\
\hphantom{[\rho (\beta (x)), \rho (y) ]\circ \beta _M}{}
\overset{\eqref{lierep1}, \;\eqref{lierep2}}{=} \rho (\alpha \beta
(x))\circ \rho (y)- \rho (\beta (y))\circ \alpha _M\circ \rho (x)\circ
\alpha _M^{-1} \\
\hphantom{[\rho (\beta (x)), \rho (y) ]\circ \beta _M}{}
\overset{\eqref{lierep1}}{=} \rho (\alpha \beta (x))\circ \rho (y)- \rho
(\beta (y))\circ \rho (\alpha (x)),
\end{gather*}
which is the right-hand side of (\ref{lierep3}). So, we have that~(\ref{lierep3}) holds if and only if $\rho ([\beta (x), y])= [\rho (\beta (x)), \rho (y)]$ for all $x, y\in L$, which is equivalent
to $\rho ([a,b])= [\rho (a), \rho (b)]$, for all $a,
b\in L$, because $\beta $ is surjective.
\end{proof}

\begin{Proposition}
Let $(L, [-] ,\alpha , \beta )$ be a BiHom-Lie algebra and
define the linear map $\operatorname{ad} \colon L\rightarrow \operatorname{End}(L)$, $\operatorname{ad} (x)(y)=[x,y]
$, for all $x, y\in L$. If the maps $\alpha $ and $\beta $ are bijective,
then $(L, \operatorname{ad},\alpha , \beta )$ is a representation of $L$.
\end{Proposition}

\begin{proof}
The conditions (\ref{lierep1}) and (\ref{lierep2}) are equivalent to $\alpha
([a,b])=  [\alpha (a), \alpha (b) ]$ and $\beta ( [a, b])= [\beta (a), \beta (b)]$ for all $a, b\in L$, so we
only need to prove~(\ref{lierep3}). Note f\/irst that the skew-symmetry
condition implies
\begin{gather*}
 \operatorname{ad}(x)(y)=-\big[\alpha ^{-1}\beta (y), \alpha \beta ^{-1}(x)\big],
\qquad\forall \, x, y\in L.
\end{gather*}
We compute the left-hand side of (\ref{lierep3}) applied to $z\in L$:
\begin{gather*}
(\operatorname{ad}( [\beta (x), y ])\circ \beta )(z) = \operatorname{ad}( [\beta (x), y ])(\beta (z))
 = -\big[\alpha ^{-1}\beta ^2(z), \alpha \beta ^{-1}([\beta (x), y])\big] \\
\hphantom{(\operatorname{ad}( [\beta (x), y ])\circ \beta )(z)}{}
=-\big[\beta ^2(\alpha ^{-1}(z)), \big[\alpha (x), \alpha \beta ^{-1}(y)\big]\big] \\
\hphantom{(\operatorname{ad}( [\beta (x), y ])\circ \beta )(z)}{}
 = -\big[\beta ^2(\alpha ^{-1}(z)), \big[\beta (\alpha \beta ^{-1}(x)),
\alpha (\beta ^{-1}(y))\big]\big].
\end{gather*}
We compute the right-hand side of~(\ref{lierep3}) applied to $z\in L$:
\begin{gather*}
 (\operatorname{ad}(\alpha \beta (x))\circ \operatorname{ad}(y))(z)-(\operatorname{ad}(\beta (y))\circ
\operatorname{ad}(\alpha (x))(z)\\
 \qquad{} = \operatorname{ad}(\alpha \beta (x))\big({-}\big[\alpha ^{-1}\beta (z), \alpha \beta ^{-1}(y)
\big]\big)- \operatorname{ad}(\beta (y))\big({-}\big[\alpha ^{-1}\beta (z), \alpha ^2\beta ^{-1}(x)
\big]\big) \\
\qquad{} = \big[\alpha ^{-1}\beta \big(\big[\alpha ^{-1}\beta (z), \alpha \beta
^{-1}(y)\big]\big), \alpha \beta ^{-1}\alpha \beta (x)\big] \\
\qquad\quad{}
 - \big[\alpha ^{-1}\beta \big(\big[\alpha ^{-1}\beta (z), \alpha ^2\beta
^{-1}(x)\big]\big), \alpha \beta ^{-1}\beta (y)\big] \\
\qquad{}
 = \big[\beta \big(\big[\alpha ^{-2}\beta (z), \beta ^{-1}(y)\big]\big), \alpha
^2(x)\big]- \big[\beta \big(\big[\alpha ^{-2}\beta (z), \alpha \beta ^{-1}(x)
\big]\big), \alpha (y)\big] \\
\qquad{}
 \overset{\text{skew-symmetry}}{=} -\big[\beta \alpha (x), \big[\alpha ^{-1}\beta
(z), \alpha \beta ^{-1}(y)\big]\big]+\big[\beta (y), \big[\alpha
^{-1}\beta (z), \alpha ^2\beta ^{-1}(x)\big]\big] \\
\qquad{}
 =  [\beta \alpha (x),  [y, z ] ]+\big[\beta (y), \big[
\beta \alpha ^{-1}(z), \alpha ^2\beta ^{-1}(x)\big]\big] \\
\qquad{}
 = \big[\beta ^2\big(\alpha \beta ^{-1}(x)\big), \big[\beta \big(\beta ^{-1}(y)\big),
\alpha \big(\alpha ^{-1}(z)\big)\big]\big]\!
 +\!\big[\beta ^2(\beta ^{-1}(y)), \big[\beta \big(\alpha ^{-1}(z)\big), \alpha
\big(\alpha \beta ^{-1}(x)\big)\big]\big],
\end{gather*}
and (\ref{lierep3}) holds because of the BiHom-Jacobi identity applied to
the elements $a=\alpha \beta ^{-1}(x)$, $a^{\prime}=\beta^{-1}(y)$ and $a^{\prime
\prime}=\alpha^{-1}(z)$.
\end{proof}

\section{BiHom-coassociative coalgebras and BiHom-bialgebras}

We introduce now the dual concept to the one of BiHom-associative
algebra.

\begin{Definition}
A \textit{BiHom-coassociative coalgebra} is a 4-tuple $(C, \Delta, \psi ,
\omega )$, in which $C$ is a linear space, $\psi , \omega \colon C\rightarrow C$
and $\Delta \colon C\rightarrow C\otimes C$ are linear maps, such that
\begin{gather*}
 \psi \circ \omega =\omega \circ \psi , \qquad
 (\psi \otimes \psi )\circ \Delta = \Delta \circ \psi , \qquad
 (\omega \otimes \omega )\circ \Delta = \Delta \circ \omega , \\
 (\Delta \otimes \psi )\circ \Delta = (\omega \otimes \Delta )\circ \Delta .
\end{gather*}

We call $\psi $ and $\omega $ (in this order) the \textit{structure maps} of
$C$.

A morphism $g\colon (C, \Delta _C , \psi _C, \omega _C)\rightarrow (D, \Delta _D ,
\psi _D, \omega _D)$ of BiHom-coassociative coalgebras is a~linear map $g\colon C\rightarrow D$ such that $\psi _D\circ g=g\circ \psi _C$, $\omega _D\circ
g=g\circ \omega _C$ and $(g\otimes g)\circ \Delta _C=\Delta _D\circ g$.

A BiHom-coassociative coalgebra $(C, \Delta, \psi , \omega )$ is called
\textit{counital} if there exists a linear map $\varepsilon\colon C\rightarrow
\Bbbk$ (called a~\textit{counit}) such that
\begin{gather*}
  \varepsilon\circ \psi= \varepsilon, \qquad \varepsilon\circ \omega=
\varepsilon, \qquad
  (\operatorname{id}_C\otimes \varepsilon) \circ \Delta=\omega \qquad \text{and} \qquad (
\varepsilon\otimes \operatorname{id}_C)\circ \Delta=\psi.
\end{gather*}

A morphism of counital BiHom-coassociative coalgebras $g\colon C\rightarrow D$ is
called \textit{counital} if $\varepsilon_D\circ g=\varepsilon_C$, where $\varepsilon _C$ and $\varepsilon _D$ are the counits of~$C$ and~$D$,
respectively.
\end{Definition}

\begin{Remark}
If $(C, \Delta _C, \psi _C, \omega _C)$ and $(D, \Delta _D, \psi _D, \omega
_D)$ are two BiHom-coassociative coalgebras, then $(C\otimes D, \Delta
_{C\otimes D}, \psi _C\otimes \psi _D, \omega _C\otimes \omega _D)$ is also
a BiHom-coassociative coalgebra (called the tensor product of $C$ and $D$),
where $\Delta _{C\otimes D}\colon C\otimes D\rightarrow C\otimes D\otimes C\otimes
D$ is def\/ined by $\Delta (c\otimes d)=c_1\otimes d_1\otimes c_2\otimes d_2$,
for all $c\in C$, $d\in D$. If $C$ and $D$ are counital with counits~$\varepsilon _C$ and respectively~$\varepsilon _D$, then $C\otimes D$ is also
counital with counit $\varepsilon _C\otimes \varepsilon _D$.
\end{Remark}

\begin{Definition}
Let $(C,\Delta _{C},\psi _{C},\omega _{C})$ be a BiHom-coassociative
coalgebra. A \textit{right $C$-comodule} is a triple $(M,\psi _{M},\omega
_{M})$, where $M$ is a linear space, $\psi _{M},\omega _{M}\colon M\rightarrow M$
are linear maps and we have a linear map (called a coaction) $\rho
\colon M\rightarrow M\otimes C$, with notation $\rho (m)=m_{(0)}\otimes m_{(1)}$,
for all $m\in M$, such that the following conditions are satisf\/ied
\begin{gather*}
\begin{split}
&  \psi _{M}\circ \omega _{M}=\omega _{M}\circ \psi _{M},  \qquad 
 (\psi _{M}\otimes \psi _{C})\circ \rho =\rho \circ \psi _{M},\qquad
 (\omega _{M}\otimes \omega _{C})\circ \rho =\rho \circ \omega _{M},\\
& (\omega _{M}\otimes \Delta _{C})\circ \rho =(\rho \otimes \psi _{C})\circ
\rho .  
\end{split}
\end{gather*}

If $(M, \psi _M, \omega _M)$ and $(N, \psi _N, \omega _N)$ are right $C$-comodules with coactions~$\rho _M$ and respecti\-ve\-ly~$\rho _N$, a morphism
of right $C$-comodules $f\colon M\rightarrow N$ is a linear map satisfying the
conditions $\psi _N\circ f=f\circ \psi _M$, $\omega _N\circ f=f\circ \omega
_M$ and $\rho _N\circ f=(f\otimes \operatorname{id}_C)\circ \rho _M$.

If $(C, \Delta _C, \psi _C, \omega _C, \varepsilon _C)$ is a counital
BiHom-coassociative coalgebra and $(M, \psi _M, \omega _M)$ is a right $C$-comodule with coaction $\rho $, then $M$ is called \textit{counital} if $(\operatorname{id}_M\otimes \varepsilon _C)\circ \rho =\omega _M$.
\end{Definition}

\begin{Remark}
If $(C, \Delta, \psi , \omega )$ is a BiHom-coassociative coalgebra, then $(C, \psi , \omega )$ is a right $C$-comodule, with coaction $\rho =\Delta $.
\end{Remark}

We discuss now the duality between BiHom-associative and BiHom-coassociative
structures.

\begin{Theorem}
Let $(C,\Delta, \psi , \omega )$ be a BiHom-coassociative coalgebra. Then
its dual linear space is provided with a structure of BiHom-associative
algebra $(C^*,\Delta^*,\omega^*,\psi^*)$, where $\Delta^*$, $\psi^*$, $\omega^*$
are the transpose maps. Moreover, the BiHom-associative algebra~$C^*$ is
unital whenever the BiHom-coassociative coalgebra~$C$ is counital.
\end{Theorem}

\begin{proof}
The product $\mu = \Delta^*$ is def\/ined from $C^* \otimes C^*$ to $C^*$ by
\begin{gather*}
(fg)(x) = \Delta^*(f,g)(x) = \langle \Delta(x),f \otimes g \rangle = (f
\otimes g)(\Delta(x)) = f(x_1)g(x_2), \qquad \forall \, x \in C,
\end{gather*}
where $\langle \cdot,\cdot \rangle$ is the natural pairing between the
linear space $C \otimes C$ and its dual linear space. For $f,g,h \in C^*$
and $x \in C$, we have
\begin{gather*}
(fg) \psi^*(h)(x) = \langle (\Delta \otimes \psi) \circ \Delta(x),f \otimes
g \otimes h \rangle,\\
 \omega^*(f)(gh)(x) = \langle (\omega
\otimes \Delta) \circ \Delta(x),f \otimes g \otimes h \rangle.
\end{gather*}
Therefore, the BiHom-associativity condition $\mu \circ (\mu \otimes \psi^* - \omega^*
\otimes \mu) = 0$ follows from the BiHom-coassociativity condition $(\Delta \otimes
\psi - \omega \otimes \Delta) \circ \Delta = 0$.

Moreover, if $C$ has a counit $\varepsilon $ then for $f \in C^*$ and $x \in
C$ we have
\begin{gather*}
 (\varepsilon f)(x) = \varepsilon (x_1) f(x_2) = f(\varepsilon (x_1) x_2) =
f(\psi(x)) = \psi^*(f)(x), \\
 (f \varepsilon )(x) = f(x_1) \varepsilon (x_2) = f(x_1 \varepsilon (x_2))
= f(\omega(x)) = \omega^*(f)(x),
\end{gather*}
which shows that $\varepsilon $ is the unit of~$C^*$.
\end{proof}

The dual of a BiHom-associative algebra $(A,\mu,\alpha,\beta)$ is not always
a BiHom-coassociative coalgebra, because $(A \otimes A)^* \supsetneq A^*
\otimes A^*$. Nevertheless, it is the case if the BiHom-associative algebra
is f\/inite-dimensional, since $(A\otimes A)^* = A^* \otimes A^*$ in this case.

More generally, we can def\/ine the f\/inite dual of $A$ by
\begin{gather*}
A^\circ = \{f \in A^*/ f(I)=0 \   \text{for some cof\/inite ideal $I$ of $A$}\},
\end{gather*}
where a cof\/inite ideal $I$ is an ideal $I \subset A$ such that $A/I$ is
f\/inite-dimensional and where we say that $I$ is an ideal of $A$ if for $x
\in I$ and $y \in A$ we have $x y\in I$, $y x\in I$ and $\alpha (x)\in I$,
$\beta (x)\in I$.

$A^\circ$ is a subspace of $A^*$ since it is closed under multiplication by
scalars and the sum of two elements of~$A^\circ$ is again in $A^\circ$
because the intersection of two cof\/inite ideals is again a~cof\/inite ideal.
If~$A$ is f\/inite-dimensional, of course $A^\circ = A^*$.
As in the classical case, one can show that if~$A$ and~$B$ are two
BiHom-associative algebras and $f \colon  A \to B$ is a~morphism of
BiHom-associative algebras, then the dual map $f^* \colon B^* \to A^*$ satisf\/ies $f^*(B^\circ)\subset A^\circ$.

Therefore, a similar proof to the one of the previous theorem leads to:

\begin{Theorem}
Let $(A,\mu,\alpha,\beta)$ be a BiHom-associative algebra. Then its finite
dual is provided with a structure of BiHom-coassociative coalgebra $(A^\circ,\Delta,\beta^\circ,\alpha^\circ)$, where $\Delta = \mu^\circ =
\mu^*|_{A^\circ}$ and~$\beta^\circ$,~$\alpha^\circ$ are the transpose maps on~$A^\circ$. Moreover, the BiHom-coassociative coalgebra is counital whenever~$A $ is unital, with counit $\varepsilon \colon A^\circ \to \Bbbk $ defined by $\varepsilon (f) = f(1_A)$.
\end{Theorem}

We can now def\/ine the notion of BiHom-bialgebra.

\begin{Definition}
A \textit{BiHom-bialgebra} is a 7-tuple $(H, \mu , \Delta, \alpha , \beta ,
\psi , \omega )$, with the property that $(H, \mu , \alpha , \beta )$ is a
BiHom-associative algebra, $(H, \Delta , \psi , \omega )$ is a~BiHom-coassociative coalgebra and moreover the following relations are
satisf\/ied, for all $h, h^{\prime }\in H$:
\begin{gather}
 \Delta (hh^{\prime })=h_1h^{\prime }_1\otimes h_2h^{\prime }_2,
\label{ghombia2} \\
 \alpha \circ \psi =\psi \circ \alpha , \qquad \alpha \circ \omega =\omega
\circ \alpha , \qquad \beta \circ \psi =\psi \circ \beta , \qquad \beta \circ
\omega =\omega \circ \beta , \nonumber\\
 (\alpha \otimes \alpha )\circ \Delta =\Delta \circ \alpha , \qquad (\beta
\otimes \beta )\circ \Delta =\Delta \circ \beta , \nonumber\\
 \psi (hh^{\prime })=\psi (h)\psi (h^{\prime }), \qquad \omega (hh^{\prime
})=\omega (h)\omega (h^{\prime }).\nonumber
\end{gather}

We say that $H$ is a unital and counital BiHom-bialgebra if, in addition, it
admits a unit $1_H$ and a counit $\varepsilon _H$ such that
\begin{gather*}
 \Delta(1_H)=1_H\otimes 1_H,\qquad \varepsilon _H(1_H)=1, \qquad \psi
(1_H)=1_H, \qquad \omega (1_H)=1_H, \\
 \varepsilon _H\circ \alpha =\varepsilon _H, \qquad \varepsilon _H\circ
\beta =\varepsilon _H, \qquad \varepsilon _H(hh^{\prime })=\varepsilon
_H(h)\varepsilon _H(h^{\prime }), \qquad \forall \,h, h^{\prime }\in H.
\end{gather*}
\end{Definition}

Let us record the formula expressing the BiHom-coassociativity of $\Delta $:
\begin{gather}
 \Delta (h_1)\otimes \psi (h_2)=\omega (h_1)\otimes \Delta (h_2),
\qquad \forall \,h\in H.  \label{ghombia1}
\end{gather}

\begin{Remark}
Obviously, a BiHom-bialgebra $(H, \mu , \Delta, \alpha , \beta , \psi ,
\omega )$ with $\alpha =\beta =\psi =\omega $ reduces to a Hom-bialgebra, as
used for instance in~\cite{mp1,mp2}, while a~BiHom-bialgebra for
which $\psi =\omega =\alpha ^{-1} =\beta ^{-1}$ reduces to a~monoidal
Hom-bialgebra, in the terminology of~\cite{stef}.
\end{Remark}

We see now that analogues of Yau's twisting principle hold for the
BiHom-structures we def\/ined (proofs are straightforward and left to the
reader):

\begin{Proposition}\label{yautwistdiverse}\quad
\begin{enumerate}\itemsep=0pt
 \item[$(i)$] Let $(A, \mu )$ be an associative algebra and $
\alpha , \beta \colon A\rightarrow A$ two commuting algebra endomorphisms. Define
a new multiplication $\mu _{(\alpha , \beta )}\colon A\otimes A\rightarrow A$, by $
\mu _{(\alpha , \beta )}:= \mu \circ (\alpha \otimes \beta )$. Then $(A, \mu
_{(\alpha , \beta )}, \alpha , \beta )$ is a BiHom-associative algebra,
denoted by $A_{(\alpha , \beta )}$. If $A$ is unital with unit~$1_A$, then $A_{(\alpha , \beta )}$ is also unital with unit~$1_A$.
\item[$(ii)$] Let $(C, \Delta )$ be a coassociative coalgebra and $\psi , \omega
\colon C\rightarrow C$ two commuting coalgebra endomorphisms. Define a new
comultiplication $\Delta _{(\psi , \omega ) }\colon C\rightarrow C\otimes C$, by $\Delta _{(\psi , \omega ) }:=(\omega \otimes \psi )\circ \Delta $. Then $(C,
\Delta _{(\psi , \omega ) }, \psi , \omega )$ is a BiHom-coassociative
coalgebra, denoted by $C_{(\psi , \omega )}$. If $C$ is counital with counit
$\varepsilon _C$, then $C_{(\psi , \omega )}$ is also counital with counit $\varepsilon _C$.
\item[$(iii)$] Let $(H, \mu , \Delta )$ be a bialgebra and $\alpha , \beta , \psi ,
\omega \colon H\rightarrow H$ bialgebra endomorphisms such that any two of them
commute. If we define $\mu _{(\alpha , \beta )}$ and $\Delta _{(\psi ,
\omega ) }$ as in $(i)$ and $(ii)$, then $H_{(\alpha , \beta , \psi , \omega
)}:=(H, \mu _{(\alpha , \beta )}, \Delta _{(\psi , \omega ) }, \alpha ,
\beta , \psi , \omega )$ is a BiHom-bialgebra.
\end{enumerate}
\end{Proposition}

More generally, a BiHom-bialgebra $(H, \mu, \Delta , \alpha , \beta , \psi ,
\omega )$ and multiplicative and comultiplicative linear maps $
\alpha^{\prime }$, $\beta^{\prime }$, $\psi^{\prime }$, $\omega^{\prime }$ such
that any two of the maps $\alpha$, $\beta$, $\psi$, $\omega$, $\alpha^{\prime }$,
$\beta^{\prime }$, $\psi^{\prime }$, $\omega^{\prime }$ commute, give rise to a
new BiHom-bialgebra $(H, \mu \circ(\alpha^{\prime }\otimes \beta^{\prime }),
(\omega ^{\prime }\otimes \psi ^{\prime })\circ \Delta ,
\alpha\circ\alpha^{\prime }, \beta\circ\beta^{\prime },
\psi\circ\psi^{\prime }, \omega\circ\omega^{\prime })$. Hence, if the maps $
\alpha$, $\beta$, $\psi$, $\omega$ are invertible, one can untwist the
BiHom-bialgebra and get a~bialgebra by taking $\alpha^{\prime }= \alpha^{-1}$, $\beta^{\prime}=\beta^{-1}$, $\psi^{\prime}=\psi^{-1}$, $\omega^{\prime}=\omega^{-1}$.

\begin{Proposition}
\label{TwistLeftMod} Let $(A, \mu _A)$ be an associative algebra and $\alpha
_A, \beta _A\colon A\rightarrow A$ two commuting algebra endomorphisms. Assume
that $M$ is a left $A$-module, with action $A\otimes M\rightarrow M$, $
a\otimes m \mapsto a\cdot m$. Let $\alpha _M, \beta _M\colon M\rightarrow M$ be
two commuting linear maps such that $\alpha _M(a\cdot m)=\alpha _A(a)\cdot
\alpha _M(m)$ and $\beta _M(a\cdot m)=\beta _A(a)\cdot \beta _M(m)$, for all
$a\in A$, $m\in M$. Then $(M, \alpha _M, \beta _M)$ becomes a left module
over $A_{(\alpha _A, \beta _A)}$, with action $A_{(\alpha _A, \beta
_A)}\otimes M\rightarrow M$, $a\otimes m\mapsto a\triangleright m:= \alpha
_A(a)\cdot \beta _M(m)$.
\end{Proposition}

\begin{Proposition}
Let $(C, \Delta _C)$ be a coassociative coalgebra and $\psi _C, \omega
_C\colon C\rightarrow C$ two commuting coalgebra endomorphisms. Assume that~$M$ is
a right $C$-comodule, with coaction $\rho \colon M\rightarrow M\otimes C$, $\rho
(m)=m_{(0)}\otimes m_{(1)}$, for all $m\in M$. Let $\psi _M, \omega _M\colon
M\rightarrow M$ be two commuting linear maps such that $(\psi _M\otimes \psi
_C)\circ \rho =\rho \circ \psi _M$ and $(\omega _M\otimes \omega _C)\circ
\rho =\rho \circ \omega _M$. Then $(M, \psi _M, \omega _M)$ becomes a right
comodule over the BiHom-coassociative coalgebra $C_{(\psi _C, \omega _C)}$,
with coaction
\begin{gather*}
 M\rightarrow M\otimes C_{(\psi _C, \omega _C)}, \qquad m\mapsto
m_{\langle 0\rangle}\otimes m_{\langle 1\rangle }:= \omega _M(m_{(0)})\otimes \psi _C(m_{(1)}).
\end{gather*}
\end{Proposition}

We describe in what follows primitive elements of a BiHom-bialgebra.

Let $(H, \mu , \Delta, \alpha , \beta , \psi , \omega )$ be a unital and
counital BiHom-bialgebra with a unit $1=\eta(1)$ and a counit $\varepsilon$.
We assume that $\alpha $ and $\beta$ are bijective.

An element $x\in H$ is called \textit{primitive} if $\Delta (x)=1\otimes x +
x\otimes 1$.

\begin{Lemma}
Let $x$ be a primitive element in~$H$. Then $\varepsilon(x)1=\omega(x)-x=
\psi(x)-x$, and therefore $\omega(x)=\psi(x)$. Moreover, $\alpha^p\beta^q(x)$
is also a primitive element for any $p,q\in\mathbb{Z}$.
\end{Lemma}

\begin{proof}
By the counit property, we have $\omega( x)=(\operatorname{id} _H\otimes
\varepsilon)(1\otimes x+x\otimes 1)=\varepsilon
(x)1+\varepsilon(1)x=\varepsilon (x)1+x$, and similarly $\psi(
x)=\varepsilon (x)1+x$.

Since $\alpha$ and $\beta$ are comultiplicative maps and $\alpha^p\beta^q(1)=1$, it follows that $\alpha^p\beta^q(x)$ is a~primitive
element whenever~$x$ is a~primitive element.
\end{proof}

\begin{Proposition}
Let $(H, \mu , \Delta, \alpha , \beta , \psi , \omega )$ be a unital and
counital BiHom-bialgebra, with unit $1=\eta(1)$ and counit~$\varepsilon $.
Assume that $\alpha $ and $\beta $ are bijective. If $x$ and $y$ are two
primitive elements in~$H$, then the commutator $[x,y]=x y -\alpha^{-1}\beta
(y) \alpha\beta^{-1}(x)$ is also a primitive element.

Consequently, the set of all primitive elements of $H$, denoted by $\operatorname{Prim}(H)$, has a structure of BiHom-Lie algebra.
\end{Proposition}

\begin{proof}
We compute
\begin{gather*}
\Delta (xy) = \Delta (x) \Delta (y)
=(1\otimes x+x\otimes 1)(1\otimes y+y\otimes 1) \\
\hphantom{\Delta (xy)}{}
= 1\otimes x y + \beta(y)\otimes \alpha(x)+\alpha(x)\otimes \beta(y) + x
y\otimes 1,
\\
\Delta \big(\alpha^{-1}\beta (y) \alpha\beta^{-1}(x)\big)   =  \Delta
\big(\alpha^{-1}\beta (y)\big) \Delta \big( \alpha\beta^{-1}(x)\big) \\
\hphantom{\Delta \big(\alpha^{-1}\beta (y) \alpha\beta^{-1}(x)\big)}{}
 = \big(1\otimes \alpha^{-1}\beta (y)+\alpha^{-1}\beta (y)\otimes 1\big)\big(1\otimes
\alpha\beta^{-1}(x)+ \alpha\beta^{-1}(x)\otimes 1\big) \\
\hphantom{\Delta \big(\alpha^{-1}\beta (y) \alpha\beta^{-1}(x)\big)}{}
 = 1\otimes \alpha^{-1}\beta (y) \alpha\beta^{-1}(x) +
\beta\big(\alpha\beta^{-1}(x)\big)\otimes \alpha\big(\alpha^{-1}\beta (y)\big) \\
\hphantom{\Delta \big(\alpha^{-1}\beta (y) \alpha\beta^{-1}(x)\big)=}{}
 +\alpha\big(\alpha^{-1}\beta (y)\big)\otimes \beta\big( \alpha\beta^{-1}(x)\big) +
\alpha^{-1}\beta (y) \alpha\beta^{-1}(x)\otimes 1 \\
\hphantom{\Delta \big(\alpha^{-1}\beta (y) \alpha\beta^{-1}(x)\big)}{}
 = 1\otimes \alpha^{-1}\beta (y) \alpha\beta^{-1}(x) + \alpha(x)\otimes
\beta (y) \\
\hphantom{\Delta \big(\alpha^{-1}\beta (y) \alpha\beta^{-1}(x)\big)=}{}
 +\beta (y)\otimes \alpha(x) + \alpha^{-1}\beta (y)
\alpha\beta^{-1}(x)\otimes 1.
\end{gather*}
Therefore, we have
\begin{gather*}
\Delta([x,y])=\Delta (xy)-\Delta \big(\alpha^{-1}\beta (y) \alpha\beta^{-1}(x)\big)=1
\otimes[x,y]+[x,y]\otimes 1,
\end{gather*}
which means that $Prim(H)$ is closed under the bracket multiplication $[\cdot,\cdot]$. Hence, $\operatorname{Prim}(H)$ is a BiHom-Lie algebra by Proposition~\ref{croset}.
\end{proof}

Now, we introduce the notion of $H$-module BiHom-algebra, where $H$ is a
BiHom-bialgebra.

\begin{Definition}
Let $(H, \mu _H, \Delta _H, \alpha _H, \beta _H, \psi _H, \omega _H)$ be a
BiHom-bialgebra for which the maps $\alpha _H$, $\beta _H$, $\psi _H$, $\omega _H$
are bijective. A BiHom-associative algebra $(A, \mu _A, \alpha _A, \beta _A)$
is called a \textit{left $H$-module BiHom-algebra} if $(A, \alpha _A,
\beta _A)$ is a left $H$-module, with action denoted by $H\otimes
A\rightarrow A$, $h\otimes a\mapsto h\cdot a$, such that the following
condition is satisf\/ied
\begin{gather}
 h\cdot (aa^{\prime })=\big[\alpha _H^{-1}\big(\omega _H^{-1}(h_1)\big)\cdot a\big] [\beta
_H^{-1}\big(\psi _H^{-1}(h_2)\big)\cdot a^{\prime }], \qquad \forall \, h\in H, \quad a,
a^{\prime }\in A.  \label{gmodalgcompat}
\end{gather}
\end{Definition}

\begin{Remark}
This concept contains as particular cases the concepts of module algebras
over a Hom-bialgebra, respectively monoidal Hom-bialgebra, introduced in~\cite{yau1}, respectively~\cite{chenwangzhang}.
\end{Remark}

The choice of (\ref{gmodalgcompat}) is motivated by the following result,
whose proof is also left to the reader:

\begin{Proposition}
\label{gyaumodalg} Let $(H, \mu _H, \Delta _H)$ be a bialgebra and $(A, \mu
_A)$ a left $H$-module algebra in the usual sense, with action denoted by $H\otimes A\rightarrow A$, $h\otimes a\mapsto h\cdot a$. Let $\alpha _H,
\beta _H, \psi _H, \omega _H\colon H\rightarrow H$ be bialgebra endomorphisms of $H $ such that any two of them commute; let $\alpha _A, \beta _A\colon A\rightarrow
A$ be two commuting algebra endomorphisms such that, for all $h\in H$ and $a\in A$, we have
\begin{gather*}
 \alpha _A(h\cdot a)=\alpha _H(h)\cdot \alpha _A(a)\qquad \text{and} \qquad \beta
_A(h\cdot a)=\beta _H(h)\cdot \beta _A(a).
\end{gather*}
If we consider the BiHom-bialgebra $H_{(\alpha _H, \beta _H, \psi _H, \omega
_H)}$ and the BiHom-associative algebra $A_{(\alpha _A, \beta _A)}$ as
defined before, then $A_{(\alpha _A, \beta _A)}$ is a left $H_{(\alpha _H,
\beta _H, \psi _H, \omega _H)}$-module BiHom-algebra in the above sense,
with action
\begin{gather*}
 H_{(\alpha _H, \beta _H, \psi _H, \omega _H)}\otimes A_{(\alpha _A, \beta
_A)}\rightarrow A_{(\alpha _A, \beta _A)}, \qquad  h\otimes a\mapsto
h\triangleright a:=\alpha _H(h)\cdot \beta _A(a).
\end{gather*}
\end{Proposition}

\section{Monoidal BiHom-Hopf algebras and BiHom-Hopf algebras}

In this section, we introduce the concept of monoidal BiHom-Hopf
algebra and discuss a possible generalization of Hom-Hopf algebras to
BiHom-Hopf algebras.

We begin with a lemma whose proof is obvious.

\begin{Lemma}
\label{subalg} Let $(A,\mu ,\alpha ,\beta )$ be a BiHom-associative algebra.
Define $\underline{A}:=\{a\in A/\alpha (a)=\beta (a)=a\}$. Then $(\underline{A},\mu )$ is an associative algebra. If~$A$ is unital with unit $1_{A}$,
then $1_{A}$ is also the unit of $\underline{A}$ $($in particular, it follows
that the unit of a BiHom-associative algebra, if it exists, is unique$)$.
\end{Lemma}

\begin{Proposition}
Let $ ( A,\mu ,\alpha,\beta  ) $ be a BiHom-associative algebra and
$(C, \Delta , \psi , \omega )$ a BiHom-coassociative coalgebra. Set, for $f,g\in \operatorname{Hom} ( C,A )$, $f \star g=\mu \circ  ( f\otimes g
 ) \circ\Delta $. Define the linear maps $\phi , \gamma \colon \operatorname{Hom}(C,
A)\rightarrow \operatorname{Hom}(C, A)$ by $\phi(f)=\alpha\circ f\circ \omega$ and $\gamma(f)=\beta\circ f\circ \psi$, for all $f\in \operatorname{Hom}(C, A)$. Then $(\operatorname{Hom}( C,A ),\star,\phi,\gamma)$ is a BiHom-associative algebra.

Moreover, if $A$ is unital with unit $1_A$ and $C$ is counital with counit $\varepsilon$, then $\operatorname{Hom}( C,A)$ is a~unital BiHom-asssociative
algebra with unit $\eta \circ \varepsilon$, where we denote by $\eta $ the
linear map $\eta \colon \Bbbk \rightarrow A$, $\eta (1)=1_A$.

In particular, if we denote by $\underline{\operatorname{Hom}}(C, A)$ the linear subspace
of $\operatorname{Hom}(C, A)$ consisting of the linear maps $f\colon C\rightarrow A$ such that $\alpha \circ f\circ \omega =f$ and $\beta \circ f\circ \psi =f$, then $(\underline{\operatorname{Hom}}(C, A), \star , \eta \circ \varepsilon )$ is an associative
unital algebra.
\end{Proposition}

\begin{proof}
Let $f,g,h\in \operatorname{Hom}( C,A) $. We have
\begin{gather*}
\phi (f)\star (g\star h)  = \mu \circ ( \phi (f)\otimes (g\star
h)) \Delta =\mu \circ  ( \phi (f)\otimes (\mu \circ  (
g\otimes h ) \circ \Delta ) ) \Delta \\
\hphantom{\phi (f)\star (g\star h)}{}
 = \mu \circ  (  ( \alpha \otimes \mu  ) \circ  ( f\otimes
g\otimes h ) \circ  ( \omega \otimes \Delta  )  ) \Delta .
\end{gather*}
Similarly,
\begin{gather*}
(f\star g)\star \gamma (h)=\mu \circ  (  ( \mu \otimes \beta  )
\circ  ( f\otimes g\otimes h ) \circ  ( \Delta \otimes \psi
 )  ) \Delta .
\end{gather*}
The BiHom-associativity of $\mu $ and the BiHom-coassociativity of~$\Delta $
lead to the BiHom-associa\-ti\-vi\-ty of the convolution product~$\star $.

The map $\eta \circ \varepsilon $ is the unit for the convolution product.
Indeed, for $f\in \operatorname{Hom}( C,A ) $ and $x\in C$, we have
\begin{gather*}
\begin{split}
&  (f\star (\eta \circ \varepsilon ))(x)=\mu \circ  ( f\otimes \eta \circ
\varepsilon  ) \circ \Delta (x)= \mu  ( f(x_{1})\otimes \eta \circ
\varepsilon (x_{2}) ) = \varepsilon (x_{2})\mu  ( f(x_{1})\otimes \eta (1) )  \\
& \hphantom{(f\star (\eta \circ \varepsilon ))(x)}{}
  =
\varepsilon (x_{2})  (  \alpha \circ f )  (x_{1}) = (  \alpha
\circ f )  (x_{1}\varepsilon (x_{2})) =\alpha \circ f\circ \omega (x).
\end{split}
\end{gather*}
A similar calculation shows that $(\eta \circ \varepsilon )\star f=\beta
\circ f\circ \psi $.

The last statement follows from Lemma~\ref{subalg}.
\end{proof}

\begin{Definition}
Let $(H, \mu , \Delta, \alpha , \beta , \psi , \omega )$ be a unital and
counital BiHom-bialgebra. We say that~$H$ is a \textit{monoidal
BiHom-bialgebra} if $\alpha$, $\beta$, $\psi$, $\omega $ are bijective and $\omega =\alpha ^{-1}$ and $\psi =\beta ^{-1}$. We will refer to a monoidal BiHom-bialgebra as the
$5$-tuple $(H, \mu , \Delta, \alpha , \beta )$.
\end{Definition}

If $(H, \mu , \Delta, \alpha , \beta )$ is a monoidal
BiHom-bialgebra, we can consider the associative unital algebra $\underline{\operatorname{Hom}}(H, H)$, and since $\omega =\alpha ^{-1}$ and $\psi =\beta ^{-1}$, it
follows that $\operatorname{id}_H\in \underline{\operatorname{Hom}}(H, H)$.

\begin{Definition}
Let $(H, \mu , \Delta, \alpha , \beta )$ be a monoidal
BiHom-bialgebra with a unit~$1_H$ and a~co\-unit~$\varepsilon _H$. A linear
map $S\colon H\rightarrow H$ is called an \textit{antipode} if $\alpha \circ
S=S\circ \alpha $ and $\beta \circ S=S\circ \beta $ (i.e., $S\in \underline{\operatorname{Hom}}(H, H)$) and $S$ is the convolution inverse of $\operatorname{id}_H$ in $\underline{\operatorname{Hom}}(H, H)$, that is
\begin{gather*}
 S(h_1)h_2=\varepsilon _H(h)1_H=h_1S(h_2), \qquad \forall \, h\in H.
\end{gather*}
A \textit{monoidal BiHom-Hopf algebra} is a monoidal BiHom-bialgebra endowed
with an antipode.

Obviously, if the antipode exists, it is unique; we will refer to the monoidal BiHom-Hopf algebra
as the $8$-tuple $(H, \mu , \Delta, \alpha , \beta , 1_H,
\varepsilon _H, S)$.
\end{Definition}

\begin{Proposition}
Let $(H, \mu , \Delta , 1_H, \varepsilon _H)$ be a Hopf algebra $($in the
usual sense$)$ with antipo\-de~$S$. Let $\alpha , \beta \colon H\rightarrow H$ be two
unital and counital commuting bialgebra automorphisms. Then $(H, \mu \circ
(\alpha \otimes \beta ), (\alpha ^{-1}\otimes \beta ^{-1})\circ \Delta ,
\alpha , \beta , 1_H, \varepsilon _H, S)$ is a
monoidal BiHom-bialgebra.
\end{Proposition}

\begin{proof}
A straightforward computation. Let us only note that~$\alpha$,~$\beta $ being
bialgebra maps, they automatically commute with~$S$.
\end{proof}

We state now the basic properties of the antipode.

\begin{Proposition}
Let $(H, \mu , \Delta, \alpha , \beta , 1_H,
\varepsilon _H, S)$ be a monoidal BiHom-Hopf algebra. Then
\begin{enumerate}\itemsep=0pt
\item[$(i)$] $S(1_H)=1_H$ and $\varepsilon _H\circ S=\varepsilon _H$;

\item[$(ii)$] $S(\beta (a)\alpha (b))=S(\beta (b))S(\alpha (a))$, for all $a, b\in H$;

\item[$(iii)$] $\alpha (S(h)_1)\otimes \beta (S(h)_2)=\beta (S(h_2))\otimes \alpha
(S(h_1))$, for all $h\in H$.
\end{enumerate}
\end{Proposition}

\begin{proof}
(i) By $\Delta (1_H)=1_H\otimes 1_H$ we obtain $S(1_H)1_H=\varepsilon
_H(1_H)1_H$, so $\alpha (S(1_H))=1_H$, and since $\alpha \circ S=S\circ
\alpha $ and $\alpha (1_H)=1_H$ we obtain $S(1_H)=1_H$. Then, if $h\in H$,
we apply $\varepsilon _H$ to the equality $h_1S(h_2)=\varepsilon _H(h)1_H$,
and we obtain $\varepsilon _H(h_1)\varepsilon _H(S(h_2)) =\varepsilon _H(h)$, so $\varepsilon _H(S(\varepsilon _H(h_1)h_2))=\varepsilon _H(h)$, hence $\varepsilon _H(S(\beta ^{-1}(h)))=\varepsilon _H(h)$, and since $S\circ
\beta =\beta \circ S$ and $\varepsilon _H\circ \beta =\varepsilon _H$ we
obtain $\varepsilon _H\circ S=\varepsilon _H$.

(ii) We def\/ine the linear maps $R, L, m\colon H\otimes H\rightarrow H$ by the
formulae (for all $a, b\in H$):
\begin{gather*}
 R(a\otimes b)=S(\beta (b))S(\alpha (a)), \qquad
 L(a\otimes b)=S(\beta (a)\alpha (b)), \qquad
 m(a\otimes b)=\beta (a)\alpha (b).
\end{gather*}
One can easily check that $R, L, m\in \underline{\operatorname{Hom}}(H\otimes H, H)$ (where
$H\otimes H$ is the tensor product BiHom-coassociative coalgebra). Thus, to
prove that $R=L$, it is enough to prove that $L$ (respective\-ly~$R$) is a left
(respectively right) convolution inverse of $m$ in $\underline{\operatorname{Hom}}(H\otimes
H, H)$. We compute
\begin{gather*}
(L\star m)(a\otimes b) = L(a_1\otimes b_1)m(a_2\otimes b_2)   = S(\beta (a_1)\alpha (b_1))(\beta (a_2)\alpha (b_2)) \\
\hphantom{(L\star m)(a\otimes b)}{}
= S(\beta (a)_1\alpha (b)_1)(\beta (a)_2\alpha (b)_2)
 = S((\beta (a)\alpha (b))_1)(\beta (a)\alpha (b))_2 \\
 \hphantom{(L\star m)(a\otimes b)}{}
 = \varepsilon _H(\beta (a)\alpha (b))1_H
 = \varepsilon _H(a)\varepsilon _H(b)1_H,
\\
(m\star R)(a\otimes b) = m(a_1\otimes b_1)R(a_2\otimes b_2)
 = (\beta (a_1)\alpha (b_1))(S(\beta (b_2))S(\alpha (a_2))) \\
 \hphantom{(m\star R)(a\otimes b)}{}
 = \alpha \big(\alpha ^{-1}\beta (a_1)b_1\big)(\beta (S(b_2))\alpha (S(a_2)))
 = \big(\big(\alpha ^{-1}\beta (a_1)b_1\big)\beta (S(b_2)))\alpha \beta (S(a_2)\big) \\
\hphantom{(m\star R)(a\otimes b)}{}
 = (\beta (a_1)(b_1S(b_2)))\alpha \beta (S(a_2))
 = (\beta (a_1)\varepsilon _H(b)1_H)\alpha \beta (S(a_2)) \\
 \hphantom{(m\star R)(a\otimes b)}{}
 = \varepsilon _H(b)\alpha \beta (a_1)\alpha \beta (S(a_2))
 = \varepsilon _H(b)\alpha \beta (a_1S(a_2)) \\
\hphantom{(m\star R)(a\otimes b)}{}
 = \varepsilon _H(b)\alpha \beta (\varepsilon _H(a)1_H)
 = \varepsilon _H(a)\varepsilon _H(b)1_H,
\end{gather*}
f\/inishing the proof.

(iii) similar to the proof of (ii), by def\/ining the linear maps $\mathcal{L}, \mathcal{R}, \delta \colon H\rightarrow H\otimes H$,
\begin{gather*}
 \mathcal{L}(h)=\alpha (S(h)_1)\otimes \beta (S(h)_2), \qquad\!
 \mathcal{R}(h)=\beta (S(h_2))\otimes \alpha (S(h_1)), \qquad\!
 \delta (h)=\alpha (h_1)\otimes \beta (h_2),
\end{gather*}
for all $h\in H$, and proving that $\mathcal{L}$ (respectively $\mathcal{R}$) is a left (respectively right) convolution inverse of $\delta $ in $\underline{\operatorname{Hom}}(H, H\otimes H)$.
\end{proof}

\begin{Remark}
We had to restrict the def\/inition of the antipode to the class of monoidal
BiHom-bialgebras because, if~$H$ is a Hopf algebra with antipode~$S$ and we
make an arbitrary Yau twist of~$H$, then in general~$S$ will not satisfy the
def\/ining property of an antipode for the Yau twist, as the next example
shows.
\end{Remark}

\begin{Example}
Let $\Bbbk $ be a f\/ield and let $H=\Bbbk \left[ X\right] $, regarded as a
Hopf algebra in the usual way. Let $\alpha \colon H\rightarrow H$ be the algebra
map def\/ined by setting $\alpha (X) =X^{2}$ and let $\beta =
\omega =\psi =\operatorname{Id}_{H}$. Then we can consider the BiHom-bialgebra $H_{(\alpha ,\beta ,\psi ,\omega )}:=(H,\mu _{(\alpha ,\beta )},\Delta
_{(\psi ,\omega )},\alpha ,\beta ,\psi ,\omega )$, where $\mu \colon H\otimes
H\rightarrow H$ is the usual multiplication and $\Delta \colon H\rightarrow
H\otimes H$ is the usual comultiplication. Moreover $H_{(\alpha ,\beta ,\psi
,\omega )}$ has unit $1_{H}=\eta _H\left( 1_{\Bbbk }\right) $ and counit $\varepsilon _{H}$ that coincide with the ones of~$H$.

Assume that there exists a linear map $S\colon H\rightarrow H$ such that $S\star
\operatorname{Id}=\operatorname{Id}\star S=\eta _H\circ \varepsilon _H$, i.e.,
\begin{gather}
\mu _{(\alpha ,\beta )}\circ  ( S\otimes \operatorname{Id} ) \circ \Delta
_{(\psi ,\omega )}=\mu _{(\alpha ,\beta )}\circ  ( \operatorname{Id}\otimes
S ) \circ \Delta _{(\psi ,\omega )}=\eta _{H}\circ \varepsilon _{H}.
\label{form:1}
\end{gather}
Then we compute
\begin{gather*}
\big( \mu _{(\alpha ,\beta )}\circ  ( \operatorname{Id}\otimes S ) \circ
\Delta _{(\psi ,\omega )}\big) (X) =\alpha (X)
S ( 1 ) +\alpha  ( 1 ) S(X) =X^{2}S (
1 ) +S(X),
\\
 ( \mu _{(\alpha ,\beta )}\circ  ( S\otimes \operatorname{Id} ) \circ
\Delta _{(\psi ,\omega )} ) (X) =\alpha  ( S (
X )  ) 1+\alpha  ( S ( 1 )  ) X,
\end{gather*}
and
\begin{gather*}
 ( \eta _{H}\circ \varepsilon _{H} ) (X) =01_{H},
\end{gather*}
so that from (\ref{form:1}) we get
\begin{gather}
S(X) =-X^{2}S(1)   \label{form:2}
\end{gather}
and
\begin{gather}
\alpha  ( S(X)  ) =-\alpha  ( S(1) ) X,  \label{form:3}
\end{gather}
and hence
\begin{gather*}
-\alpha ( S(1)  ) X\overset{\eqref{form:3}}{=}\alpha ( S(X) ) \overset{\eqref{form:2}}{=}\alpha \big( {-}X^{2}S(1)   ) \overset{\text{def.}\, \alpha }{=}-\alpha \big( X^{2}\big) \alpha  ( S(1)
 ) \overset{\text{def.}\,\alpha }{=}-X^{4}\alpha  ( S (1) ),
\end{gather*}
so that we get $\alpha ( S(1)  ) X=X^{4}\alpha (
S(1)  )$, which implies that $\alpha ( S(
1) ) =0$.

On the other hand, we have
\begin{gather*}
1= ( \eta _{H}\circ \varepsilon _{H} ) (1)  \overset{\eqref{form:1}}{=}\big( \mu _{(\alpha ,\beta )}\circ  (
S\otimes \operatorname{Id} ) \circ \Delta _{(\psi ,\omega )}\big)  (
1 ) =\alpha  ( S(1)   ) 1=0,
\end{gather*}
and this is a contradiction.
\end{Example}

In view of all the above, we propose the following def\/inition for what might
be a BiHom-Hopf algebra, that is moreover invariant under Yau twisting:

\begin{Definition}\sloppy
Let $(H, \mu , \Delta, \alpha, \beta , \psi , \omega )$ be a unital and
counital BiHom-bialgebra with a~unit~$1_H$ and a~counit $\varepsilon _H$. A~linear map $S\colon H\rightarrow H$ is called an \textit{antipode} if it commutes
with all the maps $\alpha$, $\beta$, $\psi$, $\omega $ and it satisf\/ies the
following relation:
\begin{gather*}
\beta \psi (S(h_1))\alpha \omega (h_2)=\varepsilon _H(h)1_H= \beta \psi
(h_1)\alpha \omega (S(h_2)), \qquad \forall \, h\in H.
\end{gather*}
A \textit{BiHom-Hopf algebra} is a unital and counital BiHom-bialgebra with
an antipode.
\end{Definition}

We hope to make a more detailed analysis of these structures in a
forthcoming paper.

\section{BiHom-pseudotwistors and BiHom-twisted tensor products}

Inspired by Proposition~\ref{yaugeneral}, by the concept of
pseudotwistor for associative algebras introduced in~\cite{lpvo} and its
generalization for Hom-associative algebras introduced in~\cite{mp2}, we
arrive at the following concept and result:

\begin{Theorem}
\label{generalpseudotwistor} Let $(D, \mu , \tilde{\alpha }, \tilde{\beta })$
be a BiHom-associative algebra and $\alpha , \beta \colon D\rightarrow D$ two
multiplicative linear maps such that any two of the maps $\tilde{\alpha }$,
$\tilde{\beta }$, $\alpha$, $\beta $ commute. Let $T\colon D\otimes D\rightarrow
D\otimes D$ a~linear map and assume that there exist two linear maps $\tilde{T}_1, \tilde{T}_2\colon D\otimes D\otimes D \rightarrow D\otimes D\otimes D$ such
that the following relations hold:
\begin{gather}
 (\alpha \otimes \alpha )\circ T=T\circ (\alpha \otimes \alpha ),
\label{ghommultT1} \\
 (\beta \otimes \beta )\circ T=T\circ (\beta \otimes \beta ),
\label{ghommultT2} \\
 (\tilde{\alpha }\otimes \tilde{\alpha })\circ T=T\circ (\tilde{\alpha }
\otimes \tilde{\alpha }),  \label{ghommultT3} \\
 (\tilde{\beta }\otimes \tilde{\beta })\circ T=T\circ (\tilde{\beta }
\otimes \tilde{\beta }),  \label{ghommultT4} \\
 T\circ (\tilde{\alpha }\otimes \mu )= (\tilde{\alpha }\otimes \mu )\circ
\tilde{T}_1\circ (T\otimes \operatorname{id}_D),  \label{ghompstw1} \\
 T\circ (\mu \otimes \tilde{\beta })= (\mu \otimes \tilde{\beta })\circ
\tilde{T}_2\circ (\operatorname{id}_D\otimes T),  \label{ghompstw2} \\
 \tilde{T}_1\circ (T\otimes \operatorname{id}_D)\circ (\alpha \otimes T)= \tilde{T}_2\circ
(\operatorname{id}_D\otimes T)\circ (T\otimes \beta ).  \label{ghompstw3}
\end{gather}
Then $D^T_{\alpha , \beta }:=(D, \mu \circ T, \tilde{\alpha }\circ \alpha ,
\tilde{\beta }\circ \beta )$ is also a BiHom-associative algebra. The map $T$
is called an $(\alpha , \beta )$-BiHom-pseudotwistor and the two
maps $\tilde{T}_1$, $\tilde{T}_2$ are called the companions of~$T$.
In the particular case $\alpha =\beta =\operatorname{id}_D$, we simply call $T$ a \textit{BiHom-pseudotwistor} and we denote $D^T_{\alpha , \beta }$ by~$D^T$.
\end{Theorem}

\begin{proof}
The fact that $\tilde{\alpha }\circ \alpha $ and $\tilde{\beta }\circ \beta$
are multiplicative with respect to $\mu \circ T$ follows immediately from (\ref{ghommultT1})--(\ref{ghommultT4}) and the fact that $\alpha $, $\beta $, $\tilde{\alpha }$, $\tilde{\beta }$ are multiplicative with respect to~$\mu $. Now we prove the BiHom-associativity of $\mu \circ T$:
\begin{gather*}
(\mu \circ T)\circ ((\mu \circ T)\otimes (\tilde{\beta }\circ \beta ))
 = \mu \circ T\circ (\mu \otimes \tilde{\beta })\circ (T\otimes \beta ) \\
\qquad {} \overset{\eqref{ghompstw2}}{=}\mu \circ (\mu \otimes \tilde{\beta })\circ
\tilde{T}_2\circ (\operatorname{id}_D\otimes T)\circ (T\otimes \beta ) \\
\qquad {}
 \overset{\eqref{ghompstw3}}{=} \mu \circ (\mu \otimes \tilde{\beta })\circ
\tilde{T}_1\circ (T\otimes \operatorname{id}_D)\circ (\alpha \otimes T)
 = \mu \circ (\tilde{\alpha }\otimes \mu )\circ \tilde{T}_1\circ (T\otimes
\operatorname{id}_D)\circ (\alpha \otimes T) \\
\qquad {}
\overset{\eqref{ghompstw1}}{=}\mu \circ T\circ (\tilde{\alpha }\otimes \mu
) \circ (\alpha \otimes T)
=(\mu \circ T)\circ ((\tilde{\alpha }\circ \alpha )\otimes (\mu \circ T)),
\end{gather*}
f\/inishing the proof.
\end{proof}

Obviously, if $(D, \mu )$ is an associative algebra and $\tilde{\alpha }=\tilde{\beta }=\alpha =\beta =\operatorname{id}_D$, an $(\alpha , \beta )$-BiHom-pseudotwistor reduces to a pseudotwistor (as def\/ined in~\cite{lpvo})
and the BiHom-associative algebra~$D^T_{\alpha , \beta }$ is actually
associative.

We show now that Proposition~\ref{yaugeneral} is a particular case of
Theorem~\ref{generalpseudotwistor}.

\begin{Proposition}
Let $(D, \mu , \tilde{\alpha }, \tilde{\beta })$ be a BiHom-associative
algebra and $\alpha , \beta \colon D\rightarrow D$ two multiplicative linear maps
such that any two of the maps $\tilde{\alpha }$, $\tilde{\beta }$, $\alpha$,
$\beta $ commute. Define the maps
\begin{gather*}
 T\colon \ D\otimes D\rightarrow D\otimes D, \qquad T=\alpha \otimes \beta, \\
 \tilde{T}_1\colon \ D\otimes D\otimes D \rightarrow D\otimes D\otimes D, \qquad
\tilde{T}_1= \operatorname{id}_D\otimes \operatorname{id}_D\otimes \beta , \\
 \tilde{T}_2\colon \ D\otimes D\otimes D \rightarrow D\otimes D\otimes D, \qquad
\tilde{T}_2=\alpha \otimes \operatorname{id}_D\otimes \operatorname{id}_D.
\end{gather*}
Then $T$ is an $(\alpha , \beta )$-BiHom-pseudotwistor with companions $\tilde{T}_1$, $\tilde{T}_2$ and the BiHom-associative algebras $D^T_{\alpha
, \beta }$ and $D_{(\alpha , \beta )}$ coincide.
\end{Proposition}

\begin{proof}
The conditions (\ref{ghommultT1})--(\ref{ghommultT4}) are obviously
satisf\/ied. We check (\ref{ghompstw1}), for $a, b, c\in D$:
\begin{gather*}
\big((\tilde{\alpha }\otimes \mu )\circ \tilde{T}_1\circ (T\otimes
\operatorname{id}_D)\big)(a\otimes b\otimes c) =  \big((\tilde{\alpha }\otimes \mu )\circ \tilde{T}
_1\big)(\alpha (a)\otimes \beta (b)\otimes c) \\
\qquad {} = (\tilde{\alpha }\otimes \mu )(\alpha (a)\otimes \beta (b)\otimes \beta
(c))
 = (\tilde{\alpha }\circ \alpha )(a)\otimes \beta (bc)  \\
 \qquad{}
 = T(\tilde{\alpha }(a)\otimes bc)
 = (T\circ (\tilde{\alpha }\otimes \mu ))(a\otimes b\otimes c).
\end{gather*}
The condition (\ref{ghompstw2}) is similar, so we check (\ref{ghompstw3}):
\begin{gather*}
\big(\tilde{T}_1\circ (T\otimes \operatorname{id}_D)\circ (\alpha \otimes T)\big)(a\otimes b\otimes
c) =  \big(\tilde{T}_1\circ (T\otimes \operatorname{id}_D)\big)(\alpha (a)\otimes \alpha (b)\otimes
\beta (c)) \\
\qquad {} = \tilde{T}_1\big(\alpha ^2(a)\otimes \beta \alpha (b)\otimes \beta (c)\big)
 = \alpha ^2(a)\otimes \beta \alpha (b)\otimes \beta ^2(c)
 = \tilde{T}_2\big(\alpha (a)\otimes \alpha \beta (b)\otimes \beta ^2(c)\big) \\
\qquad{} = \big(\tilde{T}_2\circ (\operatorname{id}_D\otimes T)\big)(\alpha (a)\otimes \beta (b)\otimes
\beta (c))
  = \big(\tilde{T}_2\circ (\operatorname{id}_D\otimes T\big)\circ (T\otimes \beta ))(a\otimes
b\otimes c).
\end{gather*}
It is obvious that $D^T_{\alpha , \beta }$ and $D_{(\alpha , \beta )}$
coincide.
\end{proof}

\begin{Example}
We consider the 2-dimensional BiHom-associative algebra $(D,\mu,\tilde{\alpha},\tilde{\beta})$ def\/ined with respect to a basis $\mathcal{B}=\{e_1,e_2\}$
by
\begin{gather*}
 \mu(e_1,e_1)= \mu(e_1,e_2)=e_1,\qquad \mu(e_2,e_1)=\mu(e_2,e_2)=e_2, \\
  \tilde{\alpha} (e_1 )=e_1, \qquad \tilde{\alpha}(e_2)= e_2, \qquad \tilde{\beta}
(e_1 )=e_1, \qquad \tilde{\beta}(e_2)= e_1.
\end{gather*}
We have the following multiplicative linear maps $\alpha, \beta$ def\/ined
with respect to the basis $\mathcal{B}$ by
\begin{gather*}
\alpha (e_1 )=e_1, \qquad \alpha(e_2)=a e_1+(1-a) e_2, \qquad
\beta (e_1 )=e_1, \qquad \beta(e_2)= b e_1+(1-b) e_2,
\end{gather*}
where $a$, $b$ are parameters in $\Bbbk $. One can easily see that any two of
the maps $\tilde{\alpha }$, $\tilde{\beta }$, $\alpha$, $\beta $ commute. By the
previous proposition, we can construct the BiHom-associative algebra $
D_{(\alpha,\beta)}=(D, \mu _T=\mu \circ (\alpha \otimes \beta ), \alpha_T=
\tilde{\alpha }\circ \alpha , \beta_T=\tilde{\beta }\circ \beta )$ def\/ined
on the basis $\mathcal{B}$ by
\begin{gather*}
  \mu_T(e_1,e_1)=e_1,\qquad \mu_T(e_1,e_2)= e_1,\qquad \mu_T(e_2,e_1)=a e_1+
(1-a)e_2,\\ \mu_T(e_2,e_2)=a e_1+(1-a)e_2, \qquad
  \alpha_T(e_1)=e_1, \qquad \alpha_T(e_2)=a e_1+(1-a) e_2,\\ \beta_T(e_1)=e_1, \qquad
\beta_T(e_2)= e_1.
\end{gather*}
\end{Example}

\begin{Definition}[\cite{Cap,VanDaele}] Let $(A, \mu _A)$, $(B, \mu _B)$ be two
associative algebras. A \textit{twisting map} between $A$ and $B$ is a~linear map $R\colon B\otimes A \rightarrow A\otimes B$ satisfying the conditions
\begin{gather}
 R\circ (\operatorname{id}_B\otimes \mu _A)=(\mu _A\otimes \operatorname{id}_B)\circ (\operatorname{id}_A\otimes R)\circ
(R\otimes \operatorname{id}_A),  \label{twmap1} \\
 R\circ (\mu _B\otimes \operatorname{id}_A)=(\operatorname{id}_A\otimes \mu _B)\circ (R\otimes \operatorname{id}_B)\circ
(\operatorname{id}_B\otimes R).  \nonumber 
\end{gather}
If this is the case, the map $\mu _R=(\mu _A\otimes \mu _B)\circ
(\operatorname{id}_A\otimes R\otimes \operatorname{id}_B)$ is an associative product on $A\otimes B$; the
associative algebra $(A\otimes B, \mu _R)$ is denoted by $A\otimes _RB$ and
called the \textit{twisted tensor product} of~$A$ and $B$ af\/forded by~$R$.
\end{Definition}

We introduce now twisted tensor products of BiHom-associative algebras.

\begin{Definition}
Let $(A, \mu _A, \alpha _A, \beta _A)$ and $(B, \mu _B, \alpha _B, \beta _B)$
be two BiHom-associative algebras such that the maps $\alpha _A$, $\beta _A$,
$\alpha _B$, $\beta _B$ are bijective. A linear map $R\colon B\otimes A \rightarrow
A\otimes B$ is called a~\textit{BiHom-twisting map} between~$A$ and $B$ if
the following conditions are satisf\/ied
\begin{gather}
 (\alpha _A\otimes \alpha _B)\circ R=R\circ (\alpha _B\otimes \alpha _A),
\label{ghomtwmap01} \\
 (\beta _A\otimes \beta _B)\circ R=R\circ (\beta _B\otimes \beta _A),
\label{ghomtwmap02} \\
 R\circ (\alpha _B\otimes \mu _A)=(\mu _A\otimes \beta _B)\circ
(\operatorname{id}_A\otimes R)\circ \big(\operatorname{id}_A\otimes \alpha _B\beta _B^{-1} \otimes \operatorname{id}_A\big)\circ
(R\otimes \operatorname{id}_A),  \label{ghomtwmap1} \\
 R\circ (\mu _B\otimes \beta _A)=(\alpha _A\otimes \mu _B)\circ (R\otimes
\operatorname{id}_B)\circ \big(\operatorname{id}_B\otimes \alpha _A^{-1}\beta _A \otimes \operatorname{id}_B\big) \circ
(\operatorname{id}_B\otimes R).  \label{ghomtwmap2}
\end{gather}
\end{Definition}

If we use the standard Sweedler-type notation $R(b\otimes a)=a_R\otimes
b_R=a_r\otimes b_r$, for $a\in A$, $b\in B$, then the above conditions may
be rewritten (for all $a, a^{\prime }\in A$ and $b, b^{\prime }\in B$) as
follows
\begin{gather}
 \alpha _A(a_R)\otimes \alpha _B(b_R)=\alpha _A(a)_R\otimes \alpha _B(b)_R,
\label{ghomsweed01} \\
 \beta _A(a_R)\otimes \beta _B(b_R)=\beta _A(a)_R\otimes \beta _B(b)_R,
\label{ghomsweed02} \\
 (aa^{\prime })_R\otimes \alpha _B(b)_R=a_Ra^{\prime }_r\otimes \beta
_B\big(\big[\alpha _B\beta _B^{-1}(b_R)\big]_r\big),  \label{ghomsweed1} \\
 \beta _A(a)_R\otimes (bb^{\prime })_R=\alpha _A([\alpha _A^{-1}\beta
_A(a_R)]_r)\otimes b_rb^{\prime }_R.  \label{ghomsweed2}
\end{gather}

\begin{Proposition}\label{ghomttp} Let $(A, \mu _A, \alpha _A, \beta _A)$ and $(B, \mu _B,
\alpha _B, \beta _B)$ be two BiHom-associative algebras with bijective
structure maps, $R\colon B\otimes A \rightarrow A\otimes B$ a BiHom-twisting map.
Define the linear map
\begin{gather*}
 T\colon \ (A\otimes B)\otimes (A\otimes B)\rightarrow (A\otimes B)\otimes
(A\otimes B), \\
 T((a\otimes b)\otimes (a^{\prime }\otimes b^{\prime }))= (a\otimes
b_R)\otimes (a^{\prime }_R\otimes b^{\prime }).
\end{gather*}
Then $T$ is a BiHom-pseudotwistor for the tensor product $(A\otimes B, \mu
_{A\otimes B}, \alpha _A\otimes \alpha _B, \beta _A\otimes \beta _B)$ of~$A$
and~$B$, with companions
\begin{gather*}
 \tilde{T}_1=\big(\operatorname{id}_A\otimes \alpha _B^{-1}\beta _B\otimes \operatorname{id}_A\otimes
\operatorname{id}_B\otimes \operatorname{id}_A\otimes \operatorname{id}_B\big)\circ T_{13} \\
\hphantom{\tilde{T}_1=}{} \circ (\operatorname{id}_A\otimes \alpha
_B\beta _B^{-1}\otimes \operatorname{id}_A\otimes \operatorname{id}_B\otimes \operatorname{id}_A\otimes \operatorname{id}_B), \\
 \tilde{T}_2=\big(\operatorname{id}_A\otimes \operatorname{id}_B\otimes \operatorname{id}_A\otimes \operatorname{id}_B\otimes \alpha
_A\beta _A^{-1}\otimes \operatorname{id}_B\big)\circ T_{13} \\
\hphantom{\tilde{T}_2=}{}
\circ (\operatorname{id}_A\otimes \operatorname{id}_B\otimes
\operatorname{id}_A\otimes \operatorname{id}_B\otimes \alpha _A^{-1}\beta _A\otimes \operatorname{id}_B),
\end{gather*}
where we use the standard notation for~$T_{13}$. The BiHom-associative
algebra $(A\otimes B)^T$ is denoted by $A\otimes _RB$ and is called the
\textit{BiHom-twisted tensor product} of~$A$ and~$B$; its multiplication is
defined by $(a\otimes b)(a^{\prime }\otimes b^{\prime })=aa^{\prime
}_R\otimes b_Rb^{\prime }$, and the structure maps are $\alpha _A\otimes
\alpha _B$ and $\beta _A\otimes \beta _B$.
\end{Proposition}

\begin{proof}
We begin by proving the following relation, for all $a\in A$, $b\in B$:
\begin{gather}
 \alpha _B^{-1}\beta _B\big(\big[\alpha _B\beta _B^{-1}(b)\big]_R\big)\otimes a_R=
b_R\otimes \alpha _A\beta _A^{-1}\big(\big[\alpha _A^{-1}\beta _A(a)\big]_R\big).\label{helpful}
\end{gather}
This relation is equivalent to
\begin{gather*}
 \beta _B\big(\big[\alpha _B\beta _B^{-1}(b)\big]_R\big)\otimes \beta _A(a_R)= \alpha
_B(b_R)\otimes \alpha _A\big(\big[\alpha _A^{-1}\beta _A(a)\big]_R\big),
\end{gather*}
which, by using (\ref{ghomsweed01}) and (\ref{ghomsweed02}), is equivalent
to
\begin{gather*}
\alpha _B(b)_R\otimes \beta _A(a)_R=\alpha _B(b)_R\otimes \beta _A(a)_R,
\end{gather*}
which is obviously true.

We need to prove the relations (\ref{ghommultT1})--(\ref{ghompstw3}) (with $\tilde{\alpha }=\alpha _A\otimes \alpha _B$, $\tilde{\beta }=\beta _A\otimes
\beta _B$, $\alpha =\beta =\operatorname{id}_A\otimes \operatorname{id}_B$). We will prove only~(\ref{ghompstw3}), while (\ref{ghommultT1})--(\ref{ghompstw2}) are very easy and
left to the reader. We compute ($r$~and~$\mathcal{R}$ are two more copies of
$R$)
\begin{gather*}
\tilde{T}_1\circ (T\otimes \operatorname{id})\circ (\operatorname{id}\otimes T)(a\otimes
b\otimes a^{\prime }\otimes b^{\prime }\otimes a^{\prime \prime }\otimes
b^{\prime \prime }) =\tilde{T}_1(a\otimes b_r\otimes a^{\prime }_r\otimes b^{\prime }_R\otimes
a^{\prime \prime }_R\otimes b^{\prime \prime }) \\
\qquad {} = a\otimes \alpha _B^{-1}\beta _B\big(\big[\alpha _B\beta _B^{-1}(b_r)\big]_{\mathcal{R}}\big) \otimes a^{\prime }_r\otimes b^{\prime }_R\otimes (a^{\prime \prime }_R)_{\mathcal{R}}\otimes b^{\prime \prime },
\\
\tilde{T}_2\circ (\operatorname{id}\otimes T)\circ (T\otimes \operatorname{id})(a\otimes
b\otimes a^{\prime }\otimes b^{\prime }\otimes a^{\prime \prime }\otimes
b^{\prime \prime })
 = \tilde{T}_2(a\otimes b_r\otimes a^{\prime }_r\otimes b^{\prime }_R\otimes
a^{\prime \prime }_R\otimes b^{\prime \prime }) \\
\qquad{} = a\otimes (b_r)_{\mathcal{R}}\otimes a^{\prime }_r\otimes b^{\prime
}_R\otimes \alpha _A \beta _A^{-1}\big(\big[\alpha _A^{-1}\beta _A(a^{\prime \prime
}_R)\big]_{\mathcal{R}}\big)\otimes b^{\prime \prime },
\end{gather*}
and the two terms are equal because of the relation (\ref{helpful}).
\end{proof}

\begin{Remark}
Let $(A, \mu _A, \alpha _A, \beta _A)$ and $(B, \mu _B, \alpha _B, \beta _B)$
be two BiHom-associative algebras with bijective structure maps. Then
obviously the linear map $R\colon B\otimes A \rightarrow A\otimes B$, $R(b\otimes
a)=a\otimes b$, is a BiHom-twisting map and the BiHom-twisted tensor product
$A\otimes _RB$ coincides with the ordinary tensor product $A\otimes B$.
\end{Remark}

\begin{Proposition}
\label{gendefttp} Let $(A, \mu _A)$ and $(B, \mu _B)$ be two associative
algebras, $\alpha _A, \beta _A\colon A\rightarrow A$ two commuting algebra
isomorphisms of $A$ and $\alpha _B, \beta _B\colon B\rightarrow B$ two commuting
algebra isomorphisms of $B$. Let $P\colon B\otimes A\rightarrow A\otimes B$ be a
twisting map satisfying the conditions
\begin{gather}
 (\alpha _A\otimes \alpha _B)\circ P=P\circ (\alpha _B\otimes \alpha _A),
\label{cond1P} \\
 (\beta _A\otimes \beta _B)\circ P=P\circ (\beta _B\otimes \beta _A).
\label{cond2P}
\end{gather}
Define the linear map
\begin{gather*}
 U\colon \ B\otimes A\rightarrow A\otimes B, \qquad U(b\otimes a)=\beta
_A^{-1}(\beta _A(a)_P)\otimes \alpha _B^{-1}(\alpha _B(b)_P).
\end{gather*}
Then $U$ is a BiHom-twisting map between the BiHom-associative algebras $A_{(\alpha _A, \beta _A)}$ and $B_{(\alpha _B, \beta _B)}$ and the
BiHom-associative algebras $A_{(\alpha _A, \beta _A)}\otimes _UB_{(\alpha
_B, \beta _B)}$ and $(A\otimes _PB)_{(\alpha _A\otimes \alpha _B, \beta
_A\otimes \beta _B)}$ coincide.
\end{Proposition}

\begin{proof}
We only prove (\ref{ghomsweed1}) for $U$ and leave the rest to the reader.
We compute (by denoting by $p$ another copy of $P$ and by $u$ another copy
of $U$)
\begin{gather*}
(aa^{\prime })_U\otimes \alpha _B(b)_U = [\alpha _A(a)\beta _A(a^{\prime
})]_U\otimes \alpha _B(b)_U
 = \beta _A^{-1}\big(\big[\beta _A\alpha _A(a)\beta _A^2(a^{\prime })\big]_P\big)\otimes
\alpha _B^{-1}(\alpha _B^2(b)_P) \\
\hphantom{(aa^{\prime })_U\otimes \alpha _B(b)_U}{}
 \overset{\eqref{twmap1}}{=} \beta _A^{-1}(\beta _A\alpha _A(a)_P)\beta
_A^{-1}\big(\beta _A^2(a^{\prime })_p\big) \otimes \alpha _B^{-1}((\alpha
_B^2(b)_P)_p) \\
\hphantom{(aa^{\prime })_U\otimes \alpha _B(b)_U}{} = \beta _A^{-1}(\alpha _A(\beta _A(a))_P)\beta _A^{-1}\big(\beta _A^2(a^{\prime
})_p\big) \otimes \alpha _B^{-1}([\alpha _B(\alpha _B(b))_P]_p) \\
\hphantom{(aa^{\prime })_U\otimes \alpha _B(b)_U}{} \overset{\eqref{cond1P}}{=} \beta _A^{-1}\alpha _A(\beta _A(a)_P)\beta
_A^{-1}\big(\beta _A^2(a^{\prime })_p\big) \otimes \alpha _B^{-1}(\alpha _B(\alpha
_B(b)_P)_p),
\\
a_Ua^{\prime }_u\otimes \beta _B\big(\big[\alpha _B\beta _B^{-1}(b_U)\big]_u\big) =  \alpha
_A(a_U)\beta _A(a^{\prime }_u)\otimes \beta _B\big(\big[\alpha _B\beta
_B^{-1}(b_U)\big]_u\big) \\
\hphantom{a_Ua^{\prime }_u\otimes \beta _B\big(\big[\alpha _B\beta _B^{-1}(b_U)\big]_u\big)}{}
 = \alpha _A\beta _A^{-1}(\beta _A(a)_P)\beta _A(a^{\prime }_u)\otimes \beta
_B\big(\big[\beta _B^{-1}(\alpha _B(b)_P)\big]_u\big) \\
\hphantom{a_Ua^{\prime }_u\otimes \beta _B\big(\big[\alpha _B\beta _B^{-1}(b_U)\big]_u\big)}{}
 = \alpha _A\beta _A^{-1}(\beta _A(a)_P)\beta _A(a^{\prime })_p\otimes
\alpha _B^{-1}\beta _B\big(\big[\alpha _B\beta _B^{-1}(\alpha _B(b)_P)\big]_p\big) \\
\hphantom{a_Ua^{\prime }_u\otimes \beta _B\big(\big[\alpha _B\beta _B^{-1}(b_U)\big]_u\big)}{}
 = \alpha _A\beta _A^{-1}(\beta _A(a)_P)\beta _A(a^{\prime })_p\otimes
\alpha _B^{-1}\beta _B\big(\beta _B^{-1}(\alpha _B(\alpha _B(b)_P))_p\big) \\
\hphantom{a_Ua^{\prime }_u\otimes \beta _B\big(\big[\alpha _B\beta _B^{-1}(b_U)\big]_u\big)}{}
 = \alpha _A\beta _A^{-1}(\beta _A(a)_P)\beta _A^{-1}\big(\beta _A^2(a^{\prime
})\big)_p\otimes \alpha _B^{-1}\beta _B\big(\beta _B^{-1}(\alpha _B(\alpha
_B(b)_P))_p\big) \\
\hphantom{a_Ua^{\prime }_u\otimes \beta _B\big(\big[\alpha _B\beta _B^{-1}(b_U)\big]_u\big)}{}
 \overset{\eqref{cond2P}}{=} \alpha _A\beta _A^{-1}(\beta _A(a)_P)\beta
_A^{-1}\big(\beta _A^2(a^{\prime })_p\big) \otimes \alpha _B^{-1}(\alpha _B(\alpha
_B(b)_P)_p),
\end{gather*}
f\/inishing the proof.
\end{proof}

\section{BiHom-smash products}

We construct f\/irst a large family of BiHom-twisting maps.

\begin{Theorem}
\label{BihomTwitMapR} Let $(H, \mu _H, \Delta _H, \alpha _H, \beta _H, \psi
_H, \omega _H)$ be a BiHom-bialgebra, $(A, \mu _A, \alpha _A, \beta _A)$ a
left $H$-module BiHom-algebra, with action denoted by $H\otimes A\rightarrow
A$, $h\otimes a\mapsto h\cdot a$, and assume that all structure maps $\alpha
_H, \beta _H, \psi _H, \omega _H, \alpha _A, \beta _A$ are bijective. Let $m, n, p\in \mathbb{Z}$. Define the linear map
\begin{gather*}  
 R_{m, n, p}\colon \ H\otimes A\rightarrow A\otimes H, \qquad R_{m, n, p}(h\otimes
a)= \alpha _H^m\beta _H^n\omega _H^p(h_1) \cdot \beta _A^{-1}(a)\otimes \psi
_H^{-1}(h_2).
\end{gather*}
Then $R_{m, n, p}$ is a BiHom-twisting map between $A$ and $H$.
\end{Theorem}

\begin{proof}
The relations (\ref{ghomtwmap01}) and (\ref{ghomtwmap02}) are very easy to
prove and left to the reader.

Proof of (\ref{ghomtwmap1}):
\begin{gather*}
(\mu _A\otimes \beta _H)\circ (\operatorname{id}_A\otimes R_{m, n, p})\circ
\big(\operatorname{id}_A\otimes \alpha _H\beta _H^{-1}\otimes \operatorname{id}_A\big) \circ (R_{m, n, p}\otimes
\operatorname{id}_A)(h\otimes a\otimes a^{\prime })\\
\qquad{} =(\mu _A\otimes \beta _H)\circ (\operatorname{id}_A\otimes R_{m, n, p})\big(\alpha _H^m\beta
_H^n\omega _H^p(h_1)\cdot \beta _A^{-1}(a)\otimes \alpha _H\beta _H^{-1}\psi
_H^{-1}(h_2)\otimes a^{\prime }\big) \\
\qquad{} = (\mu _A\otimes \beta _H)\big(\alpha _H^m\beta _H^n\omega _H^p(h_1)\cdot \beta
_A^{-1}(a)\otimes \alpha _H^m\beta _H^n\omega _H^p\big(\big[\alpha _H\beta
_H^{-1}\psi _H^{-1}(h_2)\big]_1\big)\cdot \beta _A^{-1}(a^{\prime }) \\
\qquad\quad {}  \otimes \psi _H^{-1}\big(\big[\alpha _H\beta _H^{-1}\psi _H^{-1}(h_2)\big]_2\big)\big) \\
\qquad{} = (\mu _A\otimes \beta _H)\big(\alpha _H^m\beta _H^n\omega _H^p(h_1)\cdot \beta
_A^{-1}(a)\otimes \alpha _H^{m+1}\beta _H^{n-1}\omega _H^p\psi
_H^{-1}((h_2)_1)\cdot \beta _A^{-1}(a^{\prime }) \\
\qquad\quad{} \otimes \alpha _H\beta _H^{-1}\psi _H^{-2}((h_2)_2)\big) \\
\qquad{} = \big[\alpha _H^m\beta _H^n\omega _H^p(h_1)\cdot \beta _A^{-1}(a)\big] \big[\alpha
_H^{m+1}\beta _H^{n-1}\psi _H^{-1}\omega _H^p((h_2)_1)\cdot \beta
_A^{-1}(a^{\prime })\big]\otimes \alpha _H\psi _H^{-2}((h_2)_2) \\
\qquad{} \overset{(\ref{ghombia1})}{=} \big[\alpha _H^m\beta _H^n\omega
_H^{p-1}((h_1)_1) \cdot \beta _A^{-1}(a)\big] \big[\alpha _H^{m+1}\beta _H^{n-1}\psi
_H^{-1}\omega _H^p((h_1)_2)\cdot \beta _A^{-1}(a^{\prime })\big]\otimes \alpha
_H\psi _H^{-1}(h_2) \\
\qquad{} = \big[\alpha_H^{-1}\omega _H^{-1}\big(\alpha _H^{m+1}\beta _H^n\omega
_H^{p}((h_1)_1)\big) \cdot \beta _A^{-1}(a)\big] \big[\beta _H^{-1}\psi _H^{-1}\big(\alpha
_H^{m+1}\beta _H^{n}\omega _H^p((h_1)_2)\big)\cdot \beta _A^{-1}(a^{\prime })\big] \\
\qquad\quad{} \otimes \alpha _H\psi _H^{-1}(h_2) \\
\qquad{} = \big\{\alpha_H^{-1}\omega _H^{-1}\big(\big[\alpha _H^{m+1}\beta _H^n\omega
_H^{p}(h_1)\big]_1\big) \cdot \beta _A^{-1}(a)\big\} \big\{\beta _H^{-1}\psi _H^{-1}\big(\big[\alpha
_H^{m+1}\beta _H^{n}\omega _H^p(h_1)\big]_2\big)\cdot \beta _A^{-1}(a^{\prime })\big\} \\
\qquad\quad{}\otimes \alpha _H\psi _H^{-1}(h_2) \\
\qquad{} \overset{(\ref{gmodalgcompat})}{=} \alpha _H^{m+1}\beta _H^n\omega
_H^{p}(h_1)\cdot \beta _A^{-1}(aa^{\prime })\otimes \alpha _H\psi
_H^{-1}(h_2)  = (R_{m, n, p}\circ (\alpha _H\otimes \mu _A))(h\otimes a\otimes a^{\prime}).
\end{gather*}

 Proof of (\ref{ghomtwmap2}):
\begin{gather*}
(\alpha _A\otimes \mu _H)\circ (R_{m, n, p}\otimes \operatorname{id}_H)\circ
\big(\operatorname{id}_H\otimes \alpha _A^{-1}\beta _A \otimes \operatorname{id}_H\big)\circ (\operatorname{id}_H\otimes R_{m, n,
p})(h\otimes h^{\prime }\otimes a)\\
\qquad{}
 = (\alpha _A\otimes \mu _H)\circ (R_{m, n, p}\otimes \operatorname{id}_H)(h\otimes \alpha
_A^{-1}\beta _A \big(\alpha _H^m\beta _H^n\omega _H^p(h^{\prime }_1)\cdot \beta
_A^{-1}(a)\big)\otimes \psi _H^{-1}(h^{\prime }_2)) \\
\qquad {} = (\alpha _A\otimes \mu _H)\circ (R_{m, n, p}\otimes \operatorname{id}_H)\big(h\otimes \alpha
_H^{m-1}\beta _H^{n+1}\omega _H^p(h^{\prime }_1)\cdot \alpha
_A^{-1}(a)\otimes \psi _H^{-1}(h^{\prime }_2)\big) \\
\qquad {}
 = (\alpha _A\otimes \mu _H)\big(\alpha _H^m\beta _H^n\omega _H^p(h_1)\cdot
\big(\alpha _H^{m-1}\beta _H^{n} \omega _H^p(h^{\prime }_1)\cdot \alpha
_A^{-1}\beta _A^{-1}(a)\big)\otimes \psi _H^{-1}(h_2)\otimes \psi
_H^{-1}(h^{\prime }_2)\big) \\
\qquad {}
 = \alpha _H^{m+1}\beta _H^n\omega _H^p(h_1)\cdot \big(\alpha _H^{m}\beta _H^{n}
\omega _H^p(h^{\prime }_1)\cdot \beta _A^{-1}(a)\big)\otimes \psi
_H^{-1}(h_2h^{\prime }_2) \\
\qquad{} \overset{(\ref{ghommod2})}{=} \big\{\big[\alpha _H^{m}\beta _H^n\omega
_H^p(h_1)\big]\big[\alpha _H^{m}\beta _H^{n} \omega _H^p(h^{\prime }_1)\big]\big\}\cdot
a\otimes \psi _H^{-1}(h_2h^{\prime }_2) \\
\qquad{} = \alpha _H^{m}\beta _H^n\omega _H^p(h_1h^{\prime }_1)\cdot a\otimes \psi
_H^{-1}(h_2h^{\prime }_2) \\
\qquad{} \overset{(\ref{ghombia2})}{=} \alpha _H^{m}\beta _H^n\omega
_H^p((hh^{\prime })_1)\cdot a\otimes \psi _H^{-1}((hh^{\prime })_2)
 = (R_{m, n, p}\circ (\mu _H\otimes \beta _A))(h\otimes h^{\prime }\otimes
a),
\end{gather*}
f\/inishing the proof.
\end{proof}

\begin{Definition}
Let $(H, \mu _H, \Delta _H, \alpha _H, \beta _H, \psi _H, \omega _H)$ be a
BiHom-bialgebra and $(A, \mu _A, \alpha _A, \beta _A)$ a~left $H$-module
BiHom-algebra, with left $H$-module structure $H\otimes A\rightarrow A$, $h\otimes a\mapsto h\cdot a$, such that all structure maps $\alpha _H$, $\beta
_H$, $\psi _H$, $\omega _H$, $\alpha _A$, $\beta _A$ are bijective. Consider the
BiHom-twisting map
\begin{gather}
 R=R_{0, -1, -1}\colon \ H\otimes A\rightarrow A\otimes H, \qquad R(h\otimes
a)=\beta _H^{-1}\omega _H^{-1}(h_1) \cdot \beta _A^{-1}(a)\otimes \psi
_H^{-1}(h_2).  \label{Rsmash}
\end{gather}
We denote the BiHom-associative algebra $A\otimes _RH$ by $A\# H$ (we denote
$a\otimes h:=a\# h$, for $a\in A$, $h\in H$) and call it the \textit{BiHom-smash product} of~$A$ and~$H$. Its structure maps are $\alpha
_A\otimes \alpha _H$ and $\beta _A\otimes \beta _H$, and its multiplication
is
\begin{gather*}
 (a\# h)(a^{\prime }\# h^{\prime })=a\big(\beta _H^{-1}\omega _H^{-1}(h_1)
\cdot \beta _A^{-1}(a^{\prime })\big)\# \psi _H^{-1}(h_2)h^{\prime }.
\end{gather*}
\end{Definition}

\begin{Remark}
If $H$ is a Hom-bialgebra, i.e., $\alpha _H=\beta _H=\psi _H=\omega _H$, and $A$ is a Hom-associative algebra, the multiplication of $A\# H$ becomes
\begin{gather*}
 (a\# h)(a^{\prime }\# h^{\prime })=a\big(\alpha _H^{-2}(h_1) \cdot \alpha
_A^{-1}(a^{\prime })\big)\# \alpha_H^{-1}(h_2)h^{\prime },
\end{gather*}
which is the formula introduced in \cite{mp2}. If $H$ is a monoidal
Hom-bialgebra, i.e., $\psi _H=\omega _H=\alpha _H^{-1}=\beta _H^{-1}$, and~$A$
is a Hom-associative algebra, the multiplication of $A\# H$ becomes
\begin{gather*}
 (a\# h)(a^{\prime }\# h^{\prime })=a\big(h_1 \cdot \alpha _A^{-1}(a^{\prime
})\big)\# \alpha_H(h_2)h^{\prime },
\end{gather*}
which is the formula introduced in \cite{chenwangzhang}, used also in~\cite{LB} for def\/ining the Radford biproduct for monoidal Hom-bialgebras.
\end{Remark}

\begin{Proposition}
In the same setting as in Proposition~{\rm \ref{gyaumodalg}}, and assuming
moreover that the maps $\alpha _A$ and $\beta _A$ are bijective, if we
denote by $A\# H$ the usual smash product between $A$ and $H$, then $\alpha
_A\otimes \alpha _H$ and $\beta _A\otimes \beta _H$ are commuting algebra
endomorphisms of $A\# H$ and the BiHom-associative algebras $(A\#
H)_{(\alpha _A\otimes \alpha _H, \beta _A\otimes \beta _H)}$ and $A_{(\alpha
_A, \beta _A)}\# H_{(\alpha _H, \beta _H, \psi _H, \omega _H)}$ coincide.
\end{Proposition}

\begin{proof}
We will apply Proposition~\ref{gendefttp}. In our situation, we have the
twisting map $P\colon H\otimes A \rightarrow A\otimes H$, $P(h\otimes a)=h_1\cdot
a\otimes h_2$, for which $A\# H=A\otimes _PH$. Obviously $P$ satisf\/ies the
condi\-tions~(\ref{cond1P}) and~(\ref{cond2P}), so, by Proposition~\ref{gendefttp}, we obtain the map
\begin{gather*}
 U\colon \ H\otimes A\rightarrow A\otimes H, \qquad U(h\otimes a)=\beta
_A^{-1}(\beta _A(a)_P)\otimes \alpha _H^{-1}(\alpha _H(h)_P),
\end{gather*}
which is a BiHom-twisting map between $A_{(\alpha _A, \beta _A)}$ and $H_{(\alpha _H, \beta _H)}$ and we have
\begin{gather*}
 (A\# H)_{(\alpha _A\otimes \alpha _H, \beta _A\otimes \beta _H)}=
A_{(\alpha _A, \beta _A)}\otimes _UH_{(\alpha _H, \beta _H)}.
\end{gather*}
Thus, the proof will be f\/inished if we prove that the map $U$ coincides with
the map $R$ af\/fording the BiHom-smash product $A_{(\alpha _A, \beta _A)}\#
H_{(\alpha _H, \beta _H, \psi _H, \omega _H)}$. We compute
\begin{gather*}
U(h\otimes a) = \beta _A^{-1}(\alpha _H(h)_1\cdot \beta _A(a))\otimes \alpha
_H^{-1}(\alpha _H(h)_2) \\
\hphantom{U(h\otimes a)}{}
=\beta _A^{-1}(\alpha _H(h_1)\cdot \beta _A(a))\otimes \alpha
_H^{-1}(\alpha _H(h_2))
=\alpha _H\beta _H^{-1}(h_1)\cdot a\otimes h_2,
\\
R(h\otimes a) = \beta _H^{-1}\omega _H^{-1}(\omega _H(h_1))\triangleright
\beta _A^{-1}(a) \otimes \psi _H^{-1}(\psi _H(h_2)) \\
\hphantom{R(h\otimes a)}{}
=\beta _H^{-1}(h_1)\triangleright \beta _A^{-1}(a)\otimes h_2
=\alpha _H\beta _H^{-1}(h_1)\cdot a\otimes h_2,
\end{gather*}
f\/inishing the proof.
\end{proof}

\begin{Example}
We construct a class of examples of $U_{q}(\mathfrak{sl}_{2})_{(\alpha
,\beta ,\psi ,\omega )}$-module BiHom-algebra structures on $\mathbb{A}_{q,\alpha ,\beta }^{2|0}$, generalizing examples of $U_{q}(\mathfrak{sl}_{2})_{\alpha }$-module Hom-algebra structures on $\mathbb{A}_{q,\gamma
}^{2|0}$ given in \cite[Example~5.7]{homquantum3} (here we take the base
f\/ield $\Bbbk =\mathbb{C}$). The quantum group $U_{q}(\mathfrak{sl}_{2})$ is
generated as a unital associative algebra by 4 generators $\{E,F,K,K^{-1}\}$
with relations
\begin{gather*}
KK^{-1}=1=K^{-1}K, \qquad
KE=q^{2}EK,\qquad KF=q^{-2}FK, \qquad
EF-FE=\frac{K-K^{-1}}{q-q^{-1}},
\end{gather*}
where $q\in \mathbb{C}$ with $q\neq 0$, $q\neq \pm 1$. The comultiplication
is def\/ined by
\begin{gather*}
 \Delta (E)=1\otimes E+E\otimes K, \qquad
 \Delta (F)=K^{-1}\otimes F+F\otimes 1, \\
 \Delta (K)=K\otimes K,\qquad \Delta \big(K^{-1}\big)=K^{-1}\otimes K^{-1}.
\end{gather*}

We f\/ix $\lambda_1,\lambda_2,\lambda_3,\lambda_4 \in \mathbb{C}$ some nonzero
elements. The BiHom-bialgebra $U_q(\mathfrak{sl}_2)_{(\alpha,\beta,\psi,\omega )}= (U_q(\mathfrak{sl}_2),\mu_{(\alpha,\beta )},\Delta_{(\psi,\omega
)},\alpha,\beta,\psi,\omega)$ is def\/ined (as in Proposition~\ref{yautwistdiverse}(iii)) by $\mu_{(\alpha,\beta )}=\mu\circ (\alpha\otimes
\beta)$ and $\Delta_{(\psi,\omega )} =(\omega\otimes \psi)\circ\Delta$,
where $\mu$ and $\Delta$ are respectively the multiplication and
comultiplication of $U_q(\mathfrak{sl}_2)$ and $\alpha,\beta,\psi,\omega\colon U_q(\mathfrak{sl}_2)\rightarrow U_q(\mathfrak{sl}_2)$ are bialgebra morphisms
such that
\begin{gather*}
\alpha (E)=\lambda_1 E, \qquad \alpha (F)=\lambda_1^{-1} F, \qquad \alpha (K)=K, \qquad
\alpha \big(K^{-1}\big)=K^{-1}, \\
\beta (E)=\lambda_2 E, \qquad \beta (F)=\lambda_2^{-1} F, \qquad \beta (K)=K, \qquad
\beta \big(K^{-1}\big)=K^{-1}, \\
\psi (E)=\lambda_3 E, \qquad \psi (F)=\lambda_3^{-1} F, \qquad \psi (K)=K, \qquad \psi
\big(K^{-1}\big)=K^{-1}, \\
\omega (E)=\lambda_4 E, \qquad \omega (F)=\lambda_4^{-1} F, \qquad \omega (K)=K, \qquad
\omega \big(K^{-1}\big)=K^{-1}.
\end{gather*}
Note that any two of the maps $\alpha$, $\beta$, $\psi$, $\omega$ commute.

Let $\mathbb{A}_{q}^{2|0}=k\langle x,y\rangle /(yx-q xy)$ be the quantum
plane. We f\/ix also some $\xi \in \mathbb{C}$, $\xi \neq 0$. The
BiHom-quantum plane $\mathbb{A}_{q,\alpha,\beta}^{2|0}=(\mathbb{A}_{q}^{2|0},\mu_{\mathbb{A},\alpha_\mathbb{A},\beta_\mathbb{A}},\alpha_\mathbb{A},\beta_\mathbb{A})$ is the BiHom-associative algebra def\/ined (as
in Proposition~\ref{yautwistdiverse}(i)) by $\mu_{\mathbb{A},\alpha_\mathbb{A},\beta_\mathbb{A}}=\mu_{\mathbb{A}}\circ(\alpha_\mathbb{A}\otimes\beta_\mathbb{A})$, where $\mu_\mathbb{A}$ is the multiplication of $\mathbb{A}_{q}^{2|0}$ and $\alpha_\mathbb{A},\beta_\mathbb{A}\colon  \mathbb{A}_{q}^{2|0}\rightarrow \mathbb{A}_{q}^{2|0} $ are the (commuting) algebra
morphisms such that
\begin{gather*}
\alpha_\mathbb{A}(x)=\xi x,\qquad \alpha_\mathbb{A}(y)= \xi \lambda_1^{-1} y
\qquad \text{and} \qquad \beta_\mathbb{A}(x)=\xi x,\qquad \beta_\mathbb{A} (y)= \xi
\lambda_2^{-1} y.
\end{gather*}

We consider $\mathbb{A}_{q}^{2|0}$ as a left $U_q(\mathfrak{sl}_2)$-module
algebra as in \cite[Example~5.7]{homquantum3} (we denote by $h\otimes
a\mapsto h\cdot a$ the $U_q(\mathfrak{sl}_2)$-action on $\mathbb{A}_{q}^{2|0} $). By the computations performed in \cite[Example~5.7]{homquantum3} we know that $\alpha_\mathbb{A}(h\cdot a)=\alpha(h)\cdot
\alpha_\mathbb{A}( a)$ and $\beta_\mathbb{A}(h\cdot a)=\beta(h)\cdot \beta_\mathbb{A}( a)$, for all $h\in U_q(\mathfrak{sl}_2)$ and $a\in \mathbb{A}_{q}^{2|0}$. Then, according to Proposition~\ref{gyaumodalg}, there exists a
$U_q(\mathfrak{sl}_2)_{(\alpha,\beta,\psi,\omega )}$-module BiHom-algebra
structure on $\mathbb{A}_{q,\alpha , \beta}^{2|0}$ def\/ined by
\begin{gather*}
 \rho \colon \  U_q(\mathfrak{sl}_2)_{(\alpha,\beta,\psi,\omega )}\otimes \mathbb{A}_{q,\alpha , \beta}^{2|0} \rightarrow \mathbb{A}_{q,\alpha , \beta}^{2|0},
\qquad  \rho (h\otimes a)=h\triangleright a=\alpha (h)\cdot \beta_\mathbb{A}(a).
\end{gather*}
By using also the computations performed in \cite[Example 5.7]{homquantum3}
one can see that the map $\rho $ is given on generators by
\begin{gather*}
  \rho \big(E\otimes x^my^n \big)=[n]_q\xi^{m+n}\lambda_1\lambda_2^{-n}
x^{m+1}y^{n-1}, \\
  \rho \big(F \otimes x^my^n \big)=[m]_q\xi^{m+n}\lambda_1^{-1}\lambda_2^{-n}
x^{m-1}y^{n+1}, \\
  \rho\big(K^{\pm 1}\otimes P \big)=P\big(q^{\pm 1}\xi x,q^{\mp 1} \xi\lambda_2^{-1}y\big),
\end{gather*}
for any monomial $x^my^n\in \mathbb{A}_{q}^{2|0}$, where $P=P(x,y)\in
\mathbb{A}_{q}^{2|0}$ and $[n]_q=\frac{q^n-q^{-n}}{q-q^{-1}}$.

Since $\xi \neq 0$ and $\lambda _i\neq 0$ for all $i=1, 2, 3, 4$, all the
maps $\alpha$, $\beta$, $\psi$, $\omega$, $\alpha_\mathbb{A}$, $\beta_\mathbb{A}$
are bijective. According to Theorem~\ref{BihomTwitMapR}, the map $R\colon U_q(\mathfrak{sl}_2)_{(\alpha,\beta,\psi,\omega )}\otimes \mathbb{A}_{q,\alpha,\beta}^{2|0} \rightarrow \mathbb{A}_{q,\alpha,\beta}^{2|0}\otimes
U_q(\mathfrak{sl}_2)_{(\alpha,\beta,\psi,\omega )} $ def\/ined by~\eqref{Rsmash} leads to the smash product $\mathbb{A}_{q,\alpha,\beta}^{2|0}\# U_q(\mathfrak{sl}_2)_{(\alpha,\beta,\psi,\omega )} $ whose multiplication
is def\/ined by
\begin{gather*}
 (a\# h)(a'\# h')=a*\big(\beta ^{-1}\omega ^{-1}(h_{(1)})
\triangleright \beta _\mathbb{A}^{-1}(a')\big)\#
\psi ^{-1}(h_{(2)})\bullet h',
\end{gather*}
where $h_{(1)}\otimes h_{(2)}=\Delta_{(\psi,\omega )}(h)$ and $*$
(respectively $\bullet $) is the multiplication of $\mathbb{A}_{q,\alpha,\beta}^{2|0}$ (respectively $U_q(\mathfrak{sl}_2)_{(\alpha,\beta,\psi,\omega )}$).

In particular, for any $G\in U_q(\mathfrak{sl}_2)$ and $m,n,r,s\in \mathbb{N}
$ we have
\begin{gather*}
 (x^my^n\# K^{\pm 1})(x^ry^s\# G)=q^{\pm r\mp s+n
r}\xi^{m+n+r+s}\lambda_1^{-n}\lambda_2^{-s}x^{m+r}y^{n+s} \# K ^{\pm
1}\beta(G), \\
 (x^my^n\# E)(x^ry^s\# G)=q^{n
r}\xi^{m+n+r+s}\lambda_1^{-n+1}\lambda_2^{-s}x^{m+r}y^{n+s}\# E \beta(G) \\
\hphantom{(x^my^n\# E)(x^ry^s\# G)=}{}
+[s]_q q^{n (r+1)} \xi ^{m+n+r+s}\lambda_1^{1-n}\lambda_2^{-s}
x^{m+r+1}y^{n+s-1}\# K \beta(G), \\
 (x^my^n\# F)(x^ry^s\# G)=q^{s-r+n
r}\xi^{m+n+r+s}\lambda_1^{-n-1}\lambda_2^{-s}x^{m+r}y^{n+s}\# F \beta(G) \\
\hphantom{(x^my^n\# F)(x^ry^s\# G)=}{} +[r]_q q^{n (r-1)} \xi^{m+n+r+s}
\lambda_1^{-n-1}\lambda_2^{-s}x^{m+r-1}y^{n+s+1}\# \beta (G),
\end{gather*}
where $K ^{\pm 1}\beta(G)$, $E \beta(G)$ and $F\beta(G)$ are multiplications
in $U_q(\mathfrak{sl}_2)$.
\end{Example}

We introduce now the BiHom analogue of comodule Hom-algebras def\/ined in \cite{yau2}.

\begin{Definition}
Let $(H, \mu _H, \Delta _H, \alpha _H, \beta _H, \psi _H, \omega _H)$ be a
BiHom-bialgebra. A \textit{right $H$-comodule BiHom-algebra} is a 7-tuple $(D, \mu _D, \alpha _D, \beta _D, \psi _D, \omega _D, \rho _D)$, where $(D,
\mu _D, \alpha _D, \beta _D)$ is a BiHom-associative algebra, $(D, \psi _D,
\omega _D)$ is a right $H$-comodule via the coaction $\rho _D\colon D\rightarrow
D\otimes H$ and moreover $\rho _D$ is a~morphism of BiHom-associative
algebras.
\end{Definition}

\begin{Example}
If $(H, \mu _H, \Delta _H, \alpha _H, \beta _H, \psi _H, \omega _H)$ is a
BiHom-bialgebra, then we have the right $H$-comodule BiHom-algebra $(H, \mu
_H, \alpha _H, \beta _H, \psi _H, \omega _H, \Delta _H)$.
\end{Example}

The next result generalizes Proposition~3.6 in \cite{mp2}.

\begin{Proposition}
\label{smashcomalg} Let $(H, \mu _H, \Delta _H, \alpha _H, \beta _H, \psi
_H, \omega _H)$ be a BiHom-bialgebra and $(A, \mu _A, \alpha _A, \beta _A)$
a left $H$-module BiHom-algebra, with notation $H\otimes A\rightarrow A$, $h\otimes a\mapsto h\cdot a$, such that all structure maps $\alpha _H$, $\beta
_H$, $\psi _H$, $\omega _H$, $\alpha _A$, $\beta _A$ are bijective. Assume that
there exist two more linear maps $\psi _A, \omega _A\colon A\rightarrow A$ such
that any two of the maps $\alpha _A$, $\beta _A$, $\psi _A$, $\omega _A$ commute
and moreover
\begin{gather}
 \omega _A(aa^{\prime })=\omega _A(a)\omega _A(a^{\prime }), \qquad \forall
\, a, a^{\prime }\in A, \nonumber\\
 \omega _A(h\cdot a)=\omega _H(h)\cdot \omega _A(a), \qquad \forall \, a\in
A,\quad h\in H.  \label{inplus}
\end{gather}
Define the linear map
\begin{gather*}
\rho _{A\# H}\colon \ A\# H\rightarrow (A\# H)\otimes H, \qquad \rho _{A\# H}(a\#
h)=(\omega _A(a)\# h_1)\otimes h_2.
\end{gather*}
Then $(A\# H, \mu _{A\# H}, \alpha _A\otimes \alpha _H, \beta _A\otimes
\beta _H, \psi _A\otimes \psi _H, \omega _A\otimes \omega _H, \rho _{A\# H})$
is a right $H$-comodule BiHom-algebra.
\end{Proposition}

\begin{proof}
We only prove that $\rho _{A\# H}$ is multiplicative and leave the other
details to the reader:
\begin{gather*}
\rho _{A\# H}((a\# h)(a^{\prime }\# h^{\prime })) = \omega _A\big(a\big(\beta
_H^{-1}\omega _H^{-1}(h_1) \cdot \beta _A^{-1}(a^{\prime })\big)\big)\# \big(\psi
_H^{-1}(h_2)h^{\prime }\big)_1\otimes \big(\psi _H^{-1}(h_2)h^{\prime }\big)_2 \\
\hphantom{\rho _{A\# H}((a\# h)(a^{\prime }\# h^{\prime }))}{}
 = \omega _A(a) \omega _A\big(\beta _H^{-1}\omega _H^{-1}(h_1) \cdot \beta
_A^{-1}(a^{\prime })\big)\# \psi _H^{-1}((h_2)_1)h^{\prime }_1\otimes \psi
_H^{-1}((h_2)_2)h^{\prime }_2 \\
\hphantom{\rho _{A\# H}((a\# h)(a^{\prime }\# h^{\prime }))}{}
\overset{\eqref{inplus}}{=} \omega _A(a)\big(\beta _H^{-1}(h_1)\cdot \omega _A
\beta _A^{-1}(a^{\prime })\big)\# \psi _H^{-1}((h_2)_1)h^{\prime }_1\otimes \psi
_H^{-1}((h_2)_2)h^{\prime }_2 \\
\hphantom{\rho _{A\# H}((a\# h)(a^{\prime }\# h^{\prime }))}{}
\overset{\eqref{ghombia1}}{=} \omega _A(a)\big(\beta _H^{-1}\omega
_H^{-1}((h_1)_1)\cdot \omega _A \beta _A^{-1}(a^{\prime })\big)\# \psi
_H^{-1}((h_1)_2)h^{\prime }_1\otimes h_2h^{\prime }_2 \\
\hphantom{\rho _{A\# H}((a\# h)(a^{\prime }\# h^{\prime }))}{}
=\omega _A(a)\big(\beta _H^{-1}\omega _H^{-1}((h_1)_1)\cdot \beta
_A^{-1}\omega _A(a^{\prime })\big)\# \psi _H^{-1}((h_1)_2)h^{\prime }_1\otimes
h_2h^{\prime }_2 \\
\hphantom{\rho _{A\# H}((a\# h)(a^{\prime }\# h^{\prime }))}{}
=(\omega _A(a)\# h_1)(\omega _A(a^{\prime })\# h^{\prime }_1)\otimes
h_2h^{\prime }_2
=\rho _{A\# H}(a\# h)\rho _{A\# H}(a^{\prime }\# h^{\prime }),
\end{gather*}
f\/inishing the proof.
\end{proof}

\begin{Example}
Let $(H, \mu _H, \Delta _H, \alpha _H, \beta _H, \psi _H, \omega _H)$ be a
BiHom-bialgebra such that all structure maps are bijective. Denote by $A$
the linear space $H^*$. Then $A$ becomes a BiHom-associative algebra with
multiplication and structure maps def\/ined by
\begin{gather*}
\begin{split}
 & (f\bullet g)(h)=f\big(\alpha _H^{-1}\omega _H^{-1}(h_1)\big)g\big(\beta _H^{-1}\psi
_H^{-1}(h_2)\big), \\
& \alpha _A\colon \ H^*\rightarrow H^*, \qquad \alpha _A(f)(h)=f\big(\alpha _H^{-1}(h)\big),
\\
& \beta _A\colon \ H^*\rightarrow H^*, \qquad  \beta _A(f)(h)=f\big(\beta _H^{-1}(h)\big),
\end{split}
\end{gather*}
for all $f, g\in H^*$ and $h\in H$. Moreover, A becomes a left $H$-module
BiHom-algebra, with action
\begin{gather*}
 \rightharpoonup \colon \ H\otimes H^*\rightarrow H^*,\qquad (h\rightharpoonup
f)(h^{\prime })=f(\alpha _H^{-1} \beta _H^{-1}(h^{\prime })h),
\end{gather*}
for all $h, h^{\prime }\in H$ and $f\in H^*$. Obviously, $\alpha _A$ and $\beta _A$ are bijective maps. Def\/ine the linear map
\begin{gather*}
 \omega _A\colon \ H^*\rightarrow H^*, \qquad \omega _A(f)(h)=f\big(\omega _H^{-1}(h)\big),
\qquad \forall \, f\in H^*, \quad h\in H,
\end{gather*}
and choose a linear map $\psi _A\colon H^*\rightarrow H^*$ that commutes with $\alpha _A$, $\beta _A$, $\omega _A$, for instance one can choose the map $\psi
_A $ def\/ined by $\psi _A(f)(h)=f(\psi _H^{-1}(h))$, for all $f\in H^*$ and $h\in H$. Then one can check that the hypotheses of Proposition~\ref{smashcomalg} are satisf\/ied, and consequently~$H^*\# H$ becomes a~right $H$-comodule BiHom-algebra.

Note also that, if $H$ is counital with counit $\varepsilon _H$ such that $\varepsilon _H\circ \alpha _H=\varepsilon _H$ and $\varepsilon _H\circ \beta
_H=\varepsilon _H$, then the BiHom-associative algebra $A=H^*$ is unital
with unit $\varepsilon _H$.
\end{Example}

\subsection*{Acknowledgements}

This paper was written while Claudia Menini was a member of GNSAGA. Florin
Panaite was supported by a grant of the Romanian National Authority for
Scientif\/ic Research, CNCS-UEFISCDI, project number PN-II-ID-PCE-2011-3-0635,
contract nr.~253/5.10.2011. Parts of this paper have been written while
Florin Panaite was a~visiting professor at University of Ferrara in
September~2014, supported by INdAM, and a visiting scholar at the Erwin
Schrodinger Institute in Vienna in July~2014 in the framework of the
``Combinatorics, Geometry and Physics'' programme; he would like to thank
both these institutions for their warm hospitality.

The authors are grateful  to the referees for their remarks and questions.

\pdfbookmark[1]{References}{ref}
\LastPageEnding

\end{document}